\documentclass[twoside]{amsart}
\usepackage{mathrsfs}
\usepackage{amsfonts}
\usepackage{amssymb}
\usepackage{amssymb}
\usepackage{amssymb}
\usepackage{amssymb}
\usepackage{amssymb}
\usepackage{amssymb}
\usepackage{amssymb}
\usepackage{amssymb}
\usepackage{amsmath}
\usepackage{amsthm}
\usepackage[all]{xy}
\usepackage{chngcntr}

\counterwithin{table}{section}

\usepackage{lastpage}

\RequirePackage{amsmath} \RequirePackage{amssymb}
\usepackage{amscd,latexsym,amsthm,amsfonts,amssymb,amsmath,amsxtra}
\usepackage
[
colorlinks=true, 
linkcolor=blue,
urlcolor=blue,
bookmarks=true,
bookmarksopen=true,
citecolor=blue,
]
{hyperref}

\newcommand{\bpm}{\begin{pmatrix}}
\newcommand{\epm}{\end{pmatrix}}
\newcommand{\bsm}{\begin{smallmatrix}}
\newcommand{\esm}{\end{smallmatrix}}
\newcommand{\bspm}{\left(\begin{smallmatrix}}
\newcommand{\espm}{\end{smallmatrix}\right)}
\newcommand{\bbm}{\begin{bmatrix}}
\newcommand{\ebm}{\end{bmatrix}}

    \newcommand{\BC}{{\mathbb {C}}} 
     \newcommand{\BF}{{\mathbb {F}}}

    \newcommand{\BQ}{{\mathbb {Q}}}

     \newcommand{\BZ}{{\mathbb {Z}}}

     \newcommand{\sH}{{\mathscr {H}}}

     \newcommand{\CB}{{\mathcal {B}}}
    \newcommand{\CC}{{\mathcal {C}}} 
     
     \newcommand{\CH}{{\mathcal {H}}}
     \newcommand{\CJ}{{\mathcal {J}}}

    \newcommand{\CS}{{\mathcal {S}}} 
     \newcommand{\CV}{{\mathcal {V}}}
    \newcommand{\CW}{{\mathcal {W}}}

       \newcommand{\fH}{{\mathfrak{H}}}

 \newcommand{\Mat}{{\mathrm {Mat}}}
      \newcommand{\RB}{{\mathrm {B}}}

    \newcommand{\RG}{{\mathrm {G}}} 
    \newcommand{\RI}{{\mathrm {I}}}

    \newcommand{\RO}{{\mathrm {O}}}

    \newcommand{\RU}{{\mathrm {U}}} 
    \newcommand{\RW}{{\mathrm {W}}} 
     
    \newcommand{\pr}{{\mathrm {pr}}}

     \newcommand{\Nm}{{\mathrm {Nm}}}
    \newcommand{\lenth}{{\mathrm {\lenth}}}

     \newcommand{\Ch}{{\mathrm{Ch}}}

     \newcommand{\GL}{{\mathrm{GL}}}
    \newcommand{\Hom}{{\mathrm{Hom}}} 
    \newcommand{\Ind}{{\mathrm{Ind}}} \newcommand{\ind}{{\mathrm{ind}}}

\renewcommand{\mod}{\ \mathrm{mod}\ }

\newcommand{\supp}{\mathrm{supp}}

 \newcommand{\SL}{{\mathrm{SL}}}
 \newcommand{\SO}{{\mathrm{SO}}}

 \newcommand{\Sp}{{\mathrm{Sp}}} \newcommand{\diag}{{\mathrm{diag}}}

    \newcommand{\bx}{{\bf {x}}}        %\newcommand{\Ch}{{\rm {Ch}}}

    \newcommand{\wt}{\widetilde} \newcommand{\wh}{\widehat} 
    
    \newcommand{\pair}[1]{\langle {#1} \rangle}
    \newcommand{\wpair}[1]{\left\{{#1}\right\}}

    \newcommand{\ov}{\overline}
    
    \newcommand{\incl}{\hookrightarrow}
    
     \newcommand{\ra}{\rightarrow}

    \theoremstyle{plain}
    
       \newtheorem*{theorem*}{Theorem}

    \newtheorem{thm}{Theorem}[section] \newtheorem{cor}[thm]{Corollary}
    \newtheorem{lem}[thm]{Lemma}  \newtheorem{prop}[thm]{Proposition}
    \newtheorem {conj}[thm]{Conjecture} 
\newtheorem{rmk}[thm]{Remark}
 \newtheorem*{Jconj}{Conjecture~$\bJ(\textbf{\textit{n}},\br)$}

\def\bJ{{\boldsymbol{\CJ}}}
\def\bN{{\boldsymbol{N}}}
\def\br{{\boldsymbol{r}}}

    \numberwithin{equation}{section}

\usepackage[top=1in, bottom=1in, left=1.25in, right=1.25in]{geometry}

\title[A converse theorem for $\RG_2$]{On a converse theorem for $\RG_2$ over finite fields}
\author{Baiying Liu}
\address{Department of Mathematics, Purdue University,
West Lafayette, IN 47906, USA}
\email{liu2053@purdue.edu}

\author{Qing Zhang}
\address{Department of Mathematics and Statistics,
University of Calgary, Calgary, AB, Canada, T2N 1N4}
\email{qing.zhang1@ucalgary.ca}


\begin{document}

\begin{abstract}
    In this paper, we prove certain multiplicity one theorems and define twisted gamma factors for irreducible generic cuspidal representations of split $\mathrm{G}_2$ over finite fields $k$ of odd characteristic. Then we prove the first converse theorem for exceptional groups, namely, $\GL_1$ and $\GL_2$-twisted gamma factors will uniquely determine an irreducible generic cuspidal representation of $\RG_2(k)$. 
\end{abstract}

\thanks{The first-named author is partially supported by NSF Grants DMS-1702218, and start-up funds from the Department of Mathematics at Purdue University. The second-named author is partially supported by NSFC grant 11801577 and a postdoc fellowship from Pacific Institute for the Mathematical Sciences (PIMS)}

\maketitle

%\tableofcontents

\section{Introduction}

In the theory of automorphic representations, the global converse problems aim to recover automorphic forms from their Fourier coefficients. Global converse theorems have played crucial roles in establishing global Langlands functoriality (\cite{CKPSS01, CKPSS04, CPSS11}) and are very important for the Langlands Program. 
In the theory of representations of reductive groups over local and finite fields, the converse problems aim to find a minimal complete set of invariants of twisted gamma factors uniquely determining irreducible generic representations. Local converse theorems have been used to 
prove the uniqueness of the generic local Langlands functoriality and the local Langlands correspondence (\cite{JS03, H93}).
While converse problems have been extensively studied for general linear and classical groups, they have not been studied for exceptional groups. The goal of this paper is to initiate the study of converse problems for exceptional groups and prove the first converse theorem for the split exceptional group of type $\RG_2$ over finite fields of odd characteristic. 
In the following, we first introduce the recent progress on the study of the converse problems for general linear and classical groups over local and finite fields. 

Let $F$ be a non-archimedean local field. Let $\pi$ be an irreducible generic representation of $\GL_n(F)$. The family of local twisted gamma factors $\gamma(s, \pi \times \tau, \psi)$, for $\tau$ any irreducible generic representation of $\GL_r(F)$, $\psi$ an additive character of $F$ and $s\in\BC$, can be defined using Rankin--Selberg convolution \cite{JPSS83} or the Langlands--Shahidi method \cite{S84}. The local converse problem is that which family of local twisted gamma factors will uniquely determine $\pi$? The following is the famous Jacquet's conjecture on the local converse problem, which is an unevolved conjecture for nearly 40 years.

\begin{conj}[Jacquet's conjecture on the local converse problem]\label{lcp}
Let $\pi_1,\pi_2$ be irreducible generic representations
of $\GL_n(F)$. Suppose that they have the same central character.
If
\[
\gamma(s, \pi_1 \times \tau, \psi) = \gamma(s, \pi_2 \times \tau, \psi),
\]
as functions of the complex variable $s$, for all irreducible
generic representations $\tau$ of $\GL_r(F)$ with $1 \leq r \leq 
[\frac{n}{2}]$, then $\pi_1 \cong \pi_2$.
\end{conj}

Conjecture \ref{lcp} has recently been proved by Chai (\cite{Ch16}), and by Jacquet and the first-named author (\cite{JL18}), independently, using different analytic methods. 

One can propose a more general family of conjectures as follows (see \cite{ALSX16}). Let $\pi_1,\pi_2$ be irreducible generic representations of $\GL_n(F)$.
We say that $\pi_1$ and $\pi_2$ satisfy hypothesis $\CH_0$
if they have the same central character. 
For $m \in \BZ_{\geq 1}$, we say that they satisfy hypothesis $\CH_m$ if they satisfy hypothesis $\CH_0$ and satisfy
$$\gamma(s,\pi_1\times\tau,\psi)=\gamma(s,\pi_2\times\tau,\psi)$$
as functions
of the complex variable $s$, for all irreducible generic
representations $\tau$ of $\GL_m(F)$.
For $r \in \BZ_{\geq 0}$, we say that $\pi_1,\pi_2$ satisfy hypothesis $\CH_{\le r}$ if they satisfy hypothesis $\CH_m$,
for $0\le m\le r$.

\begin{Jconj}
If $\pi_1,\pi_2$ are irreducible generic representations
of $\GL_n(F)$ which satisfy hypothesis $\CH_{\le r}$, then $\pi_1\simeq\pi_2$.
\end{Jconj}

Conjecture \ref{lcp} was exactly Conjecture $\CJ(n,[\frac{n}{2}])$. Conjecture $\CJ(2,1)$ was first proved by Jacquet and Langlands \cite{JL70}.
Conjecture $\CJ(3,1)$ was first proved by Jacquet, Piatetski-Shapiro, and Shalika \cite{JPSS79}. 
For general $n$, Conjecture $\CJ(n,n-1)$ was proved by Henniart in \cite{H93}, and by Cogdell-Piatetski-Shapiro using a global method in \cite{CPS94}.
Conjecture $\CJ(n,n-2)$ (for $n \geq 3$) is a theorem due
to Chen \cite{Ch96, Ch06}, to Cogdell and Piatetski-Shapiro \cite{CPS99}, and to Hakim and Offen \cite{HO15}. 

In \cite{JNS15}, Jiang, Nien and Stevens showed that Conjecture \ref{lcp} is equivalent to the same conjecture with the adjective ``generic" replaced by ``unitarizable supercuspidal" as follows: 

\begin{conj}\label{lcp2}
Let $\pi_1,\pi_2$ be irreducible unitarizable supercuspidal representations of $\GL_n(F)$. Suppose that they have the same central character. If
\[
\gamma(s, \pi_1 \times \tau, \psi) = \gamma(s, \pi_2 \times \tau, \psi),
\]
as functions of the complex variable $s$, for all irreducible
supercuspidal representations $\tau$ of $\GL_r(F)$ with $1 \leq r \leq 
[\frac{n}{2}]$, then $\pi_1 \cong \pi_2$.
\end{conj}

Making use of the construction of supercuspidal representations of $\GL_n(F)$ in \cite{BK93} and properties of Whittaker functions of supercuspidal representations constructed in \cite{PS08}, Jiang, Nien and Stevens introduced
the notion of a special pair of Whittaker functions for a pair of irreducible unitarizable supercuspidal representations $\pi_1$, $\pi_2$ of $\GL_n(F)$. They proved that if there is such a pair, and $\pi_1$, $\pi_2$ satisfy hypothesis $\CH_{\leq [\frac{n}{2}]}$, then $\pi_1 \cong \pi_2$. They also found special pairs of Whittaker functions in many cases, in particular the case of depth zero representations. In \cite{ALSX16}, Adrian, the first-named author, Stevens and Xu proved part of the case left open in \cite{JNS15}.
In particular, the results in \cite{JNS15} and \cite{ALSX16} together imply that Conjecture \ref{lcp2} is true for $\GL_n$, $n$ prime. 

It is easy to find pairs of generic representations showing that in Conjecture 1.1, $[\frac{n}{2}]$ is sharp for the generic dual of $\GL_n(F)$. In \cite{ALST16}, Adrian, the first-named author, Stevens and Tam showed that, in Conjecture \ref{lcp2}, $[\frac{n}{2}]$ is sharp for the supercuspidal dual of $\GL_n(F)$, for $n$ prime, in the tame case. It is believed that in Conjecture \ref{lcp2}, $[\frac{n}{2}]$ is sharp for the supercuspidal dual of $\GL_n(F)$, for any $n$, in all cases.
However, it is expected that for certain families of supercuspidal representations, $[\frac{n}{2}]$ may not be sharp, for example, for simple supercuspidal representations (of depth $\frac{1}{n}$), the upper bound may be lowered to 1 (see \cite[Proposition 2.2]{BH14} and \cite[Remark 3.18]{AL16} in general, and \cite{X13} in the tame case).

For general reducitve groups, one can consider analogue converse problems whenever the twisted gamma factors have been defined, using either the Rankin-Selberg convolution method or the Langlands-Shahidi method. 

Nien in \cite{N14} proved the finite fields analogue of Conjecture \ref{lcp} for cuspidal representations of $\GL_n$, using special properties of normalized Bessel functions and the twisted gamma factors defined by Roditty (\cite{Ro}). Adrain and Takeda in \cite{AT18} proved a local converse theorem for $\GL_n$ over archimedean local fields using $L$-functions.  
In \cite{M16}, Moss defined the twisted gamma factors for $\ell$-adic families of smooth representations of $\GL_n(F)$, where $F$ is a finite extension of $\BQ_p$ and $\ell$ is different from $p$, and proved an analogue of Conjecture $\CJ(n,n-1)$. In \cite{LM18}, joint with Moss, the first-named author proved an analogue of Conjecture $\CJ(n,[\frac{n}{2}])$ for $\ell$-adic families, using the idea in \cite{JL18}. 

For $p$-adic groups other than $\GL_n$, in particular classical groups, twisted gamma factors have been defined in many cases and the local converse problems have been vastly studied: $\mathrm{U}(2,1)$ and $\mathrm{GSp}(4)$ (Baruch,  \cite{B95} and \cite{B97}); $\SO(2n+1)$ (Jiang and Soudry,  \cite{JS03}); $\RU_{2n}$ (Morimoto  \cite{Mor}, and the second named author \cite{Zh18a}); $\rm{U}(1,1)$, $\RU(2,2)$,
$\Sp(2n)$ and $\RU_{2n+1}$(the second named author, \cite{Zh17a}, \cite{Zh17b}, \cite{Zh18a}, and \cite{Zh18b}). 

The ideas of converse theorems have been extended to distinction problems, namely, using special values of twisted gamma factors to characterize representations of $\GL_n(E)$ distinguished by $\GL_n(F)$, where $E/F$ is a quadratic extension, see Hakim and Offen \cite{HO15} (over $p$-adic fields), and Nien \cite{N18} (over finite fields).

In this paper, we initiate the project of studying converse problems for generic representations of exceptional groups.
Let $k$ be a finite field of odd characteristic $p$ and let $\psi$ be a fixed non-trivial additive character of $k$. 
We define $\GL_1$ and $\GL_2$-twisted gamma factors for irreducible generic cuspidal representations of $\mathrm{G}_2(k)$, and prove the first converse theorem for exceptional groups as follows:

\begin{thm}[Theorem \ref{thm: converse theorem}]\label{g2conversetheorem}
Let $\Pi_1,\Pi_2$ be two irreducible generic cuspidal representation of $\RG_2(k)$. If $$\gamma(\Pi_1\times\chi,\psi)=\gamma(\Pi_2\times \chi,\psi),$$
$$\gamma(\Pi_1\times \tau,\psi)=\gamma(\Pi_2\times \tau,\psi),$$
for all characters $\chi$ of $k^\times$ and all irreducible generic representations $\tau$ of $\GL_2(k)$, we have $\Pi_1\cong \Pi_2$.
\end{thm}

To define $\GL_1$-twisted gamma factors, we use the finite fields analogue of Ginzburg's local zeta integral in \cite{Gi} and we prove the following multiplicity one theorem to deduce the functional equation:

\begin{thm}[Theorem \ref{thm: multiplicity one}]\label{multiplicityonethm}
Let $\Pi$ be an irreducible cuspidal representation of $\RG_2(k)$, then 
$$\dim\Hom_J(\Pi,I(\chi)\otimes \omega_\psi)\le 1,$$
where $J$ is the Jacobi group contained in the maximal parabolic subgroup of $\RG_2(k)$ with the long root in the Levi $($see $\S$$\ref{subsec: G2}$ for definitions$)$, $\chi$ is a character of $k^{\times}$, 
$I(\chi)$ is the induced representation of $\SL_2(k)$ from $\chi$, and $\omega_{\psi}$ is the Weil representation of $J$ with central character $\psi$. 
\end{thm}

To prove the above theorem, we need to use the classification of representations of $\RG_2(k)$ given by Enomoto (\cite{En}, for $p=3$), and by Chang and Ree (\cite{CR}, for $p >3$), and compute the dimension of the Hom space for each irreducible cuspidal representation. 

To define $\GL_2$-twisted gamma factors, we embed $\RG_2$ into $\SO_7$ and use the finite fields analogue of the local zeta integral developed by Piatetski-Shapiro, Rallis and Schiffmann (\cite{PSRS}). The functional equation in this case follows from the following multiplicity one result

\begin{prop}[Proposition \ref{prop: uniqueness of a bilinear form}]\label{introprop: uniqueness of a bilinear form}
Let $\Pi$ be an irreducible generic cuspidal representation of $\RG_2(k)$ and let $\tau$ be an irreducible generic representation of $\GL_2(k)$. Then we have 
$$\dim \Hom_{\RG_2(k)}(I(\tau)|_{\RG_2(k)},\Pi)=1,$$
where $I(\tau)=\Ind_{\wt P}^{\SO_7(k)}(\tau \otimes 1_{\SO_3})$, $\wt{P}$ is a parabolic subgroup of $\SO_7(k)$ with the Levi subgroup isomorphic to $\GL_2(k) \times \SO_3(k)$ $($see $\S$$\ref{sec: G2 GL2}$ for definitions$)$. 
\end{prop}

The proof of Theorem \ref{g2conversetheorem}  briefly goes as follows. Let $\CB_i:=\CB_{\Pi_i}\in \CW(\Pi_i,\psi)$ be the Bessel function of $\Pi_i$ for $i=1,2$, namely, the Whittaker function associated with a Whittaker vector, normalized by $\CB_i(1)=1$ (see $\S$\ref{subsec:Bessel} and Lemma \ref{lem: basic properties of Bessel function} for the basic properties of $\CB_i$). We will prove that $\CB_1(g)=\CB_2(g)$ for all $g\in \RG_2(k)$ under the assumption of Theorem \ref{g2conversetheorem}. Since $\RG_2(k)=\coprod_{w\in \RW(\RG_2)}BwB$, 
where $B$ is a fixed Borel subgroup of $\RG_2$, 
it suffices to show that $\CB_1$ agrees with $\CB_2$ on various cells $BwB$. 
Let $\RB(\RG_2)=\wpair{w\in \RW(\RG_2): \forall \gamma\in \Delta, w\gamma>0 \implies w\gamma\in \Delta}$,
where $\Delta=\{\alpha, \beta\}$ is the set of simple roots of $\RG_2$ with $\alpha$ being the short root and $\beta$ being the long root. Let $s_\alpha$ (resp. $s_\beta$) be the simple reflection defined by $ \alpha$ (resp. $\beta$). Then $\RB(\RG_2) = \{1, w_1, w_2, w_{\ell}\}$, where $w_\ell$ is the long root, $w_1=w_\ell s_\alpha$ and $w_2=w_\ell s_\beta$. By Lemma \ref{lem: weyl element which support bessel function}, if $w\notin \RB(\RG_2)$, we have $\CB_1(g)=\CB_2(g)=0$ for $g\in BwB$. If $w=1$, we also have $\CB_1(g)=\CB_2(g),\forall g\in B$ by Lemma \ref{lem: basic properties of Bessel function}. Thus it suffices to show that $\CB_1(g)=\CB_2(g),\forall g\in BwB$ with $w=w_1,w_2,w_\ell$. It turns out that the equality of $\GL_1$-twisted gamma factors implies that $\CB_1(g)=\CB_2(g),\forall g\in Bw_1B$, and the equality of $\GL_2$-twisted gamma factors implies that $\CB_1(g)=\CB_2(g),\forall g\in BwB$ with $w=w_2,w_\ell$. 
This completes the proof of Theorem \ref{g2conversetheorem}. 

% Note that on the dual groups side we have an embedding $\RG_2(\BC) \hookrightarrow \GL_7(\BC)$, by Langlands philosophy of functoriality, irreducible generic cuspidal representations of $\RG_2(\mathbb{F}_q)$ are expected to lift to irreducible generic representations of $\GL_7(\mathbb{F}_q)$ and the lifting is expected to preserve all $\GL$-twisted gamma factors. Then the converse theorem proved by Nien in \cite{N14} (which is also true for irreducible generic representations of $\GL_n(k)$ by using a similar argument as in \cite{JL18}) implies that 
% irreducible generic cuspidal representations of $\RG_2(k)$ are uniquely determined by $\GL_1$, $\GL_2$, and $\GL_3$-twisted gamma factors (once all defined). However, Theorem \ref{g2conversetheorem} shows that only $\GL_1$, $\GL_2$-twisted gamma factors will be enough to determine irreducible generic cuspidal representations of $\RG_2(k)$. 

In \cite{NZ18}, Nien and Zhang verifies the $\CJ(n,1)$-Converse Theorem for twisted gamma factors of irreducible cuspidal representations of $\GL_n(\mathbb{F}_q)$, for $n \leq 5$, and of irreducible generic representations of $\GL_n(\mathbb{F}_q)$, for $n<\frac{q-1}{2\sqrt{q}}+1$ in the appendix by Zhiwei Yun. Note that on the dual groups side we have an embedding $\RG_2(\BC) \hookrightarrow \GL_7(\BC)$. Hence, by Langlands philosophy of functoriality, it is expected that irreducible generic cuspidal representations of $\RG_2(\mathbb{F}_q)$ would be uniquely determined by $\GL_1$-twisted gamma factors when $7<\frac{q-1}{2\sqrt{q}}+1$. In future work, the authors plan to check this expectation directly by analyzing $\GL_1$-twisted gamma factors for irreducible generic cuspidal representations of $\RG_2(\mathbb{F}_q)$. 

% If the characteristic of the finite field $k$ is $2$, the character table of $\RG_2(k)$ is given in \cite{EY86} and it's still reasonable to consider the converse problem for $\RG_2(k)$. But our approach does not work in this case. For example, our definition of gamma factors for $\RG_2(k)\times \GL_1(k)$ relies on the Weil representation of $\SL_2(k)$, which is only defined when the characteristic of $k$ is odd in \cite{Ge}. 

Theorem \ref{g2conversetheorem} inspires us to consider the local converse problem for $\RG_2(F)$ when $F$ is a $p$-adic field. In this case, our proof in $\S$\ref{sec: G2 GL2} is actually valid for an analogue of Proposition \ref{introprop: uniqueness of a bilinear form} without the restriction that $\Pi$ is cuspidal, which gives us the local functional equation of the local zeta integral of Piatetski-Shapiro-Rallis-Schiffmann (\cite{PSRS}) and hence the existence of the $\GL_2$-twisted local gamma factors. However, the existence of the $\GL_1$-twisted local gamma factors relies on the following 
\begin{conj}\label{conj: multiplicityone}
Let $F$ be a $p$-adic field and $\Pi$ be an irreducible generic representation of $\RG_2(F)$. Let $\psi$ be a nontrivial additive character of $F$. Let $\wt I(\chi,\psi)$ be the genuine induced representation on the double cover $\wt\SL_2(F)$ for a character $\chi$ of $F^\times$. Then if $\wt I(\chi,\psi)$ is irreducible, we have 
$$\dim_J(\Pi,\wt I(\chi,\psi)\otimes \omega_\psi)\le 1.$$
\end{conj}
Note that both $\wt I(\chi,\psi)$ and $\omega_\psi$ are genuine representations on a double cover of $J$ and the thus the tensor product $\wt I(\chi,\psi)$ is a representation on $J$.

In the above conjecture, we keep the requirement minimal so that it is enough to deduce the local functional equation of Ginzburg's local zeta integral (\cite{Gi}). We do expect that the following generalized conjecture is true

\begin{conj}\label{conj: multiplicityone2}
Let $F$ be a $p$-adic field and $\Pi$ be an irreducible (selfdual) representation of $\RG_2(F)$. Let $\psi$ be a nontrivial additive character of $F$. Let $\wt \pi$ be an irreducible genuine representation on the double cover $\wt\SL_2(F)$. Then we have 
$$\dim_J(\Pi,\wt \pi\otimes \omega_\psi)\le 1.$$
\end{conj}

As explained in \cite[\S 6]{LZ}, Conjecture \ref{conj: multiplicityone2} is an analogue of the uniqueness problem of Fourier-Jacobi models for $\Sp_{2n}$, which was proved in \cite{BR} (for $n=2$) and in \cite{GGP, Su}(for general $n$). Once Conjecture \ref{conj: multiplicityone} is established, we then have the local gamma factors for irreducible generic representations of $\RG_2(F)\times \GL_1(F)$ using Ginzburg's local zeta integral. Inspired by Theorem \ref{g2conversetheorem}, we propose the following conjecture on the local converse problem for $\RG_2(F)$. 

\begin{conj}\label{conj: local converse}
Let $F$ be a $p$-adic field. Suppose that Conjecture $\ref{conj: multiplicityone}$ is true. Let $\Pi_1,\Pi_2$ be two irreducible generic representations of $\RG_2(F)$. If 
\begin{align*}
    \gamma(s,\Pi_1\times \chi,\psi)&=\gamma(s,\Pi_2\times \chi,\psi),\\
     \gamma(s,\Pi_1\times \tau,\psi)&=\gamma(s,\Pi_2\times \tau,\psi),
\end{align*}
for all characters $\chi$ of $\GL_1(F)$ and all irreducible generic representations $\tau$ of $\GL_2(F)$, then $\Pi_1\cong \Pi_2$.
\end{conj}

Conjectures \ref{conj: multiplicityone}--\ref{conj: local converse} are current work in progress of the authors. 

Again, by Langlands philosophy of functoriality, 
% According the conjectural local Langlands functoriality, 
representations of $\RG_2(F)$ are expected to be lifted to representations of $\GL_7(F)$ and this lifting is expected to preserve $\GL$-twisted local gamma factors. Then the Jacquet's local converse conjecture for $\GL_n$, which was recently proved in \cite{Ch16, JL18}, implies that two irreducible generic  representations $\Pi_i,i=1,2$ of $\RG_2(F)$ would be isomorphic if the twisted local gamma factors $\gamma(s,\Pi_1\times \tau,\psi)$, $\gamma(s,\Pi_2\times \tau,\psi)$, are the same for all irreducible generic representations $\tau$ of $\GL_n(F)$ for all $n=1,2,3$ (once they are all defined). Theorem \ref{g2conversetheorem} says that we only need $\GL_1$ and $\GL_2$-twisted gamma factors over finite fields, and we expect the same is true over $p$-adic fields in Conjecture \ref{conj: local converse}.

The paper is organized as follows. In \S\ref{sec: Jacobi group}, we introduce the group $\RG_2$, the Fourier-Jacobi group $J$, the Weil representations $\omega_{\psi}$, and the Multiplicity One Theorem \ref{multiplicityonethm}. Theorem \ref{multiplicityonethm} is proved in \S\ref{sec: proof when p>3} (for $p > 3$) and \S\ref{sec: proof when p=3} (for $p=3$). We define $\GL_1$ and $\GL_2$-twisted gamma factors for irreducible generic cuspidal representations in \S\ref{sec: G2 GL1} and \S\ref{sec: G2 GL2}, respectively. Finally, Theorem \ref{g2conversetheorem} is proved in \S\ref{sec: converse theorem}. In Appendix \ref{sec: computation of gauss sum}, we compute certain Gauss sums which are used in the proof of Theorem \ref{multiplicityonethm}. In Appendix \ref{sec: embedding G2 into SO7}, we describe the embedding of $\RG_2$ into $\SO_7$ used in this paper.

\subsection*{Acknowledgements} 
The authors would like to thank James Cogdell, Clifton Cunningham, Dihua Jiang and Freydoon Shahidi for their interest, constant support and encouragement. This project was initiated when the second-named author was a student at The Ohio State University. The collaboration of the two authors started from the The 2016 Paul J. Sally, Jr. Midwest Representation Theory Conference in University of Iowa. Part of the work was done when the second-named author worked at Sun Yat-Sen University, Guangzhou, China. We would like to express our gratitude to the above mentioned institutes.

\section{The Fourier-Jacobi group and a multiplicity one theorem}\label{sec: Jacobi group}

\subsection{Some notations and conventions} 

Throughout this paper, unless specified otherwise, we fix the following notations. Let $p$ be an odd prime and $q$ is a power of $p$. Let $k=\BF_q$, the finite field with $q$ elements. Let $\epsilon(x)=\left(\frac{x}{q} \right)$, where $\left( \frac{\cdot}{\cdot}\right)$ denotes the Legendre symbol. Let $\epsilon_0=\epsilon(-1)$. Then we have $\epsilon_0=1$ if $q\equiv 1\mod 4$, and $\epsilon_0=-1$ if $q\equiv 3 \mod 4$. Let $k^{\times,2}=\wpair{x^2: x\in k^\times},$ and $ k^{\times,3}=\wpair{x^3:x\in k^\times}$. Let $k_2$ be the unique quadratic extension of $k$, i.e., $k_2=\BF_{q^2}$. We fix a generator $\kappa$ of the multiplicative group $k^\times$. Then we have $\kappa\in k^\times-k^{\times,2}$. Let $\psi$ be a fixed non-trivial additive character of $k$. Then there exists a 4-th root of unity $\epsilon_\psi$ such that for any $a\in k^\times$, we have
\begin{equation}\label{eq: basic guass sum}\sum_{x\in k}\psi(ax^2)=\epsilon_\psi\epsilon(a)\sqrt{q}.\end{equation}
Moreover, we have $\epsilon_\psi^2=\epsilon_0$. See \cite[Ex.4.1.14]{Bu} for example. By abuse of notation, we write $\epsilon_\psi $ as $\sqrt{\epsilon_0}$. 

We usually don't distinguish a representation and its space. Thus for a representation $\pi$ of a group $G$, a vector $v\in \pi$ means that a vector $v$ in the space of $\pi$. 

\subsection{Weil representations of \texorpdfstring{$\SL_2(k)$}{Lg}}

 Let $W=k^2$, endowed with the symplectic structure $\pair{~,~}$ defined by
\begin{equation} \pair{(x_1,y_1),(x_2,y_2)}=-2x_1y_2+2x_2y_1. \end{equation}
Let $\sH$ be the Heisenberg group associated with the symplectic space $W$. Explicitly, $\sH=W\oplus k$ with addition
$$[x_1,y_1,z_1]+[x_2,y_2,z_2]=[x_1+x_2,y_1+y_2,z_1+z_2-x_1y_2+x_2y_1].$$
Let $\SL_2(k)$ act on $\sH$ such that it acts on $W$ from the right and act on the third component $k$ in $\sH$ trivially. Then we can form the semi-direct product $\SL_2(k)\ltimes \sH.$ The product in $\SL_2(k)\ltimes \sH$ is given by 
$$(g_1,v_1)(g_1,v_2)=(g_1g_2, v_1.g_2+ v_2).$$
By \cite{Ge}, there is a Weil representation $\omega_\psi$ of on $\CS(k)$, the space of $\BC$-valued functions on $k$. The Weil representation $\omega_\psi$ is determined by several formula, which can be found in \cite{GH}. Note that the symplectic form in \cite{GH} is a little bit different from ours. Thus the formulas in \cite{GH} should be adapted to our slightly different symplectic structure on $W$. One can consult \cite{Ku} for the dependence on the symplectic structure. The Weil representation in our case is determined by the following formulas:
\begin{align}
\begin{split}\label{eq: Weil representation of SL2}
\omega_\psi([x,0,z])\phi(\xi)&=\psi(z)\phi(\xi+x),\\
\omega_\psi([0,y,0])\phi(\xi)&=\psi(-2\xi y)\phi(\xi),\\
\omega_\psi(\diag(a,a^{-1}))\phi(\xi)&=\epsilon(a)\phi(a\xi),\\
\omega_\psi\left(\begin{pmatrix}1&b\\ &1 \end{pmatrix} \right)\phi(\xi)&=\psi(-b\xi^2)\phi(\xi),\\
\omega_\psi\left(\begin{pmatrix}&b\\ -b^{-1}& \end{pmatrix} \right)\phi(\xi)&=\frac{1}{\gamma(b,\psi)}\sum_{x\in k} \phi(x)\psi(-2xb\xi),
\end{split}
\end{align}
for $\phi\in \CS(k),x,y,z,\xi,b\in k,a\in k^\times.$ Here $\gamma(b,\psi)=\sum_{x\in k}\psi(-bx^2)$, which can be computed using \eqref{eq: basic guass sum}. 

\subsection{The group \texorpdfstring{$\RG_2$}{Lg} and its Fourier-Jacobi subgroup}\label{subsec: G2}
In this subsection, we give a very brief review of some definitions and notations related to the group $\RG_2(k)$. More details can be found in \cite[$\S$5]{LZ}. 

Let $\RG_2$ be the split exceptional algebraic group of type $\RG_2$ over the field $k$. The group $\RG_2$ has two simple roots $\alpha,\beta$, where $\alpha$ is the short root and $\beta$ is the long root, and has 6 positive roots $\alpha,\beta,\alpha+\beta,2\alpha+\beta,3\alpha+\beta,3\alpha+2\beta$. Let $s_\alpha,s_\beta$ be the reflections determined by $\alpha,\beta$ respectively. One has $s_\alpha(\beta)=3\alpha+\beta,s_\beta(\alpha)=\alpha+\beta$. We use the standard notations of Chevalley groups \cite{St}. For a root $\gamma$, let $U_\gamma\subset \RG_2(k)$ be the corresponding root space and let $\bx_\gamma: k\ra \RG_2(k)$ be a fixed isomorphism which satisfies various Chevalley relations, see \cite[Chapter 3]{St}. The explicit commutator relations can be found in \cite[p.192]{Ch}. A matrix realization of $\bx_\gamma(r), r\in k$, is given in Appendix \ref{sec: embedding G2 into SO7}. The calculations in this paper could be performed using this explicit matrix realization. 

For a root $\gamma$, let $w_\gamma(t)=\bx_\gamma(t)\bx_{-\gamma}(-t^{-1})\bx_\gamma(t)$, $w_\gamma=w_\gamma(1)$, and $h_\gamma(t)=w_\gamma(t)w_\gamma^{-1}$. Note that $w_\gamma$ is a representative of the Weyl group element $s_\gamma$. Let $h(t_1,t_2)=h_\alpha(t_1t_2)h_\beta(t_1^2t_2)$. One can check that $h(t_1,t_2)$ agrees with the notation $h(t_1,t_2,t_1^{-1}t_2^{-1})$ in \cite{Ch, CR} and the notation $h(t_1,t_2)$ in \cite{Gi}. Let $T=\wpair{h(t_1,t_2):t_1,t_2\in k^\times}$ be the maximal torus of $\RG_2(k)$ and let $U$ be the subgroup generated by $U_\gamma,\gamma$ positive. Then $B=TU$ is a Borel subgroup of $\RG_2(k)$. It is known that $\RG_2$ has trivial center. 

Let $P'=M'V'$ be the parabolic subgroup with Levi $M'$ and unipotent $V'$ such that $U_\alpha\subset M'$. Then $M'\cong \GL_2(k)$ and $V'$ is the group generated by $U_\beta,U_{\alpha+\beta},U_{2\alpha+\beta},U_{3\alpha+\beta},U_{3\alpha+2\beta}$.

Let $P=MV$ be the parabolic subgroup with Levi $M$ and unipotent $V$ such that $U_\beta\subset M$. Note that $V$ is generated by $U_\alpha,U_{\alpha+\beta},U_{2\alpha+\beta},U_{3\alpha+\beta},U_{3\alpha+2\beta}.$ Let $Z\subset V$ be the subgroup generated by $U_{2\alpha+\beta},U_{3\alpha+\beta},U_{3\alpha+2\beta}$. We still have $M\cong \GL_2(k)$ and the isomorphism can be realized by 
\begin{align*}
    \bx_\beta(x)&\mapsto \bpm 1&x\\ &1 \epm,\\
    h(a,b)&\mapsto \bpm a& \\ &b \epm.
\end{align*}
Let $J\subset P$ be the subgroup $\SL_2(k)\ltimes V$, where $\SL_2(k)$ is viewed as a subgroup of $M$ via $\SL_2(k)\subset \GL_2(k)\cong M$. A typical element in $V$ is of the form 
$$(r_1,r_2,r_3,r_4,r_5):=\bx_\alpha(r_1)\bx_{\alpha+\beta}(r_2)\bx_{2\alpha+\beta}(r_3)\bx_{3\alpha+\beta}(r_4)\bx_{3\alpha+2\beta}(r_5).$$ There is a group homomorphism 
$$\ov{\textrm{pr}}: J=\SL_2(k)\ltimes V\ra \SL_2(k)\ltimes\sH$$
given by 
$$(g,(r_1,r_2,r_3,r_4,r_5))\mapsto (d_1gd_1,(r_1,r_2,r_3-r_1r_2)),$$
where $d_1=\diag(-1,1)$, see \cite[$\S$5]{LZ} for more details. Thus the Weil representation $\omega_\psi$ can be viewed as a representation of $J$ via the above group homomorphism. By \eqref{eq: Weil representation of SL2} and the description of the above group homomorphism, we have the following formulas
\begin{align}
\begin{split}\label{eq: Weil representation of J}
\omega_\psi((r_1,0,r_3,r_4,r_5))\phi(\xi)&=\psi(r_3)\phi(\xi+r_1),\\
\omega_\psi((0,r_2,0,0,0))\phi(\xi)&=\psi(-2\xi r_2)\phi(\xi),\\
\omega_\psi(h(a,a^{-1}))\phi(\xi)&=\epsilon(a)\phi(a\xi),\\
\omega_\psi (\bx_\beta(b))\phi(\xi)&=\psi(b\xi^2)\phi(\xi),\\
\omega_\psi\left(\begin{pmatrix}&b\\ -b^{-1}& \end{pmatrix} \right)\phi(\xi)&=\frac{1}{\gamma(b,\psi)}\sum_{x\in k} \phi(x)\psi(-2xb\xi).
\end{split}
\end{align}

Let $\chi$ be a character of $k^\times$ and we view $\chi$ as a character of the upper triangular subgroup $B_{\SL_2}$ of $\SL_2(k)$ by 
$$\chi\left(\bpm a &b\\ &a^{-1} \epm \right)=\chi(a).$$
Consider the induced representation $I(\chi):=\Ind_{B_{\SL_2}}^{\SL_2(k)}(\chi)$. We view $I(\chi)$ as a representation of $J$ via the natural quotient map $J\ra \SL_2(k)$. The first main result of this paper is the following
\begin{thm}\label{thm: multiplicity one}
Let $\Pi$ be an irreducible cuspidal representation of $\RG_2(k)$, then 
$$\dim\Hom_J(\Pi,I(\chi)\otimes \omega_\psi)\le 1.$$
\end{thm}

\begin{rmk}\rm{
(1). Note that the character table of $\RG_2(k)$ when $p>3$ (see \cite{CR}) and $p=3$ (see \cite{En}) are different. Hence, we will prove Theorem \ref{thm: multiplicity one} for $p>3$ and $p=3$ separately in the next two sections. 
Also note that, in the above theorem, we don't require that $I(\chi)$ is irreducible.

(2). One should compare Theorem \ref{thm: multiplicity one} with \cite[Remark 7.2]{LZ}, where we have shown that the dimension of the Hom space may be bigger than $1$ for general irreducible representations of $\RG_2(k)$ even when $I(\chi)$ is irreducible, however, here we show that if we consider irreducible cuspidal representations of $\RG_2(k)$, then the dimension of the Hom space is indeed less than or equal to $1$.}
\end{rmk}

\begin{cor}\label{cor: multiplicity one}
Let $\Pi$ be an irreducible cuspidal representation of $\RG_2(k)$ and $\pi$ be an irreducible representation of $\SL_2(k)$, then we have
$$\dim \Hom_J(\Pi,\pi\otimes\omega_\psi)\le 1.$$
\end{cor}
\begin{proof}
If $\pi$ is an irreducible representation of $\SL_2(k)$ which is not of the form $I(\chi)$, the assertion follows from the main theorem of \cite{LZ}. If $\pi$ is of the form $I(\chi)$, the assertion follows from Theorem \ref{thm: multiplicity one}.
\end{proof}

\section{Proof of Theorem \texorpdfstring{$\ref{thm: multiplicity one}$}{Lg} when \texorpdfstring{$p>3$}{Lg}}\label{sec: proof when p>3}
In this section, we prove Theorem \ref{thm: multiplicity one} when $p>3$. The character table of $\RG_2(k)$ when $p>3$ is given in \cite{CR}.

\subsection{Character table of \texorpdfstring{$I(\chi)\otimes\omega_\psi$}{Lg}} As a preparation for the proof of Theorem \ref{thm: multiplicity one}, in this subsection, we give the character table of the representation $I(\chi)\otimes \omega_\psi$ of $J$. Given a representation $\pi$ of a finite group, denote by $\Ch_\pi$ the character of $\pi$. It is well-known that 
\begin{equation}\label{eq: multiplicativity of character}\Ch_{I(\chi)\otimes \omega_\psi}(g)=\Ch_{I(\chi)}(g)\Ch_{\omega_\psi}(g).
\end{equation}

We first record the conjugacy classes of $\SL_2(k)$:
\begin{align*}
\begin{array}{|c|c|c|}
\hline
\textrm{Representative} & \textrm{Number of elements in class } & \textrm{ Number of classes }   \\
\hline
\begin{pmatrix}1 &\\ &1 \end{pmatrix}  &  1&1   \\
\hline
\begin{pmatrix}-1 &\\ &-1 \end{pmatrix}  & 1 &1  \\
\hline
 \begin{pmatrix}1 &1\\ &1\end{pmatrix}&(q^2-1)/2&1\\
 \hline
 \begin{pmatrix}1 &\kappa\\ &1\end{pmatrix}&(q^2-1)/2&1\\
 \hline
\begin{pmatrix}-1 &1\\ &-1\end{pmatrix}&(q^2-1)/2&1\\
\hline
 \begin{pmatrix}-1 &\kappa\\ &-1\end{pmatrix}&(q^2-1)/2&1\\
 \hline
 \begin{pmatrix}x & \\ &x^{-1} \end{pmatrix}, x\ne \pm1 &q(q+1) & (q-3)/2\\
 \hline
\begin{pmatrix}x&y\\ \kappa y&x \end{pmatrix},x\ne \pm1, y \ne 0 &q(q-1) &(q-1)/2\\
\hline
\end{array}
\end{align*}
The above table could be found in \cite{FH}, for example. As a representation of $\SL_2(k)$, the character tables of $I(\chi)$ and $\omega_\psi$ are given in \cite{LZ}. In particular, we know that 
$$\Ch_{I(\chi)}\left(\begin{pmatrix}x&y\\ \kappa y&x \end{pmatrix} \right)=0.$$
Thus by \eqref{eq: multiplicativity of character}, it suffices to consider elements of the form $gv$ with $v\in V$, and $g\in \SL_2(k)$ not of the form $ \begin{pmatrix}x&y\\ \kappa y&x \end{pmatrix}$. Recall that $Z$ is the group generated by $U_{2\alpha+\beta},U_{3\alpha+\beta},U_{3\alpha+2\beta}$. 
\begin{prop}[{\cite[Theorem 4.4, (b)]{Ge}}]
If the function $\Ch_{I(\chi)\otimes\omega_\psi}$ is nonzero on $j\in J$, then $j$ is conjugate to an element in $\SL_2(k)\ltimes Z$. 
\end{prop}

By the above proposition, we need to consider elements in $J$ which are $J$-conjugate to elements of the form $g(0,0,r_3,r_4,r_5), g\in \SL_2(k)$. For a group $H$ and $h_1,h_2\in H$, we write $h_1 \sim_H h_2$ if $h_1=h_0h_2h_0^{-1}$ for some $h_0\in H$.

\begin{table}
\begin{align*}
\begin{array}{|c|c|c|c|c|}
\hline
\textrm{Representative } t & C_J(t) & |J(t)| &  No.&  \Ch_{I(\chi)\otimes\omega_\psi} \\
\hline
\begin{pmatrix}1 &\\ &1 \end{pmatrix}  & J &  1 & 1 & q(q+1)  \\
\hline
\begin{pmatrix}1 &\\ &1 \end{pmatrix} (0,0,0,0,1) & N_{\SL_2}\ltimes V & q^2-1 & 1 & q(q+1)\\
\hline
(0,0,r_3,0,0),r_3\ne 0 & \SL_2\ltimes Z & q^2 & q-1 & q(q+1)\psi(r_3)\\
\hline
 \begin{pmatrix}1 &1\\ &1\end{pmatrix}(0,0,r_3,0,0), \atop r_3\in k & \mu_2 \pair{ U_\beta, U_{\alpha+\beta}, U_{2\alpha+\beta}, U_{3\alpha+2\beta}} %\footnote{ This is a group isomorphic to $\mu_2\times U_\beta\times U_{\alpha+\beta}\times U_{2\alpha+\beta}\times U_{3\alpha+\beta}$, but not exactly all of the elements of the form $\mu_2\bx_\beta(r)(0,r_2,r_3,0,r_5)$. This remark also applies to the next several similar cases.}
 & \frac{q^2-1}{2}q^2 & q & \sqrt{\epsilon_0 q}\psi(r_3)\\
 \hline
 \begin{pmatrix}1 &\kappa\\ &1\end{pmatrix}(0,0,r_3,0,0),\atop r_3\in k& \mu_2 \pair{ U_\beta, U_{\alpha+\beta}, U_{2\alpha+\beta}, U_{3\alpha+2\beta}}  & \frac{q^2-1}{2}q^2  & q & -\sqrt{\epsilon_0 q} \psi(r_3)\\
\hline
 \begin{pmatrix}1 &1\\ &1\end{pmatrix}(0,0,r_3,r_4,0), \atop r_3\in k, r_4\in k^\times/\pair{\pm 1} & \pair{U_\beta,U_{\alpha+\beta}, U_{2\alpha+\beta}, U_{3\alpha+2\beta} } & (q^2-1)q^2 & \frac{q(q-1)}{2} & \sqrt{\epsilon_0 q} \psi(r_3)\\
 \hline
 \begin{pmatrix}1 &\kappa\\ &1\end{pmatrix}(0,0,r_3,r_4,0), \atop r_3\in k, r_4\in k^\times/\pair{\pm 1} & \pair{U_\beta,U_{\alpha+\beta}, U_{2\alpha+\beta}, U_{3\alpha+2\beta} } & (q^2-1)q^2 & \frac{q(q-1)}{2} & -\sqrt{\epsilon_0 q} \psi(r_3)\\
\hline
 h(-1,-1)(0,0,r_3,0,0), \atop r_3\in k & \SL_2\ltimes U_{2\alpha+\beta} & q^4 & q & (q+1)\chi(-1)\epsilon_0\psi(r_3) \\
 \hline
h(-1,-1)\bx_\beta(1)\bx_{2\alpha+\beta}(r_3) & \mu_2\ltimes U_\beta\times U_{2\alpha+\beta} & \frac{q^2-1}{2}q^4 & q & \epsilon_0\chi(-1)\psi(r_3)\\
\hline
h(-1,-1)\bx_\beta(\kappa)\bx_{2\alpha+\beta}(r_3) & \mu_2\ltimes U_\beta\times U_{2\alpha+\beta} & \frac{q^2-1}{2}q^4 & q & \epsilon_0\chi(-1)\psi(r_3) \\
\hline
h(x,x^{-1})(0,0,r_3,0,0), \atop x\ne \pm 1 & A_{\SL_2}\ltimes U_{2\alpha+\beta} & q^5(q+1) & q\frac{q-3}{2} & \epsilon(\chi(x)+\chi(x^{-1}))\psi(r_3)\\
 \hline
\bpm x&y\\ \kappa y&x \epm(0,0,r_3,r_4,r_5),\atop x\ne \pm 1 & * & * & * & 0 \\
\hline
\end{array}
\end{align*}
\caption{Conjugacy classes of $J$}
\label{table: conjugacy class of J}
\end{table}

\begin{lem}\label{lem: conjugacy class of J when p>3}
The following is a set of representatives of $j\in J$ such that $j$ is conjugate to an element of the form $g(0,0,r_3,r_4,r_5)$ with $g\in \SL_2(k)$ and not of the form $\bpm x& y\\ \kappa y &x \epm$, $y\ne 0:$
\begin{enumerate}
    \item $1; (0,0,0,0,1); (0,0,r_3,0,0),r_3\in k^\times;$
    \item $\bx_\beta(b)(0,0,r_3,0,0); \bx_\beta(b)(0,0,r_3,r_4,0),b\in \wpair{1,\kappa},r_3\in k,r_4\in k^\times/\wpair{\pm 1};$
    \item $h(-1,-1)\bx_\beta(b)(0,0,r_3,0,0), r_3\in k, b\in \wpair{1,\kappa};$
    \item $h(x,x^{-1})(0,0,r_3,0,0),x\ne \pm 1, r_3\in k.$
\end{enumerate}
\end{lem}
\begin{proof}
First we notice that if $g\sim_{\SL_2(k)} g'$, then $g(0,0,r_3,r_4,r_5)\sim_J g'(0,0,r_3',r_4',r_5')$ for some $r_3',r_4',r_5'\in k$. Thus we only need to consider the case when $g$ runs over a set of representatives of $\SL_2(k)$-conjugacy classes.

We first consider the case when $g=1$. If $r_3\ne 0$, we have 
\begin{equation}\label{eq: a conjugation equation}(-r_4/(3r_3), -r_5/(3r_3),0,0,0 )(0,0,r_3,r_4,r_5) (-r_4/(3r_3), -r_5/(3r_3),0,0,0)^{-1}=(0,0,r_3,0,0).
\end{equation}
Thus for any $r_4,r_5\in k$, we have $(0,0,r_3,r_4,r_5)\sim_J (0,0,r_3,0,0)$.   If $r_3=0,r_4\ne 0$, then $(0,0,0,r_4,r_5)\sim (0,0,0,0,r_5)$. In fact, we have 
$$w_\beta \bx_\beta(-r_5/r_4) (0,0,0,r_4,r_5) (w_\beta \bx_\beta(-r_5/r_4))^{-1}= (0,0,0,0,r_5).$$
Moreover,  if $r_5\ne 0$, then $(0,0,0,0,r_5)\sim (0,0,0,0,1)$ by considering the action of $h(x,x^{-1})$. 

Next, we consider the case when $g=h(x,x^{-1}),x\ne 1$. We have 
\begin{align*}&h(x, 1/x)(0, 0, r_3,0,0)\\
=&(0, 0, 0, -r_4/(x-1), r_5x/(x-1))h(x, x^{-1})(0, 0, r_3, r_4, r_5)(0, 0, 0, -r_4/(x-1), r_5x/(x-1))^{-1}.
\end{align*}
Thus for any $r_3,r_4,r_5$, we have $h(x, x^{-1})(0, 0, r_3, r_4, r_5)\sim_J h(x,x^{-1})(0,0,r_3,0,0)$. 

Next, we consider the case when $g=\bx_\beta(b), b=1$ or $\kappa$. If, $r_5\ne 0$,  one can check that $g(0,0,r_3,r_4,r_5)\sim _J g(0,0,r_3,r_4,0)$. In fact, we have 
$$ g(0,0,r_3,r_4,0)=\bx_{3\alpha+\beta}(t)g(0,0,r_3,r_4,r_5)\bx_{3\alpha+\beta}(-t), $$
with $t=r_5/b$. Finally, we consider the action of $h(a,a^{-1})$ on $g(0,0,r_3,r_4,0)$. To preserve $r_3$, $a$ should be $\pm 1$. On the other hand, we have $h(-1,-1)g(0,0,r_3,r_4,0)h(-1,-1)=g(0,0,r_3,-r_4,0)$. 

Finally, if $g=h(x,x^{-1})\bx_\beta(b), x= -1$, then one can check that 
$$(0,0,0,s,t).g(0,0,r_3,r_4,r_5).(0,0,0,s,t)^{-1}=g.(0,0,r_3,0,0),$$
with $s=r_4/(1-x),t=-r_5x/(1-x)+br_4x^2/(1-x)^2$. This completes the proof of the lemma.
\end{proof}
Table \ref{table: conjugacy class of J} gives the conjugacy classes of $J$. In Table \ref{table: conjugacy class of J}, for an element $t\in J$, the set $C_J(t)$ is the centralizer of $t$ in $J$ and $J(t)$ is the set of $J$-congugacy classes of $t$.  The centralizer $C_J(t)$ is essentially computed in \cite{Ch}. We have $|J(t)|=|J|/|C_J(t)|$. Note that $|J|=q^6(q^2-1)$. The column ``No." means the number of classes of a given form. Note that the last column is given by \eqref{eq: multiplicativity of character} using the character tables of $I(\chi)$ and $\omega_\psi$, which can be found in \cite[$\S$2]{LZ}.

\subsection{Proof of Theorem \ref{thm: multiplicity one} when \texorpdfstring{$p>3$}{Lg}}

Following \cite{CR}, let $\fH_i, i=2,3,6$ be the 3 anisotropic torus of $\RG_2(k)$ such that $\fH_2\cong \BZ_{q+1}\times \BZ_{q+1},$ $\fH_3\cong \BZ_{q^2+q+1}$ and $\fH_6\cong \BZ_{q^2-q+1}$. 
For $i=2,3,6$, in \cite{CR}, Chang-Ree associated a class function $X_i(\pi_i)$ of $\RG_2(k)$ for each character $\pi_i$ of $\fH_i$.
\begin{prop}\label{prop: pre multiplicity one}
Let $\Pi$ be a representation of $\RG_2(k)$ of the form $X_i(\pi_i)$ with $i=2,3$, or $6$, and $\chi$ be a character of $k^\times$. Then we have
$$\pair{\Pi|_J,I(\chi)\otimes \omega_\psi}=1.$$
\end{prop}
\begin{rmk}
\rm{Here we do not require that $\Pi$ or $I(\chi)$ is irreducible. Note that $$\dim\Hom_J(\Pi,I(\chi)\otimes \omega_\psi)=\pair{\Pi|_J,I(\chi)\otimes\omega_\psi}.$$ Thus the above proposition shows that $\dim\Hom_J(\Pi,I(\chi)\otimes \omega_\psi)\le 1$.}
\end{rmk}

Proof of Proposition \ref{prop: pre multiplicity one} relies on a brute force computation. We first give the character table of $X_i(\pi_i)$ when restricted to $J$,
Table \ref{table: character table of Xi}, which follows from results in \cite{CR}. Here $\epsilon(\pi_2)$ is a number depending on the character $\pi_2$.
\begin{table}
\begin{align*}
\begin{array}{|c|c|c|c|c|}
\hline
\textrm{Representative } t & X_2(\pi_2) &X_3(\pi_3) &  X_6(\pi_6)&  \Ch_{I(\chi)\otimes\omega_\psi} \\
\hline
\begin{pmatrix}1 &\\ &1 \end{pmatrix}  & \frac{(q^2-1)(q^6-1)}{(q+1)^2} & \frac{(q^2-1)(q^6-1)}{q^2+q+1} &\frac{(q^2-1)(q^6-1)}{q^2-q+1} &q(q+1)  \\
\hline
\begin{pmatrix}1 &\\ &1 \end{pmatrix} (0,0,0,0,1)&-(q-1)(q^2-q+1)&(q-1)(q^2-1) &-(q+1)(q^2-1) &q(q+1)\\
\hline
(0,0,r_3,0,0),r_3\ne 0&(q-1)(2q-1)&-(q^2-1)&-(q^2-1)&q(q+1)\psi(r_3)\\
\hline
\bx_\beta(1)\bx_{2\alpha+\beta}(r_3), \atop r_3\in k&* &*&*&\sqrt{\epsilon_0 q}\psi(r_3)\\
 \hline
 \bx_\beta(\kappa)\bx_{2\alpha+\beta}(r_3),\atop r_3\in k &*&*&* &-\sqrt{\epsilon_0 q} \psi(r_3)\\
\hline
 \begin{pmatrix}1 &1\\ &1\end{pmatrix}(0,0,r_3,r_4,0), \atop r_3\in k, r_4\in k^\times/\pair{\pm 1} &*&*&* &\sqrt{\epsilon_0 q} \psi(r_3)\\
 \hline
 \begin{pmatrix}1 &\kappa\\ &1\end{pmatrix}(0,0,r_3,r_4,0), \atop r_3\in k, r_4\in k^\times/\pair{\pm 1} &* &*&* &-\sqrt{\epsilon_0 q} \psi(r_3)\\
\hline
 h(-1,-1)&(q-1)^2\epsilon(\pi_2)&0&0& (q+1)\chi(-1)\epsilon_0 \\
\hline
 h(-1,-1)(0,0,r_3,0,0), \atop r_3\in k^\times&-(q-1)\epsilon(\pi_2)&0&0& (q+1)\chi(-1)\epsilon_0\psi(r_3) \\
 \hline
h(-1,-1)\bx_\beta(1)&-(q-1)\epsilon(\pi_2)&0&0&\epsilon_0\chi(-1)\\
 \hline
h(-1,-1)\bx_\beta(1)\bx_{2\alpha+\beta}(r_3),\atop r_3\in k^\times &\epsilon(\pi_2)&0&0 &\epsilon_0\chi(-1)\psi(r_3)\\
\hline
h(-1,-1)\bx_\beta(\kappa)&-(q-1)\epsilon(\pi_2)&0&0 & \epsilon_0\chi(-1)\\
\hline
h(-1,-1)\bx_\beta(\kappa)\bx_{2\alpha+\beta}(r_3), \atop r_3\in k^\times &\epsilon(\pi_2)&0&0 & \epsilon_0\chi(-1)\psi(r_3)\\
\hline
h(x,x^{-1})(0,0,r_3,0,0), \atop x\ne \pm 1&0 &0&0&\epsilon(\chi(x)+\chi(x^{-1}))\psi(r_3)\\
 \hline
\begin{pmatrix}x&y\\ \kappa y&x \end{pmatrix}(0,0,r_3,r_4,r_5),\atop x\ne \pm1 &* &*&* & 0 \\
\hline
\end{array}
\end{align*}
\caption{Character table of $X_i(\pi_i)$}
\label{table: character table of Xi}
\end{table}

The middle missing part of Table \ref{table: character table of Xi} depends on $q\equiv 1\mod 3$ or $q\equiv -1 \mod 3$, which will be described separately below. Note that if $q\equiv 1\mod 3$, then $\kappa$ is a non-cube in $\BF_q$ since it is assumed to be a generator of $\BF_{q}^\times$. If $q\equiv -1$, we fix an element $\zeta\in \BF_q$ such that $x^3-3x-\zeta$ is irreducible over $\BF_q$.

Let $u$ be a unipotent element which appeared in the missing part. The value of the characters $X_i(\pi_i)$ depends on $|C_{\RG_2(k)}(u)|$, the size of the centralizer of $u$ in $\RG_2(k)$. The detailed information of $|C_{\RG_2(k)}(u)|$ is given in \cite{Ch}, which we will give a brief review below. 

 We first consider the case when $q\equiv 1 \mod 3$. If $rr_3\in k^{\times,2}$, then $\bx_\beta(r)\bx_{2\alpha+\beta}(r_3)\sim_{\RG_2(k)} \bx_\beta(1)\bx_{2\alpha+\beta}(1)$, whose centralizer in $\RG_2(k)$ has size $6q^4$, see \cite[p.202]{Ch}. If $rr_3\in \kappa k^{\times,2}$, then $\bx_\beta(r)\bx_{2\alpha+\beta}(r_3)\sim_{\RG_2(k)} \bx_\beta(\kappa)\bx_{2\alpha+\beta}(1) $, whose centralizer has size $2q^4$, see \cite[p.202]{Ch}.  Then, if $r/r_4\in k^{\times,3}$, one has that $\bx_\beta(r)\bx_{3\alpha+\beta}(r_4)\sim_{\RG_2(k)} \bx_\beta(1)\bx_{3\alpha+\beta}(1)\sim_{\RG_2(k)} \bx_\beta(1)\bx_{2\alpha+\beta}(1)$ (see \cite[p.197]{Ch}), whose centralizer has size $6q^4$. If $r/r_4\notin k^{\times,3}$, then $\bx_\beta(r)\bx_{3\alpha+\beta}(r_4)\sim_{\RG_2(k)} \bx_\beta(1)\bx_{3\alpha+\beta}(\kappa)$, whose centralizer in $\RG_2(k)$ has size $3q^4$, see \cite[p.202]{Ch}. Finally, we consider the conjugacy class $\bx_\beta(r)\bx_{2\alpha+\beta}(r_3)\bx_{3\alpha+\beta}(r_4), rr_3r_4\ne 0$. We first have 
\begin{align*}&\bx_\beta(r)(0,0,r_3,r_4,r_5)\\
=&h(-r_3^2/r_4,r_4/r_3)w_\alpha^{-1}\bx_\alpha(-1)\bx_\beta(1)(0,0,-1,z,0)(h(-r_3^2/r_4,r_4/r_3)w_\alpha^{-1}\bx_\alpha(-1))^{-1}
\end{align*}
for some appropriate $r_5$, where $z=-2-\frac{rr_4^2}{r_3^3}$. Thus we have 
$$\bx_\beta(r)\bx_{2\alpha+\beta}(r_3)\bx_{3\alpha+\beta}(r_4)\sim_{\RG_2(k)} \bx_\beta(r)(0,0,r_3,r_4,r_5)\sim_{\RG_2(k)} \bx_\beta(1)(0,0,-1,-2-rr_4^2/r_3^3,0). $$
We write $-2-rr_4^2/r_3^3=t+t^{-1} $. If $t=\pm1$, then according the calculation in \cite[p.196-197]{Ch}, one can check that \footnote{\label{footnote: a conjugation relation}This relation is not explicitly given in \cite{CR}. Due to its importance for our calculation, we give some details in this footnote. Let $\varphi_\alpha: \SL_2(k)\ra \RG_2(k)$ be the embedding such that $\varphi_\alpha\left( \bpm 1&x \\ &1 \epm \right)=\bx_\alpha(x)$ and $\varphi_\alpha(\diag(a,a^{-1}))=h_\alpha(a,a^{-1})$. For $g,h\in \RG_2(k)$, denote the conjugation $g^{-1}hg$ by $h.g$. The conjugation of $\varphi_\alpha(g)$ for $g\in \SL_2(k)$ on $\bx_\beta(r_0)(0,r_2,r_3,r_4,0)$ is given in \cite[p.196, (3.5)]{Ch}. From that description, one can check the following relations $\bx_\beta(1)(0,0,-1,2,0).\varphi_\alpha\bpm 1& 1/2\\ -1 &1/2 \epm=(0,0,1,1/2,0),$ $(0,0,1,1/2,0).\varphi_\alpha \bpm &1\\ -1 & \epm=\bx_\beta(1/2)(0,1,0,0,0),$ and $\bx_\beta(1/2)(0,1,0,0,0).\varphi_\alpha\bpm 1& 0\\ -1/6 & 1 \epm=\bx_{\alpha+\beta}(1)$. This shows that $\bx_\beta(1)(0,0,-1,2,0)\sim_{\RG_2(k)} \bx_{\alpha+\beta}(1)$. }
$$\bx_\beta(r)\bx_{2\alpha+\beta}(r_3)\bx_{3\alpha+\beta}(r_4)\sim_{\RG_2(k)} \bx_{\alpha+\beta}(1)\sim_{\RG_2(k)} \bx_{2\alpha+\beta}(1), $$
whose centralizer has order $q^4(q^2-1)$.  If $t\in \BF_q$, $t\ne \pm1$, then by \cite[p.197-198]{Ch}, we have 
$$\bx_\beta(r)\bx_{2\alpha+\beta}(r_3)\bx_{3\alpha+\beta}(r_4)\sim_{\RG_2(k)} \bx_\beta(1)\bx_{3\alpha+\beta}(t^{-1}), $$
whose centralizer has size $6q^4$ if $t\in k^{\times,3}$, and $3q^4$ if $t\in k^{\times}- k^{\times,3}$. If $t\in \BF_{q^2}-\BF_q$, then $ \bx_\beta(r)\bx_{2\alpha+\beta}(r_3)\bx_{3\alpha+\beta}(r_4)\sim_{\RG_2(k)} \bx_\beta(1)\bx_{2\alpha+\beta}(\kappa)$ (see \cite[p.198]{Ch}), whose centralizer has size $2q^4$. 
For $u=\bx_\beta(r)\bx_{2\alpha+\beta}(r_3)\bx_{3\alpha+\beta}(r_4)$ with $rr_3r_4\ne 0$, we write $z(r,r_3,r_4)=-2-rr_4^2/r_3^3=t+t^{-1}$. We write $t$ in the above expression as $t=t(r,r_3,r_4)$. From \cite{CR}, the missing part of Table \ref{table: character table of Xi} (from 5th row to 8th row) when $q\equiv 1\mod 3$ is given in Table \ref{table: when q equiv 1 mod 3}.
\begin{table}
\begin{align*}
\begin{array}{|c|c|c|c|c|}
\hline
u& |C_{\RG_2(k)}(u)|& X_2(\pi_2) & X_3(\pi_3) &X_6(\pi_6)\\
\hline
\bx_\beta(r),r\ne 0 &* &-(q-1)(q^2-q+1)&(q-1)(q^2-1)&-(q+1)(q^2-1)\\
\hline
\bx_\beta(1)\bx_{2\alpha+\beta}(r_3), \atop r_3\in k^{\times,2}&6q^4&-4q+1& q+1&-q+1 \\
\hline
\bx_\beta(1)\bx_{2\alpha+\beta}(r_3), \atop r_3\in \kappa k^{\times,2}&2q^4 &-2q+1&-q+1 & q+1\\
 \hline
 \bx_\beta(\kappa)\bx_{2\alpha+\beta}(r_3),\atop r_3\in k^{\times,2} &2q^4&-2q+1&-q+1 &q+1\\
\hline
 \bx_\beta(\kappa)\bx_{2\alpha+\beta}(r_3),\atop r_3\in \kappa k^{\times,2} &6q^4&-4q+1&q+1 &-q+1\\
\hline
\bx_\beta(1)\bx_{3\alpha+\beta}(r_4) \atop r_4\in k^{\times,3} &6q^4&-4q+1&q+1&-q+1\\
\hline
\bx_\beta(1)\bx_{3\alpha+\beta}(r_4) \atop r_4\notin k^{\times,3} &3q^4&-q+1&-2q+1 &2q+1\\
\hline
\bx_\beta(\kappa)\bx_{3\alpha+\beta}(r_4) \atop r_4\in \kappa k^{\times,3} &6q^4&-4q+1& q+1 & -q+1\\
\hline
\bx_\beta(\kappa)\bx_{3\alpha+\beta}(r_4) \atop r_4\notin \kappa k^{\times,3} &3q^4&-q+1& -2q+1&2q+1\\
\hline
 \bx_\beta(r)(0,0,r_3,r_4,0), \atop   t(r,r_3,r_4)=\pm 1 &q^4(q^2-1)&(q-1)(2q-1)&-q^2+1&-q^2+1\\
\hline
 \bx_\beta(r)(0,0,r_3,r_4,0), \atop   t(r,r_3,r_4)\in k^{\times,3}-\wpair{\pm 1} &6q^4&-4q+1&q+1&-q+1\\
\hline
 \bx_\beta(r)(0,0,r_3,r_4,0), \atop   t(r,r_3,r_4)\in k^\times- k^{\times,3} &3q^4&-q+1&-2q+1  & 2q+1\\
\hline
 \bx_\beta(r)(0,0,r_3,r_4,0), \atop   t(r,r_3,r_4)\notin k &2q^4&-2q+1&-q+1 & q+1\\
\hline
\end{array}
\end{align*}
\caption{Missing part of Table \ref{table: character table of Xi} when $q\equiv 1\mod 3$}
\label{table: when q equiv 1 mod 3}
\end{table}

Using Table \ref{table: conjugacy class of J}, Table \ref{table: character table of Xi} and Table \ref{table: when q equiv 1 mod 3}, we can compute the pair $\pair{\Pi|_J,I(\chi)\otimes \omega_\psi}$. Recall that 
\begin{align}
|J|\pair{\Pi|_J,I(\chi)\otimes \omega_\psi}&=\sum_{g\in J}\ov{\Ch_\Pi}(g)\Ch_{I(\chi)\otimes \omega_\psi}(g) \nonumber\\
&=\sum_{t}|J(t)|\ov{\Ch_\Pi}(t)\Ch_{I(\chi)\otimes \omega_\psi}(t), \label{eq: computation of the pair}
\end{align}
where in \eqref{eq: computation of the pair}, $t$ runs over a complete set of representatives of conjugacy classes of $J$ and $|J(t)|$ is the number of elements in the conjugacy class $J(t)$.

\begin{lem}\label{lem: preparation for the computation of the pair}
Let $\Pi$ be one of $X_i(\pi_i)$ for $i=2,3,6$. Then we have 
\begin{enumerate}
    \item The contribution of conjugacy classes of the form $h(-1,-1)u$, where $u$ is an unipotent element, to \eqref{eq: computation of the pair} is zero.
     \item The contribution of conjugacy classes of the form $\bx_\beta(r),r\in k^\times$, to \eqref{eq: computation of the pair} is zero.
    \item The contribution of conjugacy classes of the form $\bx_\beta(1)\bx_{3\alpha+\beta}(r_4), r_4\in k^{\times,3}$, and the contribution of $\bx_\beta(\kappa)\bx_{3\alpha+\beta}(r_4), r_4\in \kappa k^{\times,3}$, to \eqref{eq: computation of the pair} are cancelled out. Similarly, the contribution of  $\bx_\beta(1)\bx_{3\alpha+\beta}(r_4)$ , $r_4\in k^\times-k^{\times,3}$, and the contribution of $ \bx_\beta(\kappa)\bx_{3\alpha+\beta}(r_4)$, $r_4\in k^\times- \kappa k^{\times,3}$,  to \eqref{eq: computation of the pair} are cancelled out.
\end{enumerate}
\end{lem}
\begin{proof}
We only give some details for the proof of (1) when $\Pi=X_2(\pi_2)$, and the proofs of the other cases are similar or just follow from a simple observation. By Table \ref{table: conjugacy class of J} and Table \ref{table: character table of Xi}, the contribution of conjugacy classes of the form $h(-1,-1)u$ to \eqref{eq: computation of the pair} is 
\begin{align*}
    & \ q^4(q+1)\chi(-1)\epsilon_0 (q-1)^2 \epsilon(\pi_2)\\
    -& \ q^4(q+1)\chi(-1)\epsilon_0\left(\sum_{r_3\in k^\times}\psi(r_3)\right)(q-1)\epsilon(\pi_2)\\
    +&\ \frac{q^2-1}{2}q^4\epsilon_0 \chi(-1)(-(q-1))\epsilon(\pi_2)\cdot 2\\
    +&\ \frac{q^2-1}{2}q^4\epsilon_0\chi(-1)\epsilon(\pi_2)\left(\sum_{r_3\in k^\times}\psi(r_3)\right).
\end{align*}
A simple calculation shows that the above summation is zero.
\end{proof}

The following lemma is Proposition \ref{prop: pre multiplicity one} when $q\equiv 1\mod 3$.
\begin{lem}\label{lem: proof when q equiv 1 mod 3}
Let $\Pi$ be one of $X_i(\pi_i)$ for $i=2,3,6$. If $q\equiv 1\mod 3$, then 
$$\pair{\Pi|_J,I(\chi)\otimes \omega_\psi}=1.$$
\end{lem}

\begin{proof} We compute \eqref{eq: computation of the pair} for $\Pi=X_2(\pi_2), X_3(\pi_3), X_6(\pi_6)$, separately. If $\Pi=X_2(\pi_2)$, by Tables \ref{table: conjugacy class of J}, \ref{table: character table of Xi}, \ref{table: when q equiv 1 mod 3} and Lemma \ref{lem: preparation for the computation of the pair}, we have 
\begin{align*}
&\ |J|\pair{X_2(\pi_2), I(\chi)\otimes \omega_\psi}\\
=&\ q(q+1)(q^2-1)(q^6-1)/(q+1)^2-(q^2-1)q(q+1)(q-1)(q^2-q+1)\\
+&\  q^2(q-1)(2q-1)q(q+1)\left(\sum_{r_3\in k^\times}\psi(r_3)\right)\\
+&\ \frac{q^2-1}{2}q^2\sqrt{\epsilon_0 q}(-4q+1)(A_1-A_\kappa)-\frac{q^2-1}{2}q^2\sqrt{\epsilon_0 q}(-2q+1)(A_1-A_\kappa)\\
+&\ q^2(q^2-1)\sqrt{\epsilon_0 q}((q-1)(2q-1)(B_1^0-B_\kappa^0))\\
+&\ q^2(q^2-1)\sqrt{\epsilon_0 q}\left[(-4q+1)(B_1^1-B_\kappa^1)+(-q+1)(B_1^2-B_\kappa^2)+(-2q+1)(B_1^3-B_\kappa^3) \right],
\end{align*}
where 
\begin{align}\label{eq: defn of A}
A_1=\sum_{r_3\in k^{\times,2}}\psi(r_3), \quad A_{\kappa}=\sum_{r_3\in \kappa k^{\times,2}}\psi(r_3),
\end{align}
and
\begin{align}\label{eq: defn of B}
\begin{split}
B_r^0&=\sum_{r_3\in k^\times,r_4\in k^{\times}/{\pm 1}, t(r,r_3,r_4)\in \wpair{\pm 1}}\psi(r_3),\\
B_r^1&=\sum_{r_3\in k^\times,r_4\in k^{\times}/{\pm 1}, t(r,r_3,r_4)\in k^{\times,3}-\wpair{\pm 1}}\psi(r_3),\\
B_r^2&=\sum_{r_3\in k^\times,r_4\in k^{\times}/{\pm 1}, t(r,r_3,r_4)\in k^\times- k^{\times,3}}\psi(r_3),\\
B_r^3&=\sum_{r_3\in k^\times,r_4\in k^{\times}/{\pm 1}, t(r,r_3,r_4)\notin k^{\times}}\psi(r_3),
\end{split}
\end{align}
for $r=1,\kappa$. The computations of $A_1-A_\kappa,B_1^i-B_\kappa^i$ for $i=0,1,2,3$ are given in the appendix, see Lemmas \ref{lem: gauss sum} and \ref{lem: computation of B}, and the results read as 
\begin{align}
\begin{split}\label{eq: gauss sum A and B}
A_1-A_\kappa&=\sqrt{\epsilon_0 q},\\
   B_1^0-B_\kappa^0&=\epsilon_0\sqrt{\epsilon_0 q},\\
    B_1^1-B_\kappa^1&=-\frac{1}{2}(1+\epsilon_0)\sqrt{\epsilon_0 q},\\
    B_1^2-B_\kappa^2&=0,\\
    B_1^3-B_\kappa^3&=\frac{1}{2}(1-\epsilon_0)\sqrt{\epsilon_0 q}.
    \end{split}
\end{align}
Plugging these formulas into the computation of $\pair{X_2(\pi_2), I(\chi)\otimes\omega_\psi}$, it follows that 
$$|J|\pair{X_2(\pi_2),I(\chi)\otimes\omega_\psi}=q^6(q^2-1).$$
Thus we have
$$\pair{X_2(\pi_2), I(\chi)\otimes\omega_\psi}=1. $$

Similarly, we have
\begin{align*}
 &\ |J|\pair{X_3(\pi_3), I(\chi)\otimes\omega_\psi}\\
=&\ q(q+1)(q^2-1)(q^6-1)/(q^2+q+1)+(q^2-1)q(q+1)(q-1)(q^2-1)\\
-&\ q^2q(q+1)(q^2-1)\left(\sum_{r_3\in k^\times}\psi(r_3)\right)\\
+&\ \frac{q^2-1}{2}q^2\sqrt{\epsilon_0 q}((q+1)(A_1-A_\kappa)-(-q+1)(A_1-A_\kappa))\\
+&\ q^2(q^2-1)\sqrt{\epsilon_0 q}((-q^2+1)(B_1^0-B_\kappa^0)\\
+&\ q^2(q^2-1)\sqrt{\epsilon_0 q}((q+1)(B_1^1-B_\kappa^1)+(-2q+1)(B_1^2-B_\kappa^2)+(-q+1)(B_1^3-B_\kappa^3)),
\end{align*}
where $A_r,B_r^i$ for $r=1,\kappa, i=0,1,2,3$ are defined in \eqref{eq: defn of A} and \eqref{eq: defn of B}. Plugging the formulas \eqref{eq: gauss sum A and B} into the computation of $\pair{X_3(\pi_3), I(\chi)\otimes\omega_\psi}$, we obtain that 
$$\pair{X_3(\pi_3),I(\chi)\otimes\omega_\psi}=1.$$

A similar calculation shows that
\begin{align*}
&\ |J|\pair{X_6(\pi_6),I(\chi)\otimes\omega_\psi}\\
=&\ q(q+1)(q^2-1)(q^6-1)/(q^2-q+1)-(q^2-1)(q+1)(q^2-1)q(q+1)\\
-&\ q^2(q^2-1)q(q+1)\left(\sum_{r_3\in k^\times}\psi(r_3)\right)\\
+ & \ \frac{q^2-1}{2}q^2\sqrt{\epsilon_0 q}((-q+1)(A_1-A_\kappa)-(q+1)(A_1-A_\kappa))\\
+&\ q^2(q^2-1)\sqrt{\epsilon_0 q}(-q^2+1)(B_1^0-B_\kappa^0)\\
+&\ q^2(q^2-1)\sqrt{\epsilon_0 q}((-q+1)(B_1^1-B_\kappa^1)+(-2q+1)(B_1^2-B_\kappa^2)+(q+1)(B_1^3-B_\kappa^3)).
\end{align*}
By formulas \eqref{eq: gauss sum A and B}, we also obtain that $\pair{X_6(\pi_6),I(\chi)\otimes\omega_\psi}=1$.
\end{proof}

We next consider the case when $q\equiv -1 \mod 3.$ In this case, we have $k^{\times}=k^{\times,3}$. If $rr_3\in k^{\times,2}$, then $\bx_\beta(r)\bx_{2\alpha+\beta}(r_3)\sim_{\RG_2(k)} \bx_\beta(1)\bx_{2\alpha+\beta}(1)$, whose centralizer in $\RG_2(k)$ has size $2q^4$, see \cite[p.202]{Ch}. If $rr_3\in \kappa k^{\times,2}$, then $\bx_\beta(r)\bx_{2\alpha+\beta}(r_3)\sim_{\RG_2(k)} \bx_\beta(\kappa)\bx_{2\alpha+\beta}(1) $, whose centralizer has size $6q^4$, see \cite[p.202]{Ch}.
For any $r_4$, we have $\bx_\beta(r)\bx_{3\alpha+\beta}(r_4)\sim_{\RG_2(k)} \bx_\beta(1)\bx_{3\alpha+\beta}(1)\sim_{\RG_2(k)} \bx_\beta(1)\bx_{2\alpha+\beta}(1)$ (see \cite[p.197]{Ch}), whose centralizer has size $2q^4$. 
 Finally, we consider the conjugacy class $\bx_\beta(r)\bx_{2\alpha+\beta}(r_3)\bx_{3\alpha+\beta}(r_4), rr_3r_4\ne 0$. As in the previous case, we still have
\begin{align*}\bx_\beta(r)(0,0,r_3,r_4,r_5)\sim_{\RG_2(k)} \bx_\beta(1)(0,0,-1,z,0),
\end{align*}
with some appropriate $r_5$, where $z=-2-\frac{rr_4^2}{r_3^3}$. Thus we have 
$$\bx_\beta(r)\bx_{2\alpha+\beta}(r_3)\bx_{3\alpha+\beta}(r_4)\sim_{\RG_2(k)} \bx_\beta(r)(0,0,r_3,r_4,r_5)\sim_{\RG_2(k)} \bx_\beta(1)(0,0,-1,-2-rr_4^2/r_3^3,0). $$
We write $-2-rr_4^2/r_3^3=t+t^{-1} $ for some $t\in k_2=\BF_{q^2}$. Let $t\in k_2^\times-k^\times$ with $t+t^{-1}\in k$. Then one can check that $t^{1+q}=1$, i.e., $t\in k_2^1$, the norm 1 subgroup of $k_2^\times$. Thus such an element $t$ with $-2-rr_4^2/r_3^3=t+t^{-1} $ must be in $k^\times$ or in $k_2^1$.  If $t\in \wpair{\pm 1}=k^\times\cap k_2^1$, then, $$\bx_\beta(r)\bx_{2\alpha+\beta}(r_3)\bx_{3\alpha+\beta}(r_4)\sim_{\RG_2(k)} \bx_{2\alpha+\beta}(1), $$
as in the previous case. If $t\in k^\times-\wpair{\pm 1}$, then by \cite[p.197-198]{Ch}, we have 
$$\bx_\beta(r)\bx_{2\alpha+\beta}(r_3)\bx_{3\alpha+\beta}(r_4)\sim_{\RG_2(k)} \bx_\beta(1)\bx_{3\alpha+\beta}(t^{-1}), $$
whose centralizer has size $2q^4$. If $t\in (k_2^{1}-\wpair{\pm 1})\cap k_2^{\times,3}$, then $ \bx_\beta(r)\bx_{2\alpha+\beta}(r_3)\bx_{3\alpha+\beta}(r_4)\sim_{\RG_2(k)} \bx_\beta(1)\bx_{2\alpha+\beta}(\kappa)$ (see \cite[p.198]{Ch}), whose centralizer has size $6q^4$. If $t\in k_2^{1}-k_2^{\times,3}$, the centralizer of $\bx_\beta(r)\bx_{2\alpha+\beta}(r_3)\bx_{3\alpha+\beta}(r_4) $ has size $3q^4$.

For $u=\bx_\beta(r)\bx_{2\alpha+\beta}(r_3)\bx_{3\alpha+\beta}(r_4)$ with $rr_3r_4\ne 0$, we write $z(r,r_3,r_4)=-2-rr_4^2/r_3^3=t+t^{-1}$. We write $t$ in the above expression as $t=t(r,r_3,r_4)$. From \cite{CR}, for $q\equiv -1 \mod 3$, the missing part of Table \ref{table: character table of Xi} (from 5th row to 8th row) is given in Table \ref{table: when q equiv -1 mod 3}.

\begin{table}
\begin{align*}
\begin{array}{|c|c|c|c|c|}
\hline
u& |C_{\RG_2(k)}(u)|& X_2(\pi) & X_3(\pi_3) &X_6(\pi_6)\\
\hline
\bx_\beta(r),r\ne 0 &* &-(q-1)(q^2-q+1)&(q-1)(q^2-1)&-(q+1)(q^2-1)\\
\hline
\bx_\beta(1)\bx_{2\alpha+\beta}(r_3), \atop r_3\in k^{\times,2}&2q^4&-2q+1& -q+1&q+1 \\
\hline
\bx_\beta(1)\bx_{2\alpha+\beta}(r_3), \atop r_3\in \kappa k^{\times,2}&6q^4 &-4q+1&q+1 & -q+1\\
 \hline
 \bx_\beta(\kappa)\bx_{2\alpha+\beta}(r_3),\atop r_3\in k^{\times,2} &6q^4&-4q+1&q+1 &-q+1\\
\hline
 \bx_\beta(\kappa)\bx_{2\alpha+\beta}(r_3),\atop r_3\in \kappa k^{\times,2} &2q^4&-2q+1&-q+1 &q+1\\
\hline
\bx_\beta(u)\bx_{3\alpha+\beta}(r_4) \atop u=1,\kappa, r_4\in k^{\times} &2q^4&-2q+1&-q+1&q+1\\
\hline
 \bx_\beta(r)(0,0,r_3,r_4,0), \atop   t(r,r_3,r_4)\in \wpair{\pm 1} &q^4(q^2-1)&(q-1)(2q-1)&-q^2+1&-q^2+1\\
 \hline
 \bx_\beta(r)(0,0,r_3,r_4,0), \atop   t(r,r_3,r_4)\in k^{\times}-\wpair{\pm 1} &2q^4&-2q+1&-q+1&q+1\\
\hline
 \bx_\beta(r)(0,0,r_3,r_4,0), \atop   t(r,r_3,r_4)\in (k_2^{1}-\wpair{\pm 1})\cap k_2^{\times,3} &6q^4&-4q+1&q+1  & -q+1\\
\hline
 \bx_\beta(r)(0,0,r_3,r_4,0), \atop   t(r,r_3,r_4)\in k_2^1-k_2^{\times,3} &3q^4&-q+1&-2q+1 & 2q+1\\
\hline
\end{array}
\end{align*}
\caption{Missing part of Table \ref{table: character table of Xi} when $q\equiv -1 \mod 3$}
\label{table: when q equiv -1 mod 3}
\end{table}

The following lemma is Proposition \ref{prop: pre multiplicity one} when $q\equiv -1 \mod 3$.
\begin{lem}\label{lem: proof when q equiv -1 mod 3}
Let $\Pi$ be one of $X_i(\pi_i)$ for $i=2,3,6$. If $q\equiv -1 \mod 3$, then we have 
$$\pair{\Pi,I(\chi)\otimes\omega_\psi}=1.$$
\end{lem}

\begin{proof}
From Table \ref{table: conjugacy class of J}, \ref{table: character table of Xi}, \ref{table: when q equiv -1 mod 3} and Lemma \ref{lem: preparation for the computation of the pair}, we have 
\begin{align*}
 &\ |J|\pair{X_2(\pi_2), I(\chi)\otimes\omega_\psi}\\
=&\ q(q+1)(q^2-1)(q^6-1)/(q+1)^2-(q^2-1)q(q+1)(q-1)(q^2-q+1)\\
+&\ q^2q(q+1)(q-1)(2q-1)\left(\sum_{r_3\in k^\times}\psi(r_3)\right)\\
+&\ \frac{q^2-1}{2}q^2\sqrt{\epsilon_0 q}((-2q+1)(A_1-A_\kappa)-(-4q+1)(A_1-A_\kappa))\\
+&\ q^2(q^2-1)\sqrt{\epsilon_0 q}(q-1)(2q-1)(C_1^0-C_\kappa^1)\\
+&\ q^2(q^2-1)\sqrt{\epsilon_0 q}((-2q+1)(C_1^1-C_\kappa^1)+(-4q+1)(C_1^2-C_\kappa^2)+(-q+1)(C_1^3-C_\kappa^3)),
\end{align*}
where for $r=1,\kappa$, $A_r$ are defined as before, and
\begin{align}\label{eq: defn of C}
\begin{split}
C_r^0&=\sum_{r_3\in k^\times,r_4\in k^{\times}/\wpair{\pm 1}, t(r,r_3,r_4)\in \wpair{\pm 1}}\psi(r_3),\\
C_r^1&=\sum_{r_3\in k^\times,r_4\in k^{\times}/\wpair{\pm 1}, t(r,r_3,r_4)\in k^{\times}-\wpair{\pm 1}}\psi(r_3),\\
C_r^2&=\sum_{r_3\in k^\times,r_4\in k^{\times}/\wpair{\pm 1}, t(r,r_3,r_4)\in  (k_2^{1}-\wpair{\pm 1})\cap k_2^{\times,3}}\psi(r_3),\\
C_r^3&=\sum_{r_3\in k^\times,r_4\in k^{\times}/\wpair{\pm 1}, t(r,r_3,r_4)\in k_2^{1}-k_2^{\times,3}}\psi(r_3).
\end{split}
\end{align}
The quantities $C_1^i-C_\kappa^i$ are computed in Lemma \ref{lem: computation of C}. Applying those formulas in Lemma \ref{lem: computation of C}, a straightforward calculation shows that 
$$\pair{X_2(\pi_2),I(\chi)\otimes\omega_\psi}=1.$$
It is similar to show that 
$$\pair{X_i(\pi_i),I(\chi)\otimes\omega_\psi}=1,$$
for $i=3,6$ as well. We omit the details here.
\end{proof}

\begin{proof}[\textbf{Proof of Theorem $\ref{thm: multiplicity one}$ when $p>3$}]
The irreducible representations of $\RG_2(k)$ have been classified in \cite{CR}. Let $ \fH_1$ be the maximal split torus, 
$$\fH_a=\wpair{h(z^q,z^{1-q}):z^{q^2-1}=1},$$
and 
$$\fH_b=\wpair{h(z,z^q):z^{q^2-1}=1}.$$ For $i=1,2,a,b,3,6$, and a character $\pi_i$ of $\fH_i$, there is an associated character $X_i(\pi_i)$ of $\RG_2(k)$, and when $\pi_i$ is in general position, see \cite[p.398]{CR} for the precise definition, $X_i(\pi_i)$ is irreducible. There are several other isolated classes of irreducible representations of $\RG_2(k)$ constructed using linear combinations of $X_i(\pi_i)$ (when $\pi_i$ is not in general position) and 4 other class functions $Y_i,i=1,2,3,4$. If a representation $\Pi$ is a component of $X_i(\pi_i)$ for $i=1,a,b$, then $\Pi$ is not cuspidal. From the list given in \cite{CR}, it is not hard to see that if $\Pi$ is an irreducible cuspidal representation of $\RG_2(k)$, then it has to be one of following form:
$$X_i(\pi_i), (i=2,3,6), X_{33}=-\frac{1}{3}X_2(\pi_2)+\frac{1}{3}X_6(\pi_6), X_{17}, X_{18}, X_{19}, \ov X_{19}, $$
where $\pi_i$ for $i=2,3,6$ are in general positions, $X_{33}$ appears when $q\equiv -1 \mod 3$, and $X_{17}, X_{18},X_{19},\ov X_{19}$ are defined in \cite[p.402]{CR}. 

We have shown that 
$$\pair{X_i(\pi_i),I(\chi)\otimes\omega_\psi}=1,$$
for $i=2,3,6$, no matter $\pi_i$ is in general position or not. Thus, we get 
$$\pair{X_{33},I(\chi)\otimes \omega_\psi}=0.$$

Finally, to deal with the last 4 isolated cases, we need to compute $\pair{Y_i,I(\chi)\otimes\omega_\psi}$. According to the table given in \cite[p.411]{CR}, we have 
\begin{align*}
\begin{array}{|c|c|c|c|c|c|}
\hline
\textrm{Representative } t & Y_1 &Y_2 &  Y_3 & Y_4&  \Ch_{I(\chi)\otimes\omega_\psi} \\
\hline
\begin{pmatrix}1 &\\ &1 \end{pmatrix}  & 0 & 0 &0& 0&q(q+1)  \\
\hline
\begin{pmatrix}1 &\\ &1 \end{pmatrix} (0,0,0,0,1)&0&0 &0& 0&q(q+1)\\
\hline
(0,0,r_3,0,0),r_3\ne 0&0&0&0&0&q(q+1)\psi(r_3)\\
\hline
\bx_\beta(1)\bx_{2\alpha+\beta}(r_3), \atop r_3\in k&* &0&0&0&\sqrt{\epsilon_0 q}\psi(r_3)\\
 \hline
 \bx_\beta(\kappa)\bx_{2\alpha+\beta}(r_3),\atop r_3\in k &*&0&0 &0&-\sqrt{\epsilon_0 q} \psi(r_3)\\
\hline
 \begin{pmatrix}1 &1\\ &1\end{pmatrix}(0,0,r_3,r_4,0), \atop r_3\in k, r_4\in k^\times/\pair{\pm 1} &*&0&0 &0&\sqrt{\epsilon_0 q} \psi(r_3)\\
 \hline
 \begin{pmatrix}1 &\kappa\\ &1\end{pmatrix}(0,0,r_3,r_4,0), \atop r_3\in k, r_4\in k^\times/\pair{\pm 1} &* &0&0 &0&-\sqrt{\epsilon_0 q} \psi(r_3)\\
\hline
 h(-1,-1)&0&0&0& 0&(q+1)\chi(-1)\epsilon_0 \\
\hline
 h(-1,-1)(0,0,r_3,0,0), \atop r_3\in k^\times&0&0&0&0 &(q+1)\chi(-1)\epsilon_0\psi(r_3) \\
 \hline
h(-1,-1)\bx_\beta(r),r\in k^{\times}&0&0&0&0&\epsilon_0\chi(-1)\\
 \hline
h(-1,-1)\bx_\beta(r)\bx_{2\alpha+\beta}(r_3), \atop rr_3\in k^{\times,2} &0&q&0 &0&\epsilon_0\chi(-1)\psi(r_3)\\
\hline
h(-1,-1)\bx_\beta(r)\bx_{2\alpha+\beta}(r_3), \atop rr_3\in \kappa k^{\times,2}&0&-q&0 &0& \epsilon_0\chi(-1)\psi(r_3)\\
\hline
h(x,x^{-1})(0,0,r_3,0,0), \atop x\ne \pm 1&0 &0&0&0&\epsilon(\chi(x)+\chi(x^{-1}))\psi(r_3)\\
 \hline
\begin{pmatrix}x&y\\ \kappa y&x \end{pmatrix}(0,0,r_3,r_4,r_5),\atop x\ne \pm1 &* &*&* & &0 \\
\hline
\end{array}
\end{align*}
The missing part for $Y_1$ (from 5th row to 8th row) depends on the residue of $q\mod 3$. If $q\equiv 1\mod 3$, then one has 
\begin{align*}
\begin{array}{|c|c|c|}
\hline
u& |C_G(u)|& Y_1  \\
\hline
\bx_\beta(r),r\ne 0 &* &0\\
\hline
\bx_\beta(1)\bx_{2\alpha+\beta}(r_3), \atop r_3\in k^{\times,2}&6q^4&q^2 \\
\hline
\bx_\beta(1)\bx_{2\alpha+\beta}(r_3), \atop r_3\in \kappa k^{\times,2}&2q^4 &-q^2\\
 \hline
 \bx_\beta(\kappa)\bx_{2\alpha+\beta}(r_3),\atop r_3\in k^{\times,2} &2q^4&-q^2\\
\hline
 \bx_\beta(\kappa)\bx_{2\alpha+\beta}(r_3),\atop r_3\in \kappa k^{\times,2} &6q^4&q^2\\
\hline
\bx_\beta(1)\bx_{3\alpha+\beta}(r_4) \atop  r_4\in k^{\times,3} &6q^4&q^2\\
\hline
\bx_\beta(1)\bx_{3\alpha+\beta}(r_4) \atop  r_4\notin k^{\times,3} &3q^4&q^2\\
\hline
\bx_\beta(\kappa)\bx_{3\alpha+\beta}(r_4) \atop  r_4\in \kappa k^{\times,3} &6q^4&q^2\\
\hline
\bx_\beta(\kappa)\bx_{3\alpha+\beta}(r_4) \atop  r_4\notin \kappa k^{\times,3} &3q^4&q^2\\
\hline
 \bx_\beta(r)(0,0,r_3,r_4,0), \atop   t(r,r_3,r_4)\in \wpair{\pm 1} &q^4(q^2-1)&0\\
 \hline
 \bx_\beta(r)(0,0,r_3,r_4,0), \atop   t(r,r_3,r_4)\in k^{\times,3}-\wpair{\pm 1} &6q^4&q^2\\
\hline
 \bx_\beta(r)(0,0,r_3,r_4,0), \atop   t(r,r_3,r_4)\in k^\times-k^{\times,3} &3q^4&q^2\\
\hline
 \bx_\beta(r)(0,0,r_3,r_4,0), \atop   t(r,r_3,r_4)\notin k &2q^4&-q^2\\
\hline
\end{array}
\end{align*}
and when $q\equiv -1\mod 3$, one has 
\begin{align*}
\begin{array}{|c|c|c|}
\hline
u& |C_G(u)|& Y_1  \\
\hline
\bx_\beta(r),r\ne 0 &* &0\\
\hline
\bx_\beta(1)\bx_{2\alpha+\beta}(r_3), \atop r_3\in k^{\times,2}&2q^4&q^2 \\
\hline
\bx_\beta(1)\bx_{2\alpha+\beta}(r_3), \atop r_3\in \kappa k^{\times,2}&6q^4 &-q^2\\
 \hline
 \bx_\beta(\kappa)\bx_{2\alpha+\beta}(r_3),\atop r_3\in k^{\times,2} &6q^4&-q^2\\
\hline
 \bx_\beta(\kappa)\bx_{2\alpha+\beta}(r_3),\atop r_3\in \kappa k^{\times,2} &2q^4&q^2\\
\hline
\bx_\beta(u)\bx_{3\alpha+\beta}(r_4) \atop u=1,\kappa, r_4\in k^{\times} &2q^4&q^2\\
\hline
 \bx_\beta(r)(0,0,r_3,r_4,0), \atop   t(r,r_3,r_4)\in \wpair{\pm 1} &q^4(q^2-1)&0\\
 \hline
 \bx_\beta(r)(0,0,r_3,r_4,0), \atop   t(r,r_3,r_4)\in k^{\times}-\wpair{\pm 1} &2q^4&q^2\\
\hline
 \bx_\beta(r)(0,0,r_3,r_4,0), \atop   t(r,r_3,r_4)\in (k_2^{1}-\wpair{\pm 1})\cap k_2^{\times,3} &6q^4&-q^2\\
\hline
 \bx_\beta(r)(0,0,r_3,r_4,0), \atop   t(r,r_3,r_4)\in k_2^1-k_2^{\times,3} &3q^4&-q^2\\
\hline
\end{array}
\end{align*}
It is easy to see that 
$$\pair{Y_i,I(\chi)\otimes\omega_\psi}=0,$$
for $i=2,3,4$. We next compute $\pair{Y_1,I(\chi)\otimes\omega_\psi}$ when $q\equiv 1\mod 3$. We have 
\begin{align*}
     &\ |J|\pair{Y_1,I(\chi)\otimes \omega_\psi}\\
    =&\ \frac{q^2-1}{2}q^2\sqrt{\epsilon_0 q}(q^2(A_1-A_\kappa)-(-q^2)(A_1-A_\kappa))\\
    +&\ q^2(q^2-1)\sqrt{\epsilon_0 q}(q^2(B_1^1-B_\kappa^1)+q^2(B_1^2-B_\kappa^1)+(-q^2)(B_1^3-B_\kappa^3))\\
    =&\ q^4(q^2-1)\sqrt{\epsilon_0 q}((A_1-A_\kappa)+(B_1^1-B_\kappa^1)-(B_1^3-B_\kappa^3)).
\end{align*}
From the computation of $A_1-A_\kappa$ and $B_1^i-B_\kappa^i$ for $i=0,1,2,3$, in Lemma \ref{lem: gauss sum} and Lemma \ref{lem: computation of B}, one can see that 
$$\pair{Y_1,I(\chi)\otimes\omega_\psi}=0.$$
Similarly, when $q\equiv -1\mod 3$, we also have 
$$\pair{Y_1,I(\chi)\otimes\omega_\psi}=0.$$

From the definitions of $X_{17},X_{18},X_{19},\ov{X}_{19},$ given in \cite[p.402]{CR}, one can check that 
$$\pair{\Pi,I(\chi)\otimes \omega_\psi}=0,$$
if $\Pi=X_{17},X_{18},X_{19}$, or $\ov{X}_{19}$. For example, we have 
$$X_{17}=-\frac{1}{6}X_2(1)+\frac{1}{6}X_6(1)-\frac{1}{2}Y_1+\frac{1}{2}Y_2.$$
Since $\pair{X_i(1),I(\chi)\otimes\omega_\psi}=1$ for $i=2,3,6$, and $\pair{Y_i,I(\chi)\otimes\omega_\psi}=0$, we get 
$$\pair{X_{17},I(\chi)\otimes \omega_\psi}=-\frac{1}{6}+\frac{1}{6}=0.$$
The other 3 cases can be checked similarly. This completes the proof Theorem $\ref{thm: multiplicity one}$.
\end{proof}

\section{Proof of Theorem \texorpdfstring{$\ref{thm: multiplicity one}$}{Lg} when \texorpdfstring{$p=3$}{Lg}}\label{sec: proof when p=3}
In this section let $k=\BF_{3^f}$ for some integer $f$. The character table of $\RG_2(k)$ is given in \cite{En}, which will be used to prove Theorem \ref{thm: multiplicity one}. 

\begin{lem}
The following is a complete set of representatives of $j\in J$ (up to $J$-conjugacy) of the form $j=gz$ with $z\in Z$ and $g\in \SL_2(k)$ such that $g$ is not conjugate to an element of the form $\bpm x & \kappa y\\ y &x \epm $, $y\ne 0:$
\begin{enumerate}
    \item $1;  \quad \bx_{3\alpha+2\beta}(1);  \quad \bx_{2\alpha+\beta}(r_3),r_3\in k^\times;\quad
    \bx_{2\alpha+\beta}(r_3)\bx_{3\alpha+2\beta}(1),r_3\ne 0; $
    \item $\bx_\beta(b)\bx_{2\alpha+\beta}(r_3), \bx_\beta(b)\bx_{2\alpha+\beta}(r_3)\bx_{3\alpha+\beta}(r_4), b\in \wpair{1,\kappa}, r_3\in k, r_4\in k^{\times}/\wpair{\pm 1};$
    \item $ h(-1,-1)\bx_\beta(b)\bx_{2\alpha+\beta}(r_3), r_3\in k, b\in \wpair{1,\kappa};$
    \item $ h(x,x^{-1})\bx_{2\alpha+\beta}(r_3), x\in k^{\times}-\wpair{\pm 1}, r_3\in k.$
\end{enumerate}
\end{lem}
\begin{proof}
The proof of this lemma is similar to the proof of Lemma \ref{lem: conjugacy class of J when p>3}. One difference is that here we have $3=0$ in $k$ and thus \eqref{eq: a conjugation equation} is not valid. Hence,  $\bx_{2\alpha+\beta}(r_3)$ and $\bx_{2\alpha+\beta}(r_3)\bx_{3\alpha+2\beta}(1)$ are no longer in the same $J$-conjugacy class. On the other hand, if $r_3\ne 0,r_4\ne 0$, we have 
$$w_\beta \bx_\beta(-r_5/t_4)(0,0,r_3,r_4,r_5) (w_\beta \bx_\beta(-r_5/r_4))^{-1}=(0,0,r_3,0,r_4).$$
Thus any element of the form $(0,0,r_3,r_4,r_5)$ is $J$-conjugate to an element of the form $(0,0,r_3,0,r_4)$. The other parts of the proof is exactly the same as that of Lemma \ref{lem: conjugacy class of J when p>3}. We omit the details here. 
\end{proof}

The conjugacy classes in $J$ and the character of $I(\chi)\otimes \omega_\psi$ are given in Table \ref{table: conjugacy class of J when p=3}. 
\begin{table}
\begin{align*}
\begin{array}{|c|c|c|c|c|}
\hline
\textrm{Representative } t & C_J(t) & |J(t)| &  No.&  \Ch_{I(\chi)\otimes\omega_\psi} \\
\hline
1 & J &  1 & 1 & q(q+1)  \\
\hline
\bx_{3\alpha+2\beta}(1) & N_{\SL_2}\ltimes V & q^2-1 & 1 & q(q+1)\\
\hline
\bx_{2\alpha+\beta}(r_3),r_3\ne 0 & \SL_2(k)\ltimes V & 1 & q-1 & q(q+1)\psi(r_3)\\
\hline
\bx_{2\alpha+\beta}(r_3)\bx_{3\alpha+2\beta}(1),r_3\ne 0&N_{\SL_2}\ltimes V&q^2-1&q-1&q(q+1)\psi(r_3)\\
\hline
 \bx_\beta(1)\bx_{2\alpha+\beta}(r_3), r_3\in k & \mu_2 \pair{ U_\beta, U_{\alpha+\beta}, U_{2\alpha+\beta}, U_{3\alpha+2\beta}} & \frac{q^2-1}{2}q^2 & q & \sqrt{\epsilon_0 q}\psi(r_3)\\
 \hline
 \bx_{\beta}(\kappa)\bx_{2\alpha+\beta}(r_3), r_3\in k& \mu_2 \pair{ U_\beta, U_{\alpha+\beta}, U_{2\alpha+\beta}, U_{3\alpha+2\beta}}  & \frac{q^2-1}{2}q^2  & q & -\sqrt{\epsilon_0 q} \psi(r_3)\\
\hline
 \bx_\beta(1)\bx_{2\alpha+\beta}(r_3)\bx_{3\alpha+\beta}(r_4) \atop r_3\in k, r_4\in k^\times/\pair{\pm 1} & \pair{U_\beta,U_{\alpha+\beta}, U_{2\alpha+\beta}, U_{3\alpha+2\beta} } & (q^2-1)q^2 & \frac{q(q-1)}{2} & \sqrt{\epsilon_0 q} \psi(r_3)\\
 \hline
 \bx_{\beta}(\kappa)\bx_{2\alpha+\beta}(r_3)\bx_{3\alpha+\beta}(r_4), \atop r_3\in k, r_4\in k^\times/\pair{\pm 1} & \pair{U_\beta,U_{\alpha+\beta}, U_{2\alpha+\beta}, U_{3\alpha+2\beta} } & (q^2-1)q^2 & \frac{q(q-1)}{2} & -\sqrt{\epsilon_0 q} \psi(r_3)\\
\hline
 h(-1,-1)\bx_{2\alpha+\beta}(r_3), \atop r_3\in k & \SL_2\ltimes U_{2\alpha+\beta} & q^4 & q & (q+1)\chi(-1)\epsilon_0\psi(r_3) \\
 \hline
h(-1,-1)\bx_\beta(1)\bx_{2\alpha+\beta}(r_3) & \mu_2\ltimes U_\beta\times U_{2\alpha+\beta} & \frac{q^2-1}{2}q^4 & q & \epsilon_0\chi(-1)\psi(r_3)\\
\hline
h(-1,-1)\bx_\beta(\kappa)\bx_{2\alpha+\beta}(r_3) & \mu_2\ltimes U_\beta\times U_{2\alpha+\beta} & \frac{q^2-1}{2}q^4 & q & \epsilon_0\chi(-1)\psi(r_3) \\
\hline
h(x,x^{-1})\bx_{2\alpha+\beta}(r_3), \atop x\ne \pm 1 & A_{\SL_2}\ltimes U_{2\alpha+\beta} & q^5(q+1) & q\frac{q-3}{2} & \epsilon(\chi(x)+\chi(x^{-1}))\psi(r_3)\\
 \hline
\begin{pmatrix}x&y\\ \kappa y&x \end{pmatrix}(0,0,r_3,r_4,r_5),\atop x\ne \pm 1 & * & * & * & 0 \\
\hline
\end{array}
\end{align*}
\caption{Conjugacy class of $J$ when $p=3$}
\label{table: conjugacy class of J when p=3}
\end{table}
Note that in Table \ref{table: conjugacy class of J when p=3}, the element $\bx_{2\alpha+\beta}(r_3)$ is in fact in the center of $J$, see the commutator relation in \cite[p.192]{En}.

 As in $\S$\ref{sec: proof when p>3}, we still let $\fH_i,i=2,3,6$, be the 3 anisotropic torus of $\RG_2(k)$ such that $\fH_2\cong \BZ_{q+1}$, $\fH_3\cong\BZ_{q^2+q+1}$ and $\fH_6\cong \BZ_{q^2-q+1}$. Then given a character $\pi_i$ of $\fH_i$, there is a class function $X_i(\pi_i)$ on $\RG_2(k)$ as in the case of $p>3$. The notation in \cite{En} is different from that of \cite{CR}. If $i=2$, the character $\pi_2$ of $\fH_2$ is determined by two integers $k,l$ and the associated class function is denoted by $\chi_{12}(k,l)$ in \cite{En}. If $i=3,6$, the character $\pi_i$ of $\fH_i$ is determined by a single integer $k$, and the corresponding class functions are denoted by $ \chi_{13}(k)$ and $\chi_{14}(k)$ respectively in \cite{En}. The restrictions of the characters $\chi_{12}(k,l)$, $\chi_{13}(k)$ and $\chi_{14}(k)$ to $J$ can be read out directly from the table in \cite[p.246-247]{En} and are given in Table \ref{table: character table of Xi when p=3}. In Table \ref{table: character table of Xi when p=3}, $\epsilon(k,l)$ is a number depending on $k,l$.
 
 \begin{table}
\begin{align*}
\begin{array}{|c|c|c|c|c|}
\hline
\textrm{Representative } t & \chi_{12}(k,l) &\chi_{13}(k) &  \chi_{14}(k)&  \Ch_{I(\chi)\otimes\omega_\psi} \\
\hline
1  & \frac{(q^2-1)(q^6-1)}{(q+1)^2} & \frac{(q^2-1)(q^6-1)}{q^2+q+1} &\frac{(q^2-1)(q^6-1)}{q^2-q+1} &q(q+1)  \\
\hline
\bx_{3\alpha+2\beta}(1)&-(q-1)(q^2-q+1)&(q-1)(q^2-1) &-(q+1)(q^2-1) &q(q+1)\\
\hline
\bx_{2\alpha+\beta}(r_3),r_3\ne 0&-(q-1)(q^2-q+1)&(q-1)(q^2-1)&-(q+1)(q^2-1)&q(q+1)\psi(r_3)\\
\hline
\bx_{2\alpha+\beta}(r_3)\bx_{3\alpha+2\beta}(1), \atop r_3\ne 0 &2q^2-2q+1&-(q^2+q-1)&-(q^2-q-1)& q(q+1)\psi(r_3)\\
\hline
\bx_\beta(1)&-(q-1)(q^2-q+1)&(q-1)(q^2-1)&-(q+1)(q^2-1)&\sqrt{\epsilon_0 q}\\
\hline
\bx_\beta(1)\bx_{2\alpha+\beta}(r_3), \atop r_3\in k^\times&-(2q-1) &-(q-1)&q+1&\sqrt{\epsilon_0 q}\psi(r_3)\\
 \hline
 \bx_\beta(\kappa)&-(q-1)(q^2-q+1)&(q-1)(q^2-1)&-(q+1)(q^2-1)& -\sqrt{\epsilon_0 q}\\
 \hline
 \bx_\beta(\kappa)\bx_{2\alpha+\beta}(r_3),\atop r_3\in k^\times &-(2q-1)&-(q-1)&q+1 &-\sqrt{\epsilon_0 q} \psi(r_3)\\
\hline
 \bx_\beta(1)(0,0,r_3,r_4,0), \atop r_3\in k, r_4\in k^\times/\pair{\pm 1} &*&*&* &\sqrt{\epsilon_0 q} \psi(r_3)\\
 \hline
 \bx_\beta(\kappa)(0,0,r_3,r_4,0), \atop r_3\in k, r_4\in k^\times/\pair{\pm 1} &* &*&* &-\sqrt{\epsilon_0 q} \psi(r_3)\\
\hline
 h(-1,-1)&(q-1)^2\epsilon(k,l)&0&0& (q+1)\chi(-1)\epsilon_0 \\
\hline
 h(-1,-1)(0,0,r_3,0,0), \atop r_3\in k^\times&-(q-1)\epsilon(k,l)&0&0& (q+1)\chi(-1)\epsilon_0\psi(r_3) \\
 \hline
h(-1,-1)\bx_\beta(1)&-(q-1)\epsilon(k,l)&0&0&\epsilon_0\chi(-1)\\
 \hline
h(-1,-1)\bx_\beta(1)\bx_{2\alpha+\beta}(r_3),\atop r_3\in k^\times &\epsilon(k,l)&0&0 &\epsilon_0\chi(-1)\psi(r_3)\\
\hline
h(-1,-1)\bx_\beta(\kappa)&-(q-1)\epsilon(k,l)&0&0 & \epsilon_0\chi(-1)\\
\hline
h(-1,-1)\bx_\beta(\kappa)\bx_{2\alpha+\beta}(r_3), \atop r_3\in k^\times &\epsilon(k,l)&0&0 & \epsilon_0\chi(-1)\psi(r_3)\\
\hline
h(x,x^{-1})(0,0,r_3,0,0), \atop x\ne \pm 1&0 &0&0&\epsilon(\chi(x)+\chi(x^{-1}))\psi(r_3)\\
 \hline
\begin{pmatrix}x&y\\ \kappa y&x \end{pmatrix}(0,0,r_3,r_4,r_5),\atop x\ne \pm1 &* &*&* & 0 \\
\hline
\end{array}
\end{align*}
\caption{Character table of $X_i(\pi_i)$ when $p=3$}
\label{table: character table of Xi when p=3}
\end{table}
The missing part of the Table \ref{table: character table of Xi when p=3} (10th row and 11th row) is determined as follows. If $r_3=0,r_4\ne 0$, then $$\bx_\beta(r)(0,0,r_3,r_4,0)\sim_{\RG_2(k)} \bx_\beta(1)\bx_{2\alpha+\beta}(1),$$ 
since every element in $k$ has a cubic root, see the calculation in \cite[p.197]{Ch} or the discussion in the previous section. As in the previous section, if $r_3\ne 0$, we have 
$$\bx_\beta(r)(0,0,r_3,r_4,0)\sim_{\RG_2(k)} \bx_\beta(1)(0,0,-1,z,0),$$
with $z=-2-rr_4^2/r_3^3$. Write $z=t+t^{-1}$. If $t=\pm 1$, then we have 
$$\bx_\beta(1)(0,0,r_3,r_4,0)\sim_{\RG_2(k)} \bx_{\beta}(1/2)\bx_{\alpha+\beta}(1)\sim_{\RG_2(k)} \bx_{2\alpha+\beta}(1)\bx_{3\alpha+2\beta}(1),$$
see Footnote \ref{footnote: a conjugation relation} for the first relation. Note that $1/6$ is undefined in the last equation of Footnote \ref{footnote: a conjugation relation} and thus we cannot obtain that $\bx_\beta(1/2)\bx_{\alpha+\beta}(1)\sim_{\RG_2(k)} \bx_{\alpha+\beta}(1)$ here. On the other hand, we have $w_\beta w_\alpha \bx_\beta(-1)\bx_{\alpha+\beta}(1)(w_\beta w_\alpha)^{-1}=\bx_{2\alpha+\beta}(1)\bx_{3\alpha+2\beta}(-1)$. By considering a conjugation of the torus, we then get $\bx_{\beta}(1/2)\bx_{\alpha+\beta}(1)\sim_{\RG_2(k)} \bx_{2\alpha+\beta}(1)\bx_{3\alpha+2\beta}(1) $.

If $t\ne \pm 1$, using the description in \cite{Ch} and the fact that any element in $\BF_q$ and $ \BF_{q^2}$ has a cubic root, one can check that $\bx_\beta(1)(0,0,r_3,r_4,0)\sim_{\RG_2(k)} \bx_\beta(1)\bx_{2\alpha+\beta}(1)$ if $t\in k^{\times}-\wpair{\pm 1}$, and $$\bx_\beta(1)(0,0,r_3,r_4,0)\sim_{\RG_2(k)} \bx_\beta(1)\bx_{2\alpha+\beta}(\kappa)$$ if $t\in \BF_{q^2}-\BF_q$. Thus we obtain Table \ref{table: missing part of character table when p=3} following the table in \cite[p.246-247]{En}. Here, $t(r,r_3,r_4)$ is the number $t$ such that $t+t^{-1}=-2-rr_4^2/r_3^3$.
\begin{table}
\begin{align*}
\begin{array}{|c|c|c|c|}
\hline
\textrm{representative} & \chi_{12}(k,l) & \chi_{13}(k) & \chi_{14}(k)\\
\hline
\bx_\beta(1)\bx_{3\alpha+\beta}(r_4)\atop r_4\ne 0&-(2q-1)&-(q-1)&q+1\\
\hline
\bx_\beta(1)(0,0,r_3,r_4,0)\atop t(1,r_3,r_4)\in \wpair{\pm1} & 2q^2-2q+1&-(q^2+q-1)& -(q^2-q-1)\\
\hline
\bx_\beta(1)(0,0,r_3,r_4,0)\atop t(1,r_3,r_4)\notin\wpair{\pm 1} & -(2q-1)& -(q-1)& q+1\\
\hline
\bx_\beta(\kappa)\bx_{3\alpha+\beta}(r_4),\atop r_4\ne 0 &-(2q-1)&-(q-1)&q+1\\
\hline
\bx_\beta(\kappa)(0,0,r_3,r_4,0)\atop t(\kappa,r_3,r_4)\in \wpair{\pm 1} & 2q^2-2q+1&-(q^2+q-1)& -(q^2-q-1)\\
\hline
\bx_\beta(\kappa)(0,0,r_3,r_4,0)\atop t(\kappa,r_3,r_4)\notin \wpair{\pm 1} & -(2q-1)& -(q-1)& q+1\\
\hline
\end{array}
\end{align*}
\caption{Missing part of Table \ref{table: character table of Xi when p=3}}
\label{table: missing part of character table when p=3}
\end{table}

\begin{lem}\label{lem: pair is one when p=3}
Let $\Pi$ be $\chi_{12}(k,l),\chi_{13}(k)$ or $\chi_{14}(k)$. Then we have 
$$\pair{\Pi,I(\chi)\otimes\omega_\psi}=1.$$
\end{lem}

We have 
\begin{equation}\label{eq: the pair when p=3}
    |J|\pair{\Pi,I(\chi)\otimes\omega_\psi}=\sum_{t} |J(t)| \ov{\Ch_\Pi}(t)\Ch_{I(\chi)\otimes\omega_\psi}(t),
\end{equation}
where $t$ runs over a complete set of representatives of conjugacy classes of $J$ and $|J(t)|$ is the number of elements in the conjugacy class $J(t)$.
Before proving Lemma \ref{lem: pair is one when p=3}, we first record the following result.

\begin{lem}\label{lem: preparation for the computation of the pair when p=3}
Let $\Pi$ be $\chi_{12}(k,l),\chi_{13}(k)$ or $\chi_{14}(k)$.
\begin{enumerate}
    \item The contribution of conjugacy classes of the form $h(-1,-1)u$, with $ u$ in the unipotent, to $\eqref{eq: the pair when p=3}$ is zero.
    \item The contribution of conjugacy classes of the form $\bx_\beta(1)\bx_{2\alpha+\beta}(r_3),r_3\in k$, and the contribution of conjugacy classes of the form  $\bx_\beta(\kappa)\bx_{2\alpha+\beta}(r_3),r_3\in k$, to $\eqref{eq: the pair when p=3}$ are cancelled out.
\end{enumerate}
\end{lem}

The proof of Lemma \ref{lem: preparation for the computation of the pair when p=3} is the same as that of Lemma \ref{lem: preparation for the computation of the pair} and we omit the details.

\begin{proof}[Proof of Lemma $\ref{lem: pair is one when p=3}$]
This lemma can be checked case by case and we only give the details when $\Pi=\chi_{12}(k,l)$ and omit the details of the other two cases. We suppose that $\Pi=\chi_{12}(k,l)$. By Tables \ref{table: conjugacy class of J when p=3}, \ref{table: character table of Xi when p=3} and \ref{table: missing part of character table when p=3}, we have
\begin{align*}
    &\ |J|\pair{\Pi,I(\chi)\otimes \omega_\psi}\\
    =&\ \frac{(q^2-1)(q^6-1)}{(q+1)^2}q(q+1)-(q^2-1)(q-1)(q^2-q+1)q(q+1)\\
    -&\ (q-1)(q^2-q+1)q(q+1)\left(\sum_{r_3\in k^\times}\psi(r_3)\right)\\
    +&\ (q^2-1)(2q^2-2q+1)q(q+1)\left(\sum_{r_3\in k^\times}\psi(r_3)\right)\\
    +&\ (q^2-1)q^2\sqrt{\epsilon_0 q}((2q^2-2q+1)(D_1^0-D_\kappa^0))\\
    +&\ (q^2-1)q^2\sqrt{\epsilon_0 q}(-(2q-1)(D_1^1+D_1^2-D_\kappa^1-D_\kappa^2)),
\end{align*}
where 
\begin{align}\label{eq: defn of D}
\begin{split}
D_r^0&=\sum_{r_3\in k^\times, r_4\in k^{\times}/\wpair{\pm 1}, t(r,r_3,r_4)\in \wpair{\pm 1}} \psi(r_3)   ,\\
D_r^1&=\sum_{r_3\in k^\times, r_4\in k^{\times}/\wpair{\pm 1}, t(r,r_3,r_4)\in k^\times- \wpair{\pm 1}} \psi(r_3),\\
D_r^2&=\sum_{r_3\in k^\times, r_4\in k^{\times}/\wpair{\pm 1}, t(r,r_3,r_4)\in \BF_{q^2}-\BF_q} \psi(r_3),
\end{split}
\end{align}
for $r=1,\kappa$. The computation of $D_1^i-D_\kappa^i$ is given in Appendix \ref{sec: computation of D}. By Lemma \ref{lem: computation of D}, we have $D_1^0-D_\kappa^0=\epsilon_0 \sqrt{\epsilon_0 q}$ and $D_1^1+D_1^2-D_\kappa^1-D_\kappa^2=-\epsilon_0 \sqrt{\epsilon_0 q}$. Plugging these formulas into the computation of $|J|\pair{\Pi,I(\chi)\otimes\omega_\psi}$, we can get 
$$|J|\pair{\Pi,I(\chi)\otimes\omega_\psi}=q^6(q^2-1).$$
Thus we have $\pair{\Pi,I(\chi)\otimes\omega_\psi}=1$.
\end{proof}
We can now start the proof of Theorem \ref{thm: multiplicity one} in the case $p=3$.

\begin{proof}[\textbf{Proof of Theorem $\ref{thm: multiplicity one}$ when $p=3$}]

Irreducible representations of $\RG_2(k)$ when $k=\BF_{3^f}$ are classified in \cite{En}. Using the notation of \cite{En}, there are 12 isolated irreducible representations $\theta_i$ $0\le i\le 11$ and 15 families of irreducible representations $\theta_{12}(k)$, $\chi_i(k), 1\le i\le 11$, $\chi_{12}(k,l)$, $\chi_{13}(k)$ and $\chi_{14}(k)$, where $k,l$ are integers. From the definitions given in \cite[Section 5]{En}, the representations $\chi_{i}(k), 1\le i\le 11$, are not cuspidal. By Lemma \ref{lem: pair is one when p=3}, we only need to consider cuspidal representations among $\theta_i, 0\le i\le 11$, and $\theta_{12}(k)$. From the definitions of $\theta_i$ in \cite[Section 5]{En}, one can check that $\theta_0, \theta_1, \theta_2, \theta_3, \theta_4, \theta_6, \theta_7, \theta_8, \theta_9$, are components of parabolic induced representations and thus cannot be cuspidal. Precisely, first, from the definitions in \cite[p.204]{En}, one has that $\theta_1+\theta_2=\mu_5=\Ind_{P'}^{\RG_2(k)}(\chi_1(0))-\theta_0-\theta_3$, and hence $\theta_i,0\le i\le 3$, are components of $\Ind_{P'}^{\RG_2(k)}(\chi_1(0))$. Here $\chi_1(0)$ is a character of $P'$ and $\mu_5$ is an auxiliary representation. Note that the notation in \cite{En} is a little bit different from ours. In particular, the group $P$ in \cite{En} is our $P'$. Moreover, one has $\theta_4=\Ind_{P'}^{\RG_2(k)}(\chi_3(0))-\Ind_{P'}^{\RG_2(k)}(\chi_1(0))+\theta_3$, see \cite[p.204]{En}. Since $\theta_0+\theta_1+\theta_2+\theta_3=\Ind_{P'}^{\RG_2(k)}(\chi_1(0))$, we can get, $ \theta_0+\theta_1+\theta_2+\theta_4=\Ind_{P'}^{\RG_2(k)}(\chi_3(0))$. Hence $\theta_4$ is a component of $\Ind_{P'}^{\RG_2(k)}(\chi_3(0))$ and thus not cuspidal. Here $\chi_3(0)$ is a character on $P'$. Furthermore, from the description in \cite[p.201]{En}, we have $\theta_6+\theta_9=\Ind_{P'}^{\RG_2(k)}(\chi_1(\frac{1}{2}(q-1)))$ and $\theta_7+\theta_8=\Ind_{P'}^{\RG_2(k)}(\chi_3(\frac{1}{2}(q-1)))$. Thus $\theta_6,\theta_7,\theta_8,\theta_9$ are not cuspidal either. Consequently, it suffices to consider the cases when $\Pi=\theta_5,\theta_{10},\theta_{11},\theta_{12}(k)$.

 \begin{table}
\begin{align*}
\begin{array}{|c|c|c|c|c|}
\hline
\textrm{Representative } t & \theta_5 & \theta_{10} &  \theta_{11}& \theta_{12}(k) \\
\hline
1  & q^6 & \frac{1}{6}q(q-1)^2(q^2-q+1) &\frac{1}{2}q(q-1)(q^3-1) &\frac{1}{3}q(q^2-1)^2 \\
\hline
\bx_{3\alpha+2\beta}(1)&0&\frac{1}{6}q(q-1)(2q-1) &-\frac{1}{2}q(q-1) &-\frac{1}{3}q(q^2-1)\\
\hline
\bx_{2\alpha+\beta}(r_3),\atop r_3\ne 0&0&\frac{1}{6}q(q-1)(2q-1)&-\frac{1}{2}q(q-1) &-\frac{1}{3}q(q^2-1)\\
\hline
\bx_{2\alpha+\beta}(r_3)\bx_{3\alpha+2\beta}(1), \atop r_3\ne 0 &0&-\frac{1}{6}q(3q-1)&-\frac{1}{2}q(q-1)& \frac{1}{3}q\\
\hline
\bx_\beta(1)&0& \frac{1}{6}q(q-1)(2q-1)& -\frac{1}{2}q(q-1)& -\frac{1}{3}q(q^2-1)\\
\hline
\bx_\beta(1)\bx_{2\alpha+\beta}(r_3), \atop r_3\in k^{\times,2}&0 &\frac{1}{6}q(q+1)&-\frac{1}{2}q(q-1) &\frac{1}{3}q(q+1)\\
 \hline
 \bx_\beta(1)\bx_{2\alpha+\beta}(r_3), \atop r_3\in \kappa k^{\times,2}&0&-\frac{1}{6}q(q-1)&\frac{1}{2}q(q+1)&-\frac{1}{3}q(q-1)\\
 \hline
 \bx_\beta(\kappa)&0&\frac{1}{6}q(q-1)(2q-1)&-\frac{1}{2}q(q-1)&-\frac{1}{3}q(q^2-1)\\
 \hline
 \bx_\beta(\kappa)\bx_{2\alpha+\beta}(r_3),\atop r_3\in k^{\times,2} &0&- \frac{1}{6}q(q-1)& \frac{1}{2}q(q+1)&-\frac{1}{3}q(q-1)\\
\hline
\bx_\beta(\kappa)\bx_{2\alpha+\beta}(r_3),\atop r_3\in \kappa k^{\times,2}&0&\frac{1}{6}q(q+1)&-\frac{1}{2}q(q-1)& \frac{1}{3}q(q+1)\\
\hline
\bx_\beta(1)\bx_{3\alpha+\beta}(r_4), \atop r_4\ne 0 &0&\frac{1}{6}q(q+1)&-\frac{1}{2}q(q-1)&\frac{1}{3}q(q+1)\\
\hline
 \bx_\beta(1)(0,0,r_3,r_4,0), r_3\ne 0,\atop  r_4\in k^\times/\wpair{\pm 1}, t(1,r_3,r_4)\in \wpair{\pm 1}&0 &-\frac{1}{6}q(3q-1)&-\frac{1}{2}q(q-1)& -\frac{1}{3}q\\
 \hline
 \bx_\beta(1)(0,0,r_3,r_4,0), r_3\ne 0,\atop  r_4\in k^\times/\wpair{\pm 1}, t(1,r_3,r_4)\in k^{\times}-\wpair{\pm 1}&0 &\frac{1}{6}q(q+1)&-\frac{1}{2}q(q-1)&\frac{1}{3}q(q+1)\\
  \hline
 \bx_\beta(1)(0,0,r_3,r_4,0), r_3\ne 0,\atop  r_4\in k^\times/\wpair{\pm 1}, t(1,r_3,r_4)\in \BF_{q^2}-\BF_{q}&0 &-\frac{1}{6}q(q-1)&\frac{1}{2}q(q+1)&-\frac{1}{3}q(q-1)\\
 \hline
 \bx_\beta(\kappa)\bx_{3\alpha+\beta}(r_4) \atop r_4\ne 0&0&\frac{1}{6}q(q+1)&-\frac{1}{2}q(q-1)&\frac{1}{3}q(q+1)\\
  \hline
 \bx_\beta(\kappa)(0,0,r_3,r_4,0), r_3\ne 0,\atop  r_4\in k^\times/\wpair{\pm 1}, t(\kappa,r_3,r_4)\in k^{\times}-\wpair{\pm 1}&0 &\frac{1}{6}q(q+1)&-\frac{1}{2}q(q-1)&\frac{1}{3}q(q+1)\\
  \hline
 \bx_\beta(\kappa)(0,0,r_3,r_4,0), r_3\ne 0,\atop  r_4\in k^\times/\wpair{\pm 1}, t(\kappa,r_3,r_4)\in \BF_{q^2}-\BF_{q}&0 &-\frac{1}{6}q(q-1)&\frac{1}{2}q(q+1)&-\frac{1}{3}q(q-1)\\
 \hline
  \bx_\beta(\kappa)(0,0,r_3,r_4,0), r_3\ne 0,\atop  r_4\in k^\times/\wpair{\pm 1}, t(\kappa,r_3,r_4)\in \wpair{\pm 1}&0 &-\frac{1}{6}q(3q-1)&-\frac{1}{2}q(q-1)& -\frac{1}{3}q\\
\hline
 h(-1,-1)&q^2&-\frac{1}{2}(q-1)^2&-\frac{1}{2}(q-1)^2& 0\\
\hline
 h(-1,-1)(0,0,r_3,0,0), \atop r_3\in k^\times&0&\frac{1}{2}(q-1)&\frac{1}{2}(q-1)& 0 \\
 \hline
h(-1,-1)\bx_\beta(1)&0&\frac{1}{2}(q-1)&\frac{1}{2}(q-1)&0\\
 \hline
h(-1,-1)\bx_\beta(1)\bx_{2\alpha+\beta}(r_3),\atop r_3\in k^{\times,2} &0&-\frac{1}{2}(q+1)&\frac{1}{2}(q-1) &0\\
 \hline
h(-1,-1)\bx_\beta(1)\bx_{2\alpha+\beta}(r_3),\atop r_3\in \kappa k^{\times,2} &0&\frac{1}{2}(q-1)&-\frac{1}{2}(q+1) &0\\
\hline
h(-1,-1)\bx_\beta(\kappa)&0&\frac{1}{2}(q-1)&\frac{1}{2}(q-1)&0\\
\hline
h(-1,-1)\bx_\beta(\kappa)\bx_{2\alpha+\beta}(r_3), \atop r_3\in k^\times &0&\frac{1}{2}(q-1)&-\frac{1}{2}(q+1) & 0\\
\hline
h(-1,-1)\bx_\beta(\kappa)\bx_{2\alpha+\beta}(r_3), \atop r_3\in \kappa k^\times &0&-\frac{1}{2}(q+1)&\frac{1}{2}(q-1) & 0\\
\hline
h(x,x^{-1}), \atop x\ne \pm 1&q &0&0&0\\
\hline
h(x,x^{-1})\bx_{2\alpha+\beta}(r_3), \atop x\ne \pm 1, r_3\ne 0&0 &0&0&0\\
 \hline
\begin{pmatrix}x&y\\ \kappa y&x \end{pmatrix}(0,0,r_3,r_4,r_5),\atop x\ne \pm1 &* &*&* & * \\
\hline
\end{array}
\end{align*}
\caption{Character table of $\theta_5,\theta_{10},\theta_{11},\theta_{12}(k)$}
\label{table: character table of theta when p=3}
\end{table}

Following \cite{En}, the character table of $\theta_5,\theta_{10},\theta_{11},\theta_{12}(k)$, is given in Table \ref{table: character table of theta when p=3}. Recall that $U$ is the maximal unipotent subgroup of $\RG_2(k)$. From the character table, we see that for $u\in U$, we have $\theta_5(u)\ne 0$ if and only if $u=1$. In particular, we have 
$$\sum_{u\in U}\theta_5(u)=\theta_5(1)=q^6.$$
This implies that $\theta_5 $ is not a cupidal character.\footnote{Recall that an irreducible character $\theta$ of a reductive group $H$ over a finite field is cuspidal if and only if for any proper parabolic subgroup $Q=M_QU_Q$ with Levi $M_Q$ and unipotent $U_Q$, one has $\sum_{u \in U_Q}\theta(uh)=0$ for all $h\in H$, see \cite[Corollary 9.1.2]{Ca} for example. In fact, from the character table \ref{table: character table of theta when p=3}, one can check that
$$\pair{\theta_5,I(\chi)\otimes\omega_\psi}=\left\{\begin{array}{lll} 1, & \textrm{ if } \epsilon \chi\ne 1, \\ 2, & \textrm{ if } \epsilon \chi=1. \end{array}\right.$$ Thus $\theta_5$ indeed does not satisfy the conclusion of Theorem \ref{thm: multiplicity one} if $\chi=\epsilon^{-1}=\epsilon$.} Thus it suffices to show that $\pair{\Pi, I(\chi)\otimes \omega_\psi}\le 1$, when $\Pi=\theta_{10}, \theta_{11}, \theta_{12}(k)$. We now compute $\pair{\theta_{10},I(\chi)\otimes\omega_\psi}$. Similar to Lemma \ref{lem: preparation for the computation of the pair when p=3}, the contribution of terms of the form $h(-1,-1)u$ to $\pair{\theta_{10},I(\chi)\otimes \omega_\psi}$ is zero. Thus we get 
\begin{align*}
    &\ |J|\pair{\theta_{10},I(\chi)\otimes\omega_\psi}\\
    =&\ \frac{1}{6}q(q-1)^2(q^2-q+1)q(q+1)+(q^2-1)q(q+1)\frac{1}{6}q(q-1)(2q-1)\\
    +&\ q(q+1)\frac{1}{6}q(q-1)(2q-1)\left(\sum_{r_3\in k^\times}\psi(r_3)\right)+ (q^2-1)q(q+1)(-\frac{1}{6})q(3q-1)\left(\sum_{r_3\in k^\times}\psi(r_3)\right)\\
    +&\ \frac{q^2-1}{2}q^2\sqrt{\epsilon_0 q}\left(\frac{1}{6}q(q+1)(A_1(1)-A_\kappa(1)) +\frac{1}{6}q(q-1)(A_1(1)-A_\kappa(1))\right)\\
   +&\ q^2(q^2-1)\sqrt{\epsilon_0 q}\left(-\frac{1}{6}q(3q-1)(D_1^0-D_\kappa^0)+\frac{1}{6}q(q+1)(D_1^1-D_\kappa^1)-\frac{1}{6}q(q-1)(D_1^2-D_\kappa^2) \right).
\end{align*}
Plugging the formula of $A_1(1)-A_\kappa(1)$ from Lemma \ref{lem: gauss sum} and and the formulas of  $D_1^i-D_\kappa^i$ for $i=0,1,2$, from Lemma \ref{lem: computation of D}, into the above equation, a simple calculation shows that $\pair{\theta_{10},I(\chi)\otimes\omega_\psi} =0$. Similarly, one can check that $\pair{ \theta_{11},I(\chi)\otimes \omega_\psi}=0 $ and $\pair{\theta_{12}(k),I(\chi)\otimes\omega_\psi}=0$. We omit the details.
\end{proof}

\section{Gamma factors for \texorpdfstring{$\RG_2(k)\times \GL_1(k)$}{Lg}}\label{sec: G2 GL1}
\subsection{Generic representations and Bessel functions}\label{subsec:Bessel}

Recall that $U$ is the maximal unipotent subgroup of $\RG_2(k)$. Let $\psi_U $ be the character of $U$ defined by 
$$\psi_U(\bx_\alpha(x)\bx_\beta(y)u')=\psi(x+y), x,y\in k, u'\in [U,U].$$
We will write $\psi_U$ as $\psi$ by abuse of notation. An irreducible representation $\Pi$ of $\RG_2(k)$ is called $\psi$-generic if 
$$\Hom_U(\Pi,\psi)\ne 0.$$
It is well-known that $\dim \Hom_U(\Pi,\psi)\le 1$. 

\begin{rmk}\label{rmk: dependence on psi}
{\rm A character $\psi'$ of $U$ is called generic if $\psi'|_{U_a}$ is nontrivial for $a=\alpha,\beta$. There is only one $T$-conjugacy class of generic characters of $U$. Thus if $\Pi$ is $\psi$-generic, then it is generic with respect to any generic character of $U$.}
\end{rmk}

Let $\Pi$ be an irreducible generic representation of $\RG_2(k)$. We fix a nonzero element $l\in \Hom_U(\Pi,\psi)$. For a vector $v$ in the space of $\Pi$, we consider the function 
$$W_v(g):=l(\Pi(g)v).$$
Then the space $\CW(\Pi,\psi):=\wpair{W_v: v\in \Pi}$ is called the $\psi$-Whittaker model of $\Pi$.

Let $\Pi(U,\psi)$ be the subspace of $\Pi$ generated by elements of form $\Pi(u)v-\psi(u)v$ for $u\in U,v\in \Pi$. Let $\Pi_{U,\psi}=\Pi/\Pi(U,\psi)$ be the twisted Jacquet module. By Jacquet-Langlands Lemma \cite[Lemma 2.33]{BZ}, an element $v\in \Pi(U,\psi)$ if and only if $\sum_{u\in U}\psi^{-1}(u)\Pi(u)v=0$. Note that for an irreducible generic representation $\Pi$, we have $\dim \Pi_{U,\psi}=1$.  For a vector $v\in \Pi,v\notin \Pi(U,\psi)$, we consider the vector
$$v_0=\frac{1}{|U|}\sum_{u\in U}\psi^{-1}(u)\Pi(u)v,$$
where $|U|$ is the number of elements in $U$. From the choice of $v$ and Jacquent-Langlands Lemma \cite[Lemma 2.33]{BZ}, we have $v_0\ne 0$. On the other hand, we have 
\begin{align}\label{eq: whittaker vector}
    \Pi(u)v_0&=\psi(u)v_0, \forall u\in U.
\end{align}
A vector which satisfies the above condition is called a {\it Whittaker vector}. Let $\pair{~,~}$ be a nontrivial $\RG_2(k)$-invariant bilinear form $\Pi\times \wt \Pi\ra \BC$, where $\wt\Pi$ is the dual representation of $\Pi$. Let $v_0$ be a Whittaker vector of $\Pi$. Then $\tilde v\mapsto \pair{v_0,\tilde v}$ defines a nonzero element in $\Hom_U(\wt \Pi,\psi^{-1})$. Conversely, a nonzero element in $\Hom_U(\wt\Pi,\psi^{-1})$ can be viewed as a Whittaker vector of $\Pi$ via the natural isomorphism $\wt{\wt\Pi}\cong \Pi$. By the uniqueness of Whittaker model, the Whittaker vectors are unique up to scalar. 

Let $\CB_\Pi\in \CW(\Pi,\psi)$ be the Whittaker function associated with a Whittaker vector, normalized by $\CB_\Pi(1)=1$. By the above discussion, the function $\CB_\Pi$ is unique. The function $\CB_\Pi$ is called the Bessel function of $\Pi$. 

\begin{lem}\label{lem: basic properties of Bessel function}
We have 
 $$\CB_\Pi(u_1gu_2)=\psi(u_1u_2)\CB_\Pi(g), \forall u_1,u_2\in U, g\in \RG_2(k).$$
\end{lem}

\begin{proof}
This a direct consequence of the definition of $\CB_\Pi$. 
\end{proof}

\begin{lem}\label{lem: further properties of Bessel functions}
\begin{enumerate}
\item Let $t=h(a,b)\in T$. If $\CB_\Pi(t)\ne 0$, then $a=b=1$.
\item If $r\ne 0$, then $\CB_\Pi(h(a,1)\bx_{-\beta}(r))=0$ for all $a\in k^\times$.
\end{enumerate}
\end{lem}
\begin{proof}
(1) Let $a$ be the simple root $\alpha$ or $\beta$. First, we have 
$$t\bx_a(r)=\bx_a(a(t)r)t, \forall r\in k.$$
Then, by Lemma \ref{lem: basic properties of Bessel function}, we have 
$$\CB_\Pi(t)\psi(r)=\psi(a(t)r)\CB_\Pi(t).$$
Thus if $B_\Pi(t)\ne 0$, we have $\psi(r)=\psi(a(t)r)$ for all $r\in k$. Since $\psi$ is a nontrivial character, we must have $a(t)=1$. Since $\alpha(t)=b,\beta(t)=a/b$, we get $a=b=1$ if $B_\Pi(t)\ne 0$.

(2) Take $s\in k$. We have $\bx_{-\beta}(r)=w_\beta \bx_\beta(-r)w_\beta^{-1}$ and $\bx_{\alpha+\beta}(s)=w_\beta \bx_\alpha(s)w_\beta^{-1}$. Thus from the commutator relations, we have 
\begin{align*}
    \bx_{-\beta}(r)\bx_{\alpha+\beta}(s)&=w_\beta \bx_\beta(-r)\bx_\alpha(s)w_\beta^{-1}\\
    &=w_\beta u_1 \bx_{\alpha+\beta}(-rs)\bx_\alpha(s)\bx_\beta(-r)w_\beta^{-1}\\
    &=u_2\bx_\alpha(-rs) \bx_{\alpha+\beta}(s)\bx_{-\beta}(r),
\end{align*}
where $u_2=w_\beta u_1w_\beta^{-1}\in [U,U]$. Thus, we get 
$$h(a,1)\bx_{-\beta}(r)\bx_{\alpha+\beta}(s)=u_3\bx_\alpha(-rs)\bx_{\alpha+\beta}(as)h(a,1)\bx_{-\beta}(r),$$
where $u_3=h(a,1)u_2h(a^{-1},1)$. Note that $\psi(\bx_{\alpha+\beta}(s))=1$ and 
$\psi(u_2\bx_\alpha(-rs) \bx_{\alpha+\beta}(as))=\psi(-rs)$. By Lemma \ref{lem: basic properties of Bessel function}, we get 
$$ \CB_\Pi(h(a,1)\bx_{-\beta}(r))=\psi(-rs)\CB_\Pi(h(a,1)\bx_{-\beta}(r)),\forall s\in k.$$ Thus if $\CB_\Pi(h(a,1)\bx_{-\beta}(r))\ne 0$, we have $\psi(-rs)=1$ for all $s\in k$. Since $\psi$ is nontrivial, we then get $r=0$.
\end{proof}

\subsection{Ginzburg's local zeta integral}
 
Let $\Pi$ be an irreducible generic representation of $\RG_2(k)$ and let $\chi$ be a character of $k^\times$. For $W\in \CW(\Pi,\psi),f\in I(\chi),\phi\in \CS(k)$, we consider the following sum
\begin{align}\label{eq: defn of ginzburg's local zeta integral}
Z(W,\phi,f)&=\sum_{g\in N_{\SL_2}\backslash \SL_2(k)}\sum_{x,y\in k}W(\bx_{-\beta}(y)\bx_{-(\alpha+\beta)}(x)j(g))(\omega_{\psi^{-1}}(g)\phi)(x)f(g)\\
&=\sum_{g\in N_{\SL_2}\backslash \SL_2(k)}\sum_{x,y\in k}W(j(\bx_{3\alpha+\beta}(y)\bx_\alpha(x)g))(\omega_{\psi^{-1}}(g)\phi)(x)f(g),\nonumber
\end{align}
where we embed $\SL_2(k)$ in $\RG_2(k)$ by embedding it in the Levi subgroup of $P$, and 
$j(g)=w_\beta w_\alpha gw_\alpha^{-1}w_\beta^{-1}$ for $g\in \RG_2(k)$. One can easily check that the above sum on the quotient $N_{\SL_2}\backslash \SL_2(k)$ is well-defined using the commutator relations of $\RG_2$. The above sum is the finite fields analogue of Ginzburg's local zeta integral \cite{Gi}.
\begin{lem}\label{lem: nonvanishing of Ginzburg's local zeta integral}
Let $ \delta_0\in \CS(k)$ be the function $\delta_0(x)=0$ if $x\ne 0$ and $\delta_0(0)=1$. Let $f_0\in I(\chi)$ be the function such that $\supp(f_0)\subset N_{\SL_2}A_{\SL_2}$ and $f_0(1)=1$. Then 
$$Z(\CB_\Pi,\delta_0,f_0)=1.$$
\end{lem}
\begin{proof}
We have $\SL_2(k)=N_{\SL_2}A_{\SL_2}\coprod N_{\SL_2}A_{\SL_2}w^1N_{\SL_2}$, where $w^1=\bpm &1\\-1&\epm$. Since $f_0$ is zero on $ N_{\SL_2}A_{\SL_2}w^1N_{\SL_2}$, we have 
$$Z(\CB_\Pi,\delta_0,f_0)=\sum_{a\in k^\times}\sum_{x,y\in k}\CB_\Pi(\bx_{-\beta}(y)\bx_{-(\alpha+\beta)}(x)j(t(a)))\omega_{\psi^{-1}}(t(a))\delta_0(x) \chi(a),$$
where $t(a)=\diag(a,a^{-1})\in \SL_2(k)$. Note that $j(t(a))=h(a,1)$. On the other hand, we have $\omega_{\psi^{-1}}(t(a))\delta_0(x)=\epsilon(a)\delta_0(ax)$, which is zero if $x\ne 0$ since $a\ne 0$. Thus we have 
$$Z(\CB_\Pi,\delta_0,f_0)=\sum_{a\in k^\times}\sum_{y\in k}\CB_\Pi(\bx_{-\beta}(y)h(a,1))\epsilon(a)\chi(a).$$
By Lemma \ref{lem: further properties of Bessel functions} (2) and (1), we have 
\begin{align*}
    Z(\CB_\Pi,\delta_0,f_0)&=\sum_{a\in k^\times}\CB_\Pi(h(a,1))\epsilon(a)\chi(a)\\
    &=\CB_\Pi(1)\\
    &=1.
\end{align*}
This completes the proof of the Lemma.
\end{proof}

\begin{lem}\label{lem: invariance property of Ginzburg's local zeta integral}
The trilinear form $(W,\phi,f)\mapsto Z(W,\phi,f)$ on $\CW(\Pi,\psi)\times \omega_{\psi^{-1}}\times I(\chi)$ satisfies the property
\begin{equation}\label{eq: invariance} Z((\Pi(j(h)))W,\omega_{\psi^{-1}}(\ov{\pr}(h))\phi,r(\ov{\pr}(h))f)=Z(W,\phi,f), \forall h\in J,
\end{equation}
where $r$ denotes the right translation action, and for $(g,h) \in \SL_2(k)\ltimes \sH$, $r(g,h)f:=r(g)f$. Recall that $\ov{\pr}$ is the projection map $J\ra \SL_2(k)\ltimes \sH$ in $\S$$\ref{subsec: G2}$.
\end{lem}

\begin{proof}
Note that for $h\in \SL_2(k)$, Eq.\eqref{eq: invariance} follows from a simple changing of variables. 
Hence, we only need to check formula \eqref{eq: invariance} when $h\in V$. Suppose that $h=(s_1,s_2,s_3,s_4,s_5)$. Since 
$\ov \pr(h)=(1,(s_1,s_2, s_3-s_1s_2))\in \SL_2(k)\ltimes \sH$,
$r(\ov \pr(h))f=f$. 
For $g=\begin{pmatrix}a&b\\ c&d \end{pmatrix}$, we have 
$$h':=ghg^{-1}=(s_1',s_2',s_3',s_4',s_5'), $$
with $s_1'=ds_1+cs_2, s_2'=bs_1+as_2, s_3'-s_1's_2'=s_3-s_1s_2$. Then, $\ov \pr(h')=(1,(s'_1,s'_2, s'_3-s'_1s'_2))\in \SL_2(k)\ltimes \sH$. Thus 
\begin{equation}\label{eq: a formula of Weil} \omega_{\psi^{-1}}(g \ov\pr(h))\phi(x)=\omega_{\psi^{-1}}( \ov \pr(h') g)\phi(x)=\psi^{-1}(s_3'-2s_1's_2'-2xs_2')(\omega_{\psi^{-1}}(g)\phi)(x+s_1'),\end{equation}
by \eqref{eq: Weil representation of J}. 

Next, we compute $W(j(\bx_{3\alpha+\beta}(r_1) \bx_\alpha(r_2) g h ) )$. Using the commutator relations, see \cite[p.192]{Ch}, one can check that: 
\begin{align*}
&\ \quad \bx_{3\alpha+\beta}(y) \bx_\alpha(x) gh\\
&=\bx_{3\alpha+\beta}(y) \bx_\alpha(x) h'g\\
&=\bx_{3\alpha+2\beta}(s_5')\bx_{3\alpha+\beta}(y+s_4') \bx_{\alpha}(x+s_1')\bx_{\alpha+\beta}(s_2')\bx_{2\alpha+\beta}(s_3')g\\
&=\bx_{3\alpha+2\beta}(s_5'')\bx_{3\alpha+\beta}(y+s_4'')\bx_{\alpha+\beta}(s_2')\bx_{\alpha}(x+s_1')\bx_{2\alpha+\beta}(s_3'-2(r_2+s_1')s_2')g\\
&=\bx_{3\alpha+2\beta}(s_5'')\bx_{3\alpha+\beta}(y+s_4''')\bx_{\alpha+\beta}(s_2')\bx_{2\alpha+\beta}(s_3'-2(r_2+s_1')s_2') \bx_{\alpha}(x+s_1')g\\
&=\bx_{3\alpha+2\beta}(s_5'') \bx_{\alpha+\beta}(s_2') \bx_{2\alpha+\beta}(s_3'-2(r_2+s_1')s_2') \bx_{3\alpha+\beta}(y+s_4''') \bx_{\alpha}(x+s_1')g.
\end{align*}
where $s_5''=s_5'+3(x+s_1')(s_2')^2, s_4''=s_4'+3(x+s_1')^2s_2'$, and $s_4'''=s_4''+3(x+s_1')(s_3'-2(x+s_1')s_2')$. Note that $j(\bx_{3\alpha+2\beta}(s_5''))\in U_{3\alpha+\beta}$, $j(\bx_{\alpha+\beta}(s_2'))\in U_{2\alpha+\beta}$, $ j( \bx_{2\alpha+\beta}( s_3'-2(x+s_1')s_2'))=\bx_\alpha(s_3'-2(x+s_1')s_2')$ and $W\in \CW(\Pi, \psi)$, we get
\begin{equation}\label{eq: a formula of Whittaker} W(j(  \bx_{3\alpha+\beta}(y) \bx_\alpha(x) gh))=\psi(s_3'-2(x+s_1')s_2') W(j (\bx_{3\alpha+\beta}(y+s_4''') \bx_{\alpha}(x+s_1')g)).\end{equation}
Plugging \eqref{eq: a formula of Weil} and \eqref{eq: a formula of Whittaker} into the left hand side of \eqref{eq: invariance}, we get
\begin{align*}
&\quad\Psi( \Pi(j(h))W,  \omega_{\psi^{-1}}(\ov \pr(h))\phi, r(\ov \pr(h))f)\\
&=\sum_{g\in N_{\SL_2}\setminus \SL_2(k) }\sum_{x,y\in k} W(j (\bx_{3\alpha+\beta}(y+s_4''') \bx_{\alpha}(x+s_1')g)) (\omega_\psi(g)\phi)(x+s_1')f(g).
\end{align*}
By changing variables, we get 
$$ \Psi( \Pi(j(h))W,  \omega_{\psi^{-1}}(\ov \pr(h))\phi, r(\ov \pr(h))f)=\Psi(W,\phi,f).$$
The completes the proof of the lemma.
\end{proof}
\begin{cor}\label{cor: generic big than one}
If $\Pi$ is an irreducible generic representation of $\RG_2(k)$, then we have 
$$\dim\Hom_J(\Pi,\omega_{\psi}\otimes I(\chi))\ge 1.$$
\end{cor}
\begin{proof}
Let $\Pi^j$ be the representation defined by $\Pi^j(g)=\Pi(j(g))$. Note that $\Pi^j\cong \Pi$ since $j$ is an inner automorphism. The assertion then follows from Lemma \ref{lem: nonvanishing of Ginzburg's local zeta integral} and Lemma \ref{lem: invariance property of Ginzburg's local zeta integral} directly.
\end{proof}
\begin{rmk}
\rm{In the proof of Theorem \ref{thm: multiplicity one} when $p>3$ in $\S$\ref{sec: proof when p>3}, we showed that if $\Pi=X_{33}, X_{17}, X_{18}, X_{19}, \ov{X}_{19}$, then $\pair{\Pi,I(\chi)\otimes \omega_\psi}=0$. Thus by Corollary \ref{cor: generic big than one}, the representations $X_{33},X_{17},X_{18},X_{19}, \ov{X}_{19}$ can not be generic. In particular, the irreducible generic cuspidal representations of $\RG_2(k)$ when $p>3$ must be in the families of the representations $X_i(\pi_i)$ for $i=2,3,6$ when $\pi_i$ are in general positions. Similarly in the proof of Theorem \ref{thm: multiplicity one} when $p=3$ in $\S$\ref{sec: proof when p=3}, we showed that if $\Pi=\theta_{10},\theta_{11},\theta_{12}(k)$, then $\pair{\Pi,I(\chi)\otimes \omega_\psi}=0$. Thus Corollary \ref{cor: generic big than one} shows that $\theta_{10},\theta_{11},\theta_{12}(k)$ cannot be generic. Consequently, the irreducible generic cuspidal representations of $\RG_2(k)$ when $p=3$ must be in the families of the representations $\chi_{12}(k,l),\chi_{13}(k),\chi_{14}(k)$.
}
\end{rmk}

\subsection{\texorpdfstring{$\GL_1$}{Lg}-twisted gamma factors for generic cuspidal representations}
Consider the standard intertwining operator $M: I(\chi)\ra I(\chi^{-1})$
defined by 
$$M(f)(g)=\sum_{x\in k}f((w^1)^{-1} n(x)g),$$
where $w^1=\bpm &1\\-1& \epm$ and $n(x)=\bpm 1&x\\ &1 \epm$. Note that under the embedding $\SL_2(k)\incl \GL_2(k)\cong M\incl \RG_2(k)$, $w^1$ is mapped to $w_\beta$.

\begin{prop}\label{prop: defn of gamma factors}
Let $\Pi$ be an irreducible generic cuspidal representation of $\RG_2(k)$ and $\chi$ be a character of $k^\times$. Then there is a number $\gamma(\Pi\times \chi,\psi)\in \BC$ such that
$$Z(W,\phi,M(f))=\gamma(\Pi\times \chi,\psi)Z(W,\phi,f),$$
for all $W\in \CW(\Pi,\psi),\phi\in \CS(k),f\in I(\chi)$.
\end{prop}

\begin{proof}
Note that $(W,\phi,f)\mapsto Z(W,\phi,f) $ and $(W,\phi,f)\mapsto Z(W,\phi,M(f))$ define two trilinear forms in $\Hom_J(\Pi^j\otimes\omega_{\psi^{-1}}\otimes I(\chi),\BC)$. Then the assertion follows from Theorem \ref{thm: multiplicity one} directly.
\end{proof}

\begin{lem}\label{lem: computation of the gamma factor}
We have 
$$\gamma(\Pi\times\chi,\psi)=\frac{q^{5/2}}{\sqrt{\epsilon_0}}\sum_{a\in k^\times}\CB_{\Pi}(h(a,1)w_1)\epsilon \chi^{-1}(a),$$
where $w_1=w_\beta w_\alpha w_\beta w_\alpha^{-1}w_\beta^{-1}=j(w_\beta)$.
\end{lem}

\begin{proof}
As in Lemma \ref{lem: nonvanishing of Ginzburg's local zeta integral}, 
let $ \delta_0\in \CS(k)$ be the function $\delta_0(x)=0$ if $x\ne 0$ and $\delta_0(0)=1$. 
And let $f_0\in I(\chi)$ be the function such that $\supp(f_0)\subset N_{\SL_2}A_{\SL_2}$ and $f_0(1)=1$. 
Then, by Lemma \ref{lem: nonvanishing of Ginzburg's local zeta integral}, we have $Z(B_\Pi,\delta_0,f_0)=1$. Thus we get 
$$\gamma(\Pi\times \chi,\psi)=Z(\CB_\Pi,\delta_0,M(f_0)),$$
where, 
$$M(f_0)(g)=\sum_{x\in k} f_0\left((w^1)^{-1}n(x) g \right).$$
Since $M(f_0)\in I(\chi^{-1})$ and $\SL_2(k)=N_{\SL_2}A_{\SL_2}\coprod N_{\SL_2}A_{\SL_2}w^1 N_{\SL_2}$, we need to determine the value of $M(f_0)$ at $1$ and at $w^1n(r), r\in k$. Since for any $x\in k$, we have $(w^1)^{-1}n(x)\notin B_{\SL_2}$, we have $M(f_0)(1)=0$. Since 
$(w^1)^{-1}n(x)w^1 n(r)\in B_{\SL_2} \textrm{ iff } x=0,$  $M(f_0)(w^1n(r))=f_0(n(r))=1, \forall r\in k.$ Thus we get 
\begin{align*}
    \gamma(\Pi\times \chi,\psi)&=\sum_{g\in N_{\SL_2}\backslash \SL_2(k)}\sum_{x,y\in k} \CB_\Pi(\bx_{-\beta}(y)\bx_{-(\alpha+\beta)}(x)j(g))\omega_{\psi^{-1}}(g)\delta_0(x)M(f_0)(g)\\
    &=\sum_{a\in k^\times,r\in k}\sum_{x,y\in k}\CB_\Pi(\bx_{-\beta}(y)\bx_{-(\alpha+\beta)}(x)j(t(a)w^1n(r)))\\
    &\qquad \qquad \cdot \omega_{\psi^{-1}}(t(a)w^1n(r))\delta_0(x)M(f_0)(t(a)w^1n(r))\\
    &=\sum_{a\in k^\times,r\in k}\sum_{x,y\in k}\CB_\Pi(\bx_{-\beta}(y)\bx_{-(\alpha+\beta)}(x)j(t(a)w^1n(r)))\\
    &\qquad \qquad \cdot \omega_{\psi^{-1}}(w^1n(r))\delta_0(ax)\epsilon\chi^{-1}(a).
\end{align*}

 Note that $w_1=j(w^1)$ and $j(n(r))=\bx_{3\alpha+2\beta}(r) $, and
$$\bx_{-\beta}(y)\bx_{-(\alpha+\beta)}(x)h(a,1)w_1 \bx_{3\alpha+2\beta}(r)=h(a,1) w_1\bx_{3\alpha+\beta}(ay)\bx_{2\alpha+\beta}(ax)\bx_{3\alpha+2\beta}(r).$$
By Lemma \ref{lem: basic properties of Bessel function}, we have $\CB_\Pi(\bx_{-\beta}(y)\bx_{-(\alpha+\beta)}(x)j(t(a)w^1n(r)) )=\CB_{\Pi}(h(a,1)w_1)$. Thus we get 
\begin{align*}\gamma(\Pi\times \chi,\psi)&=\sum_{a\in k^\times,r\in k}\sum_{x,y\in k}\CB_{\Pi}(h(a,1)w_1)\omega_{\psi^{-1}}(w^1n(r))\delta_0(ax)\epsilon\chi^{-1}(a)\\
&=q\sum_{a\in k^\times}\sum_{x,r\in k}\CB_\Pi(h(a,1)w_1)(\omega_{\psi^{-1}}(w^1n(r))\delta_0)(ax)\epsilon\chi^{-1}(a).
\end{align*}
We have 
\begin{align*}
    (\omega_{\psi^{-1}}(w^1 n(r))\delta_0)(ax)&=\frac{1}{\gamma(1,\psi^{-1})}\sum_{y\in k}\psi^{-1}(-2axy)(\omega_{\psi^{-1}}(n(r)))\delta_0(y)\\
    &=\frac{1}{\gamma(1,\psi^{-1})}\sum_{y\in k}\psi^{-1}(-2axy)\psi(ry^2)\delta_0(y)\\
    &=\frac{1}{\gamma(1,\psi^{-1})}.
\end{align*}
Recall that $$\gamma(1,\psi^{-1})=\sum_{x\in k}\psi^{-1}(-x^2)=\sum_{x\in k}\psi(x^2)=\sqrt{\epsilon_0 q},$$
see \eqref{eq: basic guass sum}. Thus we get 
$$\gamma(\Pi\times \chi,\psi)=\frac{1}{\sqrt{\epsilon_0}}q^{5/2}\sum_{a\in k^\times}\CB_{\Pi}(h(a,1)w_1)\epsilon \chi^{-1}(a).$$
This completes the proof of the lemma.
\end{proof}

\section{Gamma factors for \texorpdfstring{$\RG_2\times \GL_2$}{Lg}}\label{sec: G2 GL2}
In this section, we review the integral for $\RG_2\times \GL_2$ (similar to $\S$\ref{sec: G2 GL1}, 
in our case, it is a sum rather than an integral) developed by Piatetski-Shapiro, Rallis and Schiffmann in \cite{PSRS} and define the $\GL_2$-twisted gamma factors. 

In this section, $k$ is a finite field of odd characteristic unless in subsection \ref{subsec: Pi_Z}, where $k$ can be either a finite field or a $p$-adic field.

\subsection{Embedding of \texorpdfstring{$\RG_2$}{Lg} into \texorpdfstring{$\SO_7$}{Lg}} 

To introduce the integral of Piatetski-Shapiro, Rallis and Schiffmann, we need to embed $\RG_2$ into $\SO_7$. We use the embedding of $\RG_2$ into $\SO_7$ given in \cite{RS}. Let $\CH$ be a quaternion algebra over $k$. We can write $\CH=ke_0\oplus \CH^0$, where $e_0$ is the neutral element and $\CH^0$ is the three dimensional subspace of pure quaternions. We denote by $e_1,e_2,e_3$ a basis of $\CH^0$ such that $e_ie_j=-e_je_i$ and $e_3=(e_1e_2-e_2e_1)/2=e_1e_2$. Let $\lambda=e_1e_1$ and $\mu=e_2e_2$. Then $e_3e_3=-\lambda \mu.$ For $x=a_0e_0+a_1e_1+a_2e_2+a_3e_3\in \CH$ with $a_i\in k$, we define $\bar x=a_0e_0-a_1e_1-a_2e_2-a_3e_3$. 

Let $\CC=\CH\times \CH$. We consider the non-associative product on $\CC$ given by 
$$(a,b)(c,d)=(ac+\bar d b, da+b\bar c).$$
With this product, $\CC$ is called a Cayley or octonion algebra. The conjugate of $(a,b)\in \CC$ is defined by $\ov{(a,b)}=(\bar a,-b)$ and its norm is $Q((a,b))=(a,b)\ov{(a,b)}=a\bar a-b\bar b$.

 Note that $k$ can be embedded in to $\CC$ by the map $a\mapsto (ae_0,0)$. We have a decomposition $\CC=k\oplus \CC^0$, where $\CC^0$ is the space of pure Cayley numbers. Note that $\dim _k \CC^0=7$. We put 
$$X^+=(1/2,1/2), X^-=(1/2,-1/2), X^0=X^+-X^-=(0,-1).$$
One can check that $\CH^0X^+, \CH^0X^-$ are totally isotropic subspaces of $\CC^0$, and we have 
$$\CC^0=\CH^0X^+\oplus kX^0\oplus \CH^0X^-,$$
cf. \cite[p.805]{RS}. The group $\RG_2(k)$ can be defined to be the automorphism group of the algebra $\CC$. Note that if $g\in \RG_2(k)$, then $g(1)=1$, where $1=(e_0,0)\in \CC$ is the unit element, and $(gu)(gv)=g(uv)$ for $u,v\in \CC$. In particular, $g\in \RG_2(k)$ preserves the norm form $Q$. Consider the bilinear form $(v_1,v_2)_Q:=Q(v_1)+Q(v_2)-Q(v_1+v_2)$. Then $g\in \RG_2(k)$ preserves the bilinear form $(~,~)_Q$. One can check that the decomposition $\CC=k\oplus \CC^0$ is an orthogonal decomposition with respect to $(~,~)_Q$. Thus $g\in \RG_2(k)$ preserves $\CC^0$ and $Q|_{\CC^0}$. In particular, we have $\RG_2(k)\subset \RO(\CC^0,Q|_{\CC^0})=\wpair{g\in \GL(\CC^0): (gv_1,gv_2)_Q=(v_1,v_2)_Q, \forall v_1,v_2\in \CC^0}$. By \cite[Corollary 4, p.810]{RS}, one has 
$$\RG_2(k)\subset \SO(\CC^0,Q|_{\CC^0})=\wpair{g\in \RO(\CC^0,Q|_{\CC^0}),\det(g)=1}.$$

Note that the group $\RG_2(k)$ defined above in fact depends on the choice of the quaternion algebra $\CH$. If $\CH$ is taken to be the split quaternion algebra, then the group $\RG_2(k)$ is also split. Since we only care about split $\RG_2$, from now on we take $\CH$ to be the split quaternion algebra. Thus we can assume that $\lambda=\mu =1$.

A basis of $\CC^0$ is given by $e_1^+:=e_1 X^+, e_2^+:=e_2 X^+, e_3^+:=e_3 X^+, e_0:=X^0, e_3^-=e_3 X^-, e_2^-:=e_2 X^-, e_1 ^-:=e_1 X^-$.  From the formulas given in \cite[p.805]{RS}, we can check that the bilinear form $(~,~)_Q$ with respect to the basis $(e_1^+,e_2^+,e_3^+,e_0,e_3^-,e_2^-,e_1^-)$ is given by the following matrix (which is still denoted by $Q$ by abuse of notation)
$$Q=\begin{pmatrix} &&s_3\\ &2& \\ {}^t \!s_3 \end{pmatrix},$$
where $$s_3=\begin{pmatrix} &&1\\ &1&\\ -1&&\end{pmatrix}.$$
Thus $\SO(\CC^0,Q|_{\CC^0})=\wpair{g\in \GL_7(k): {}^t\! g Q g=Q, \det(g)=1}$, where we view elements in $\CC^0$ as column vectors and $\SO(\CC^0,Q|_{\CC^0})$ acts on them from the left hand side. In the following, we will fix $\SO(\CC^0,Q|_{\CC^0})$ as the above form and write it as $\SO_7(k)$. We then get our desired embedding $\RG_2(k)\ra \SO_7(k)$.

Let $T_{\SO(Q)}$ be the diagonal torus of $\SO_7(k)$. A typical element in $T_{\SO(Q)}$ has the form $t=\diag(a_1,a_2,a_3,1,a_3^{-1},a_2^{-1},a_1^{-1})$. Let $\epsilon_i$ be the character of $T_{\SO(Q)}$ of the form $\epsilon_i(t)=a_i$ for $1\le i\le 3$, where $t=\diag(a_1,a_2,a_3,1,a_3^{-1},a_2^{-1},a_1^{-1})$. Then the positive roots of $\SO_7(k)$ relative to the upper triangular Borel subgroup is the set $\wpair{\epsilon_i-\epsilon_j,\epsilon_i+\epsilon_j, \epsilon_i \textrm{ with } 1\le i\le 3, 1\le j\le 3,i<j}$. Under the above embedding $\RG_2(k)\ra \SO_7(k)$, one has $\alpha=\epsilon_2, \beta=\epsilon_1-\epsilon_2, \alpha+\beta=\epsilon_1, 2\alpha+\beta=-\epsilon_3,3\alpha+\beta=\epsilon_2-\epsilon_3,3\alpha+2\beta=\epsilon_1-\epsilon_3$. See \cite[p.812]{RS}.

The embedding $\RG_2(k)\ra \SO_7(k)$ can be explicitly realized by giving matrix realizations of $\bx_\gamma(r)$ for all roots $\gamma$ of $\RG_2$, which is given in Appendix \ref{sec: embedding G2 into SO7}. From this explicit realization, one can see how subgroups of $\RG_2$ are embedded in $\SO_7$. For example, the Levi subgroup $M\cong \GL_2(k)$ is embedded into $\SO_7(k)$ by the map
$$m\mapsto \diag(m, \det(m)^{-1},1, \det(m),m^*), m\in \GL_2(k),$$ where $m^*=\begin{pmatrix}&1\\ 1& \end{pmatrix} {}^t\! m^{-1}\begin{pmatrix}&1\\ 1& \end{pmatrix}  .$

\subsection{The Piatetski-Shapiro-Rallis-Schiffmann local zeta integral for \texorpdfstring{$\RG_2\times \GL_2$}{Lg}} \label{PSRS local zeta integral}

Let $\wt P$ be the parabolic subgroup of $\SO_7(k)$ which is isomorphic to $(\GL_2(k)\times \SO_3(k))\ltimes \wt U$ where $\GL_2(k)\times \SO_3(k)$ is the Levi factor of the form
$$\wpair{\begin{pmatrix}a&&\\ & b& \\ &&a^*  \end{pmatrix},a\in \GL_2(k), b\in \SO_3(k).}$$
Here $\SO_3(k)$ is the special orthogonal group realized by the matrix $\begin{pmatrix}&&-1\\&2&\\-1&& \end{pmatrix}$. Note that the $\GL_2(k)$ part in the Levi of $\wt P$ is exactly the Levi subgroup $M$ of $P\subset \RG_2(k)$. A typical element of $\wt P$ will be written as $(x,y,u)$, where $x\in \GL_2(k),y\in \SO_3(k),u\in \wt U$. Denote $H=M\ltimes Z\subset \RG_2(k)$. Under our fixed embedding $\RG_2(k)\ra \SO_7(k)$, we have $H=\RG_2(k)\cap \wt P$, see \cite[Lemma 1.2, p.1273]{PSRS} and its proof there. One can also see this from the matrix realizations of the embedding in Appendix \ref{sec: embedding G2 into SO7}. 

 Let $(\tau,V_\tau)$ be an irreducible generic representation of $\GL_2(F)\cong M$, we consider the induced representation 
$$I(\tau)=\Ind_{\wt P}^{\SO_7(k)}(\tau \otimes 1_{\SO_3}).$$
A section $\xi\in I(\tau)$ is a map $\xi: \SO_7(k)\ra V_{\tau}$ such that 
$$\xi((x,y,u)g)= \tau(x)\xi(g), x\in \GL_2(k), y\in \SO_3(k), u\in \wt U.$$
We fix a nontrivial $\psi$-Whittaker functional $\Lambda\in \Hom_{N_{\GL_2}}(\tau,\psi^{-1})$ of $\tau$, where $N_{\GL_2}$ is the upper triangular unipotent subgroup of $\GL_2(F)$. We then consider the $\BC$-valued function $f_{\xi}$ on $\SO_7(k)\times \GL_2(k)$ by 
$$f_{\xi}(g,a)=\Lambda(\tau(a)\xi(g)),g\in \SO_7(k),a\in \GL_2(k).$$
We denote by $I(\CW(\tau,\psi^{-1}))$ the space consisting of all functions of the form $f_{\xi}$, $\xi\in I(\tau)$. 

Let $U_H$ be the subgroup of $H$ generated by root spaces of $\beta, 2\alpha+\beta, 3\alpha+\beta, 3\alpha+2\beta$. We have $U_H\subset H=G\cap \wt P$. Let $\psi_{U_H}$ be the character of $U_H$ such that $\psi_{U_H}|_{U_\beta}=\psi$ and $\psi_{U_H}|_{U_\gamma}=1$ for $\gamma=2\alpha+\beta,3\alpha+\beta,3\alpha+2\beta$. For $u\in U_H$, we have
\begin{equation}\label{eq: invariance of induced section} f_{\xi}(ug,\RI_2)=\psi^{-1}_{U_H}(u)f_{\xi}(g,\RI_2),\end{equation}
where $\RI_2$ is the $2\times 2$ identity matrix.

Let $\Pi$ be an irreducible $\psi=\psi_U$-generic representation of $\RG_2(k)$ and $\tau$ be an irreducible generic representation of $\GL_2(k)$. For $W\in \CW(\Pi,\psi)$, and $f\in I(\CW(\tau,\psi^{-1}))$, we consider the following Piatetski-Shapiro-Rallis-Schiffmann local zeta integral
\begin{equation}\label{eq: defn of PSRS integral}\Psi(W,f)=\sum_{g\in U_H\setminus \RG_2(k)} W(g)f(g,\RI_2).
\end{equation}
Note that by \eqref{eq: invariance of induced section}, the above sum $\Psi(W,f)$ is  well-defined. 

\subsection{Decomposition of \texorpdfstring{$I(\tau)|_{\RG_2(k)}$}{Lg}}
Recall that $P'=M'V'$ denotes the standard parabolic subgroup of $\RG_2(k)$ such that $U_\alpha$ is included in the Levi subgroup $M'$. Let $\wt H=w_\alpha w_\beta P'(w_\alpha w_\beta)^{-1}$, which is a conjugate of $P'$ and thus still a parabolic subgroup of $\RG_2(k)$. It is clear that $\wt H=\GL_2(k)\ltimes U'$, where $\GL_2(F)=w_\alpha w_\beta M'(w_\alpha w_\beta)^{-1}$ and $U'=w_\alpha w_\beta V'(w_\alpha w_\beta)^{-1}$. Note that $U' $ is generated by the root subgroups of $\beta, 3\alpha+2\beta,-(3\alpha+\beta), \alpha+\beta,-\alpha$, see \cite[Corollary to Lemma 1.2, p.1276]{PSRS}.

The double coset $\wt P\backslash\SO_7(k)/\RG_2(k)$ has two elements. From \cite[Lemma 1.2 and its Corollary]{PSRS} and Mackey's theory, we have the following decomposition
\begin{equation}\label{eq: decomposition of I(tau)} 0\ra \ind_H^{\RG_2(k)}(\tau\otimes 1_Z)\ra I(\tau)|_{\RG_2(k)}\ra \ind_{\wt H}^{\RG_2(k)}(\tau\otimes 1_{U'})\ra 0. 
\end{equation}
See \cite[p.1287]{PSRS} for a local fields analogue of the above decomposition. Note that over finite fields, the above exact sequence splits. 
\begin{rmk}\label{rmk: ind and Ind}\rm{In the exact sequence \ref{eq: decomposition of I(tau)}, $\ind$ also denotes the induced representation. Note that over finite fields, the notations $\ind$ and $\Ind$ have no difference, but over local fields, they are different. Here, we try to keep the notations the same as in the literature \cite{PSRS} and thus we used two different notations ($\ind$ and $\Ind$) to denote the same object (induced representation). Hopefully, this won't cause any confusion.}
\end{rmk}
\subsection{On the Jacquet functor \texorpdfstring{$\Pi_Z$}{Lg}}\label{subsec: Pi_Z}
In general, let $(\Pi,V_\Pi)$ be representation of a group $ L$ and let $\chi$ be a character of a subgroup $K\subset L$, then the twisted Jacquet functor $\Pi_{K,\chi}$ is defined to be $V_\Pi/\pair{\Pi(k)v-\chi(k)v,k\in K, v\in V_{\Pi}}$. If $\chi=1$ is the trivial character, then we write $\Pi_{K,1}$ as $\Pi_K$.

In this subsection, let $k$ be either a finite field or a $p$-adic field. We go back to our $\RG_2$ notation. Let $\Pi$ be an irreducible generic smooth representation of $\RG_2(k)$. We consider the Jacquet functor $\Pi_Z$. Let $P^1=\wpair{\bpm *&*\\ &1 \epm}$ be the mirabolic subgroup of $\GL_2(k)$. Let $\psi_V$ be the character of $V$ such that $\psi_V|_{U_\alpha}=\psi$ and $\psi_V|_{U_\gamma}=1$ for $\gamma=\alpha+\beta,2\alpha+\beta,3\alpha+\beta,3\alpha+2\beta$.
\begin{lem}\label{lem: Jacquet functor Pi_Z}
We have the following exact sequences
\begin{equation}\label{eq: exact sequence 1}
    0\ra \ind_{P^1}^{\GL_2(k)}(\Pi_{V,\psi_V})\ra \Pi_Z\ra \Pi_V\ra 0,
\end{equation}
and 
\begin{equation}\label{eq: exact sequence 2}
    0\ra \ind_{N_{\GL_2}}^{P^1}(\psi)\ra \Pi_{V,\psi_V}\ra \Pi_{U,\psi'_U}\ra 0,
\end{equation}
where $\ind$ means compact induction when $k$ is a local field, $\psi'_U$ is the degenerate character of $U$ defined by $\psi'_U|_{U_\alpha}=\psi$ and $\psi'_U|_{U_\beta}=1$, and $N_{\GL_2}$ is the upper triangular unipotent subgroup of $\GL_2(k)$.
\end{lem}
\begin{rmk}
{\rm Lemma \ref{lem: Jacquet functor Pi_Z} is the finite and $p$-adic fields analogue of \cite[Theorem 5, p.824]{RS} and its proof given in the following is also parallel to the one given in \cite{RS}. Note that over finite fields, the above exact sequences split and the topology is discrete.}
\end{rmk}

\begin{proof}[\textbf{Proof of  Lemma $\ref{lem: Jacquet functor Pi_Z}$}] 

Note that $\RG_2(k)$ is an $\ell$-group in the sense of \cite{BZ}. Note that when $k$ is a finite field, the topology on $\RG_2(k)$ is discrete. We use the language of sheaf theory on $\ell$-spaces, see \cite{BZ}. 

Note that the parabolic subgroup $P$ normalizes $Z$, and thus $\Pi_Z$ can be viewed as a representation of $P$. Since $Z$ acts on $\Pi_Z$ trivially, we can view $\Pi_Z$ as a representation of $P/Z=M\ltimes (V/Z)$. Note that $V/Z\cong k^2$. Moreover, as a representation of $V/Z$, $\Pi_Z$ is smooth. Denote the space of $\Pi_Z$ by $V_{\Pi_Z}$. The smoothness of $\Pi_Z$ implies that $\CS(V/Z).V_{\Pi_Z}=V_{\Pi_Z}$, see \cite[$\S$2.5]{BZ} for example. Let $\wh{V/Z}$ be the dual group of $V/Z$, i.e., the set of characters on $V/Z$. The Fourier transform defines an isomorphism $\CS(V/Z)\cong \CS(\wh{V/Z})$. Under this isomorphism, we view $V_{\Pi_Z}$ as a module over $\CS(\wh{V/Z})$. By \cite[Proposition 1.14]{BZ}, up to isomorphism, there is a unique sheaf $\CV_{\Pi_Z}$ on $ \widehat{V/Z}$ such that as a $\CS(\wh{V/Z}) $-module, $V_{\Pi_Z}$ is isomorphic to the finite cross sections $(\CV_{\Pi_Z})_c$. For the definition of finite cross sections, see \cite[$\S$1.13]{BZ}.

Note that $V/Z$ is generated by the root space of $\alpha$ and $\alpha+\beta$. The set $\wh{V/Z}$ is consisting of characters of the form $\psi_{\kappa_1,\kappa_2}, \kappa_1,\kappa_2\in k$, where 
$$\psi_{\kappa_1,\kappa_2}(\bx_{\alpha+\beta}(r_1)\bx_\alpha(r_2))=\psi(\kappa_1 r_1+\kappa_2 r_2).$$
The map $(\kappa_1,\kappa_2)\mapsto \psi_{\kappa_1,\kappa_2}$ defines a bijection $k^2\cong \wh{V/Z}$. Under this bijection, we consider the action of $M\cong \GL_2(k)$ on $\wh{V/Z}$ given by 
$$(g, (\kappa_1,\kappa_2))={}^t\! g^{-1} \cdot  {}^t\!(\kappa_1,\kappa_2).$$
This action has two orbits: the open orbit $O=\wpair{\psi_{\kappa_1,\kappa_2}:  (\kappa_1,\kappa_2)\in k^2-\wpair{0}}$ and the closed orbit $C=\wpair{\psi_{0,0}}$. We then have the exact sequence
\begin{equation}\label{eq: sheaf exact sequence}0\ra (\CV_{\Pi_Z})_c(O)\ra (\CV_{\Pi_Z})_c\ra (\CV_{\Pi_Z})_c(C)\ra 0,\end{equation}
see \cite[$\S$1.16]{BZ}.

We consider $\psi_{0,1}\in O$. The stabilizer of $\psi_{0,1}$ in $M\cong \GL_2(k)$ is $P^1$, and the map $g\mapsto g.\psi_{0,1}$ defines a bijection $P^1\backslash \GL_2(k)\ra O$. A simple calculation shows that the stalk of the sheaf $\CV_{\Pi_Z}$ at the point $\psi_{0,1}\in O$ is 
$$(\Pi_Z)_{V/Z,\psi_{0,1}}=\Pi_{V,\psi_V},$$
see \cite[Lemma 5.10]{BZ} for a similar calculation in the $\GL_n$-case. 
Thus by \cite[Proposition 2.23]{BZ}, we have 
$$  (\CV_{\Pi_Z})_c(O)\cong \ind_{P^1}^{\GL_2(k)}(\Pi_{V,\psi_V}).$$
Similarly, consider the stalk of the sheaf $\CV_{\Pi_Z}$ at $\psi_{0,0}$, we have $$(\CV_{\Pi_Z})_c(C)=(\Pi_Z)_{V/Z,\psi_{0,0}}=(\Pi_Z)_{V/Z}=\Pi_V.$$
Now the exact sequence \eqref{eq: exact sequence 1} follows from the exact sequence \eqref{eq: sheaf exact sequence}.

In general, given any smooth representation $\rho$ of $P^1$, we have an exact sequence
\begin{equation} \label{eq: GL2 exact sequence}
0\ra \ind_{N_{\GL_2}}^{P^1}(\rho_{N_{\GL_2},\psi})\ra \rho \ra \rho_{N_{\GL_2}}\ra 0,\end{equation}
 see \cite[Proposition 5.12]{BZ} for example. We now apply the exact sequence \eqref{eq: GL2 exact sequence} to the representation $\rho=\Pi_{V,\psi_V}$. Note that $(\Pi_{V,\psi_V})_{N_{\GL_2}}=\Pi_{U,\psi'_U}$ and $(\Pi_{V,\psi_V})_{N_{\GL_2},\psi}=\Pi_{U,\psi_U}$ (transitivity of Jacquet functors, see \cite[Lemma 2.32, p.24]{BZ}). By the uniqueness of Whittaker model, we get $\dim \Pi_{U,\psi_U}=1$. Thus, as a representation of $N_{\GL_2}$, we have $(\Pi_{V,\psi_V})_{N_{\GL_2},\psi}=\psi$. Then, the exact sequence \eqref{eq: exact sequence 2} follows.
\end{proof}

\subsection{Intertwining operator} \label{intertwining operator}

Denote $w_2=h(1,-1)w_\alpha w_\beta w_\alpha w_\beta^{-1}w_\alpha^{-1}$, one can check that $w_2^2=1$. In matrix form under the embedding $\RG_2(k)\incl \SO_7(k)$, we have 
$$w_2=\bpm &&j_3\\ &-1&\\ {}^t\! j_3 \epm, \textrm{ with } j_3=\bpm &-1&\\ &&-1\\ -1&&\epm.$$ 
Recall that $\wt U$ is the unipotent subgroup of the parabolic $\wt P$ of $\SO_7(k)$ and one can check that $w_2\wt Uw_2$ is the opposite of $\wt U$. Let $\tau$ be a representation of $\GL_2(k)$ and $I(\tau)=\Ind_{\wt P}^{\SO_7(k)}(\tau\otimes 1_{\SO_3})$ as before. For $\xi\in I(\tau)$, we define
$$M_{w_2}(\xi)(g)=\sum_{u\in \wt U}\xi(w_2 ug), g\in \SO_7(k).$$
Given $g\in \GL_2(k)$, denote by $\tilde g=(g,\det(g)^{-1},1,\det(g),g^*)$, which is the matrix realization of $g$ in $\SO_7(k)$. Recall that $g^*=\bpm &1\\ 1& \epm {}^t\! g^{-1}\bpm &1\\ 1& \epm $. One can check that
$$w_2 \tilde g w_2=(g^*,\det(g),1,\det(g)^{-1},g).$$
From this observation, one can check that $M_{w_2}(\xi)\in I(\tau^*)$, where $\tau^*$ is the representation given by $\tau^*(g)=\tau(g^*)$. Notice that $\tau^*$ is isomorphic to the contragredient representation of $\tau$. For $f\in I(\CW(\tau,\psi^{-1}))$, we define 
$$M_{w_2}(f)(g,a)=\sum_{u\in \wt U} f(w_2ug, d_1a^*),$$
where $d_1=\diag(-1,1)$. Here the factor $d_1$ is added to make sure that the function $a\mapsto M_{w_2}(f)(g,a)$ is a $\psi^{-1}$-Whittaker function on $\GL_2(k)$. Hence, $M_{w_2}(f)\in I(\CW(\tau^*,\psi^{-1}))$, and for $W\in \CW(\Pi,\psi)$, one can consider the sum $$\Psi(W,M_{w_2}(f))=\sum_{g\in U_H\setminus \RG_2(k)}W(g)M_{w_2}(f)(g).$$

\subsection{\texorpdfstring{$\GL_2$}{Lg}-twisted gamma factors for generic cuspidal representations}
\begin{prop}\label{prop: uniqueness of a bilinear form}
Let $\Pi$ be an irreducible generic cuspidal representation of $\RG_2(k)$ and let $\tau$ be an irreducible generic representation of $\GL_2(k)$. Then we have 
$$\dim \Hom_{\RG_2(k)}(I(\tau)|_{\RG_2(k)},\Pi)=1.$$
\end{prop}

\begin{proof}
We use the decomposition $I(\tau)|_{\RG_2(k)}$ given in \eqref{eq: decomposition of I(tau)}. By Frobenius reciprocity, we have
\begin{align*}
    \Hom_{\RG_2(k)}(\ind_{\wt H}^G(\tau\otimes 1_{U'}),\Pi)
    =&\Hom_{\wt H}(\tau\otimes 1_{U'},\Pi|_{\wt H})\\
    =&\Hom_{\GL_2(k)}(\tau,\Pi_{U'}).
\end{align*}
Since $U'$ is the unipotent of a nontrivial parabolic subgroup and $\Pi$ is cuspidal, we get $\Pi_{U'}=0$. Thus we have 
$$\Hom_{\RG_2(k)}(\ind_{\wt H}^{\RG_2(k)}(\tau\otimes 1_{U'}),\Pi)=0. $$
From the decomposition of $I(\tau)|_{\RG_2(k)}$ in \eqref{eq: decomposition of I(tau)} and Frobenius reciprocity, we have 
\begin{align*} \dim \Hom_{\RG_2(k)}(I(\tau)|_{\RG_2(k)},\Pi)&=\Hom_{\RG_2(k)}(\ind_H^{\RG_2(k)}(\tau\otimes 1_Z),\Pi)\\
&=\Hom_H(\tau\otimes 1_Z,\Pi)\\
&=\Hom_{\GL_2(k)}(\tau,\Pi_Z).
\end{align*}
We now apply exact sequences in Lemma \ref{lem: Jacquet functor Pi_Z}. Note that $V$ is a unipotent subgroup of a nontrivial parabolic, we have $\Pi_V=0$. Thus we have $\Pi_Z\cong \ind_{P^1}^{\GL_2(k)}(\Pi_{V,\psi_V})$ by \eqref{eq: exact sequence 1}. By Frobenius reciprocity again, we get 
\begin{align*}
    \Hom_{\RG_2(k)}(I(\tau)|_{\RG_2(k)},\Pi)&=\Hom_{\GL_2(k)}(\tau,\Pi_Z)\\
    &=\Hom_{\GL_2(k)}(\tau,\ind_{P^1}^{\GL_2(k)}(\Pi_{V,\psi_V}))\\
    &=\Hom_{P^1}(\tau|_{P^1},\Pi_{V,\psi_V}).
\end{align*}
Since $\psi'_U|_{U_\beta}=1$, we have $\Pi_{U,\psi_U'}=(\Pi_{V'})_{U_\alpha,\psi}=0$ since $V'$ is the unipotent of the nontrivial parabolic subgroup $P'$ and $\Pi$ is cuspidal. Thus \eqref{eq: exact sequence 2} shows that $\Pi_{V,\psi_V}\cong \ind_{N_{\GL_2}}^{P^1}(\psi)$. We then have 
\begin{align*}
    \Hom_{\RG_2(k)}(I(\tau)|_{\RG_2(k)},\Pi)&=\Hom_{P^1}(\tau|_{P^1},\Pi_{V,\psi_V})\\
    &=\Hom_{P^1}(\tau,\ind_{N_{\GL_2}}^{P^1}(\psi))\\
    &=\Hom_{N_{\GL_2}}(\tau,\psi).
\end{align*}
Since $\tau$ is irreducible generic, we have $\Hom_{N_{\GL_2}}(\tau,\psi)=1$ by the uniqueness of Whittaker model for $\GL_2(k)$. This completes the proof.
\end{proof}
\begin{rmk}\label{rmk: why do we require cuspidal}
\rm{Note that if $\Pi$ is not cuspidal, from the above proof, we cannot expect that $$\dim\Hom_{\RG_2(k)}(I(\tau)|_{\RG_2(k)},\Pi)=1$$ in general. This is because the \textit{tale} terms, say, $\Pi_{U'},\Pi_V, \Pi_{U,\psi'_U}$ can cause some trouble. For example, if $\Pi_{U'}\ne 0$ and $\Hom_{\GL_2(k)}(\tau,\Pi_{U'})\ne 0$, then the above proof shows that 
\begin{align*}
    &\quad \dim\Hom_{\RG_2(k)}(I(\tau)|_{\RG_2(k)},\Pi)\\
    &=\dim\Hom_{\RG_2(k)}(\ind_{\wt H}^G(\tau\otimes 1_{U'}),\Pi)+\dim \Hom_{\RG_2(k)}(\ind_H^{\RG_2(k)}(\tau\otimes 1_Z),\Pi)\\
    &\ge \dim\Hom_{\RG_2(k)}(\ind_{\wt H}^G(\tau\otimes 1_{U'}),\Pi)+1\\
    &\ge 2.
\end{align*}
Note that over a $p$-adic field $k$, we can introduce a complex number parameter in the induced representation $I(\tau)$ and consider the induced representation 
$$I(s,\tau):=\Ind_{\wt P}^{\SO_7(k)}(\tau|\det|^s\otimes 1_{\SO_3})$$
on $\SO_7(k)$. Then the same strategy can show that, \textit{except for a finite number of $q^s$}, where $q$ is the number of residue field of $k$,  we have 
$$\dim\Hom_{\RG_2(k)}(I(s,\tau)|_{\RG_2(k)},\Pi)=1,$$ for any irreducible generic representation $\Pi$ of $\RG_2(k)$. Here we don't need the cuspidal condition on $\Pi$, because the \textit{tale} terms $\Pi_{U'}, \Pi_V$, all have finite length as a representation of $\GL_2(k),$ and $\Pi_{U,\psi'_U}$ has finite dimension, and thus the corresponding Hom spaces are still zero if we exclude a finite number of $q^s$.}
\end{rmk}

\begin{thm}\label{thm: existence of gamma factor for G2 times GL2}
Let $\Pi$ be an irreducible generic cuspidal representation of $\RG_2(k)$ and $\tau$ be an irreducible generic representation of $\GL_2(k)$. Then there exists a number $\gamma(\Pi\times  \tau,\psi)$ such that 
$$\Psi(W,M_{w_2}(f))=\gamma(\Pi\times \tau,\psi)\Psi(W,f)$$
for all $W\in \CW(\Pi,\psi)$ and $f\in I(\CW(\tau,\psi^{-1}))$.
\end{thm}
\begin{proof}
Note that $(W,f)\mapsto \Psi(W,f)$ and $(W,f)\mapsto \Psi(W,M(f))$ define two elements in $\Hom_{\RG_2(k)}(\Pi\otimes I(\tau),\BC)$. The theorem follows from Proposition \ref{prop: uniqueness of a bilinear form} directly.
\end{proof}
\begin{rmk}
{\rm Let $k$ be a $p$-adic field, $\Pi$ be a irreducible generic  smooth representation of $\RG_2(k)$ and $\tau$ be a irreducible generic  representation of $\GL_2(k)$. Following a similar strategy, we can prove the existence of a local gamma factor $\gamma(s,\Pi\times\tau,\psi)$, which satisfies a similar local functional equation using the local Piatetski-Shapiro-Rallis-Schiffmann integral. Note that in the $p$-adic field case, we do not need to require that $\Pi$ is cuspidal, see \ref{rmk: why do we require cuspidal} for an explanation.}
\end{rmk}

\section{A converse theorem}\label{sec: converse theorem}
In this section $k$ is a finite field of odd characteristic.
\subsection{Weyl elements supporting Bessel functions}

Let $\Delta=\wpair{\alpha,\beta}$ be the set of simple roots of $\RG_2$ and let $\RW(\RG_2)$ be the Weyl group of $\RG_2$. The group $\RW(\RG_2)$ is generated by $s_\alpha,s_\beta$ and has 12 elements. Let $\RB(\RG_2)=\wpair{w\in \RW(\RG_2): \forall \gamma\in \Delta, w\gamma>0 \implies w\gamma\in \Delta}$. The set $\RB(\RG_2)$ is called the set of Weyl elements which support Bessel functions and the name is justified by the following 
\begin{lem}\label{lem: weyl element which support bessel function}
Let $\Pi$ be an irreducible generic representation of $\RG_2(k)$ and $\CB_\Pi\in \CW(\Pi,\psi)$ be the Bessel function. If $w\in \RW(\RG_2)-\RB(\RG_2)$ and $\dot w\in \RG_2(k)$ is a representative of $w$, then $$\CB_\Pi(t\dot w)=0,\forall t\in T. $$
\end{lem}

\begin{proof}
Since $w\notin \RB(\RG_2)$, there exists an element $\gamma\in \Delta$ such that $w\gamma>0$ but $w\gamma$ is not simple. For any $r\in k$, we consider the element $\bx_\gamma(r)\in U_\gamma\subset U$. We have 
$$ t\dot w \bx_\gamma(r)=t\bx_{w\gamma}(cr)\dot w=\bx_{w\gamma}(w\gamma(t)cr)t\dot w,$$
where $c\in \wpair{\pm 1}$. Note that $\psi_U(\bx_{w\gamma}(w\gamma(t)cr))=1$ since $w\gamma$ is not a simple root. By Lemma \ref{lem: basic properties of Bessel function}, we have 
$$\psi(r)\CB_\Pi(t\dot w)=\CB_\Pi(t\dot w),\forall r\in k.$$
Since $\psi$ is not trivial, we must have $\CB_\Pi(t\dot w)=0$. 
\end{proof}

Let $w_\ell=(s_\alpha s_\beta)^3$, which is the long Weyl element in $\RW(\RG_2)$. One can check that $\RB(\RG_2)=\wpair{1,w_\ell s_\alpha,w_\ell s_\beta,w_\ell}$. Note that $w_1=w_\beta w_\alpha w_\beta w_\alpha^{-1}w_\beta^{-1}$ is a representative of $w_\ell s_\alpha$ and $w_2=h(1,-1)w_\alpha w_\beta w_\alpha w_\beta^{-1}w_\alpha^{-1}$ is a representative of $w_\ell s_\beta$.
\subsection{An auxiliary lemma}
Let $t$ be a positive integer and $ N_t$ be the upper triangular unipotent subgroup of $\GL_t(k)$. Let $\psi_t$ be a generic character of $N_t$.
\begin{lem}[{\cite[Lemma 3.1]{N14}}]\label{lem: density}
Let $\phi$ be a function on $\GL_t(k)$ such that $\phi(ng)=\psi_t(n)\phi(g)$ for all $n\in N_t$ and $g\in \GL_t(k)$. If 
$$\sum_{g\in N_t\backslash \GL_t(k)}\phi(g)W(g)=0,$$
for all $W\in \CW(\pi,\psi^{-1}_t)$ and all irreducible generic representations $\pi$ of $\GL_t(k)$, then $\phi\equiv 0$.
\end{lem}
Note that in the above lemma, when $t=1$, $N_t$ is trivial. We will only use the above lemma for $t=1,2$.

\subsection{The converse theorem and twisting by \texorpdfstring{$\GL_1$}{Lg}}\label{subsection: twisting by GL1}

The following theorem is the main result of this paper. 

\begin{thm}\label{thm: converse theorem}
Let $k$ be a finite field with odd characteristic. Let $\Pi_1,\Pi_2$ be two irreducible generic cuspidal representation of $\RG_2(k)$. If $$\gamma(\Pi_1\times\chi,\psi)=\gamma(\Pi_2\times \chi,\psi),$$
$$\gamma(\Pi_1\times \tau,\psi)=\gamma(\Pi_2\times \tau,\psi),$$
for all characters $\chi$ of $k^\times$ and all irreducible generic representations $\tau$ of $\GL_2(k)$, then $\Pi_1\cong \Pi_2$.
\end{thm}

The proof of Theorem \ref{thm: converse theorem} will be given in the following subsections. The strategy of the proof is as follows. Let $\CB_i:=\CB_{\Pi_i}\in \CW(\Pi_i,\psi)$ be the Bessel function of $\Pi_i$ for $i=1,2$. We will prove that $\CB_1(g)=\CB_2(g)$ for all $g\in \RG_2(k)$ under the assumption of Theorem \ref{thm: converse theorem}. Since $\RG_2(k)=\coprod_{w\in \RW(\RG_2)}BwB$, it suffices to show that $\CB_1$ agree with $\CB_2$ on various cells $BwB$. By Lemma \ref{lem: weyl element which support bessel function} and Lemma \ref{lem: basic properties of Bessel function}, if $w\notin \RB(\RG_2)$, we have $\CB_1(g)=\CB_2(g)=0$ for $g\in BwB$. If $w=1$, we also have $\CB_1(g)=\CB_2(g),\forall g\in B$ by Lemma \ref{lem: basic properties of Bessel function} and Lemma \ref{lem: further properties of Bessel functions}. Thus it suffices to show that $\CB_1(g)=\CB_2(g),\forall g\in BwB$ with $w=w_1,w_2,w_\ell$. Here we do not distinguish a Weyl element and its representative. We start from $w_1$.

\begin{lem}\label{lem: twisting by GL1}
If $\gamma(\Pi_1\times\chi,\psi)=\gamma(\Pi_2\times\chi,\psi)$ for all characters $\chi$ of $\GL_1(k)$, then $\CB_1(g)=\CB_2(g)$, for all $g\in Bw_1B$.
\end{lem}
\begin{proof}
By Lemma \ref{lem: computation of the gamma factor}, we have 
$$\gamma(\Pi_i\times \chi,\psi)=\frac{q^{5/2}}{\sqrt{\epsilon_0}}\sum_{a\in k^\times}\CB_i(h(a,1)w_1)\epsilon \chi^{-1}(a).$$
Thus the assumplition implies that 
$$\sum_{a\in k^\times} (\CB_1(h(a,1)w_1)-\CB_2(h(a,1)w_1))\epsilon\chi^{-1}(a)=0,$$
for all character $\chi$ of $k^\times$. Then we get
$$\CB_1(h(a,1)w_1)-\CB_2(h(a,1)w_1)=0 $$
for all $a\in k^\times$ by Lemma \ref{lem: density}. 

On the other hand, for any $a,b\in k^\times$, one can check the following identity
$$\bx_\alpha(br)h(a,b)w_1=h(a,b)w_1\bx_\alpha(r),\forall r\in k.$$
Thus by Lemma \ref{lem: basic properties of Bessel function}, we have 
$$\psi(br)\CB_i(h(a,b)w_1)=\psi(r)\CB_i(h(a,b)w_1),\forall r\in k.$$
Since $\psi$ is nontrivial, we get $$\CB_i(h(a,b)w_1)=0, \textrm{ if } b\ne 1.$$

Therefore, we get $\CB_1(tw_1)=\CB_2(tw_1)$ for all $t\in T$. Since $Bw_1B=UTw_1U$, we get 
$$\CB_1(g)=\CB_2(g),\forall g\in Bw_1B$$
by Lemma \ref{lem: basic properties of Bessel function}.
\end{proof}

\subsection{Sections in the induced representation \texorpdfstring{$I(\tau)$}{Lg}}\label{subsection: section of induced representation}

Let $(\tau,V_\tau)$ be an irreducible generic representation of $\GL_2(k)$. Recall that $I(\tau)=\Ind_{\wt P}^{\SO_7}(\tau\otimes 1_{\SO_3})$. Fix a nonzero $v \in V_{\tau}$, we consider the function $\xi_v$ on $\RG_2(k)$ defined by $\supp(\xi_v)=H$ and 
$$\xi_v(az)=\tau(a)v,a\in M\cong \GL_2(k),z\in Z.$$
Note that $\xi_v\in \ind_H^{\RG_2(k)}(\tau\otimes 1_Z)$. For an explanation of the notations $\ind$ and $\Ind$, see Remark \ref{rmk: ind and Ind}. By zero extension, $\xi_v$ can be viewed as an element in $\Ind_{\wt P}^{\SO_7(k)}(\tau\otimes 1_{\SO_3}\otimes 1_{\wt U})=I(\tau)$. This can be checked by a direct computation or can be checked from the exact sequence \eqref{eq: decomposition of I(tau)}. Following $\S$\ref{PSRS local zeta integral}, we fix a nonzero Whittaker functional $\Lambda\in \Hom_{N_{\GL_2}}(\tau,\psi^{-1})$ and consider the following function in 
 $I(\CW(\tau,\psi^{-1}))$,
$$ f_{\xi_v}(g,a)=\Lambda(\tau(a)\xi_v(g)),g\in \SO_7(k),a\in \GL_2(k).$$
Let 
$$\tilde f_{v}=M_{w_2}(f_{\xi_v}).$$

\begin{lem}\label{lem: computation of intertwining operator} Let $g=m\bx_\alpha(r_1)\bx_{\alpha+\beta}(r_2)w_2 \bx_{\alpha}(s_1)\bx_{\alpha+\beta}(s_2)s'$ with $m\in M, s'\in Z$. Then if $\widetilde f_v(g, I_2)\ne 0$, we have $r_1=r_2=s_1=s_2=0$. Moreover, if $r_1=r_2=s_1=s_2=0$, we have 
$$\widetilde f_v(g, I_2)=W_v^*(m),$$
where $W_v^*(m):=\Lambda(\tau(d_1m^*)v)$, recall that $d_1=\diag(1,-1)$. 

\end{lem}

\begin{proof}
By the definition of intertwining operator in $\S$\ref{intertwining operator}, we have 
\begin{align}\label{eq: intertwining operator}
\widetilde f_v(g,\RI_2)&=\sum_{u\in \wt U}f_{\xi_v}(w_2u g,d_1)\\
&=\sum_{\bar u\in \ov{\wt U}}f_{\xi_v}(\bar u m^* \bx_{-(\alpha+\beta)}(-r_1)\bx_{-\alpha}(-r_2) \bx_{\alpha}(s_1)\bx_{\alpha+\beta}(s_2)s',d_1 ),\nonumber
\end{align}
where $\ov{\wt U}$ is the opposite of $\wt U$.
If $\widetilde f_v(g, I_2)\ne 0$, then there exists $\bar u\in \ov{\wt U}$, such that 
$$\bar u m^* \bx_{-(\alpha+\beta)}(-r_1)\bx_{-\alpha}(-r_2)\bx_{\alpha}(s_1)\bx_{\alpha+\beta}(s_2)s'=h\in H=\wt P\cap \RG_2(k).$$ 
Then we have 
\begin{equation}\label{eq: an equation in the matrix computation of G2}
\bx_{-(\alpha+\beta)}(-r_1)\bx_{-\alpha}(-r_2)\bx_{\alpha}(s_1)\bx_{\alpha+\beta}(s_2)=m_1\bar u^{-1}h (s')^{-1},
\end{equation}
where $m_1=(m^*)^{-1}$. Suppose that $h=m_2 z_2$ with $m_2\in M, z_2\in Z$, and write $z=z_2(s')^{-1}\in Z$. 

Note that a typical element in $Z$ has the form 
$$ \bx_{2\alpha+\beta}(r_3)\bx_{3\alpha+\beta}(r_4)\bx_{3\alpha+2\beta}(r_5)=\bpm 1&0&r_5&0&0&-r_3&0\\ 0&1&r_4&0&0&0&r_3\\ 0&0&1&0&0&0&0\\ 0&0&-r_3&1&0&0&0\\ 0&0&r_3^2&-2r_3&1&r_4&r_5\\0 &0&0&0&0&1&0\\0 &0&0&0&0&0&1\\ \epm, $$
where the matrix form can be computed using the matrix realization of $\RG_2(k)$ given in Appendix \ref{sec: embedding G2 into SO7}. For simplicity, we write the element $z\in Z$ as
$$z=\bpm \RI_2&x_1&x_2\\ 0&b&x_1^*\\ &&\RI_2 \epm,$$
with $b\in \SO_{ 3}(k),x_2\in \Mat_{2\times 2}(k)$ and $x_1\in \Mat_{2\times 3}(k)$ of the form
$$x_1=\bpm * &0&0\\ *&0&0\epm.$$
We write $m_i=\diag(a_i,I_3,a_i^*)$ with $a_i\in \GL_2(k)$ for $i=1,2$, and
$$\bar u^{-1}=\bpm \RI_2&&\\ \bar u_1 &\RI_3 &\\ \bar u_2&\bar u_1^* &\RI_2 \epm,$$
where $\bar u_1\in \Mat_{3\times 2}(k),\bar u_2\in \Mat_{2\times 2}(k),$ and $\bar u_1^*$ is determined by $\bar u_1$. 

Then we have 
\begin{align*}
m_1\bar u^{-1}h(s')^{-1}&=m_1\bar u^{-1}m_2z\\
&=\bpm a_1a_2& a_1a_2x_1&a_1a_2x_2\\ \bar u_1 a_2 &*&*\\ a_1^*\bar u_2 a_2&*&* \epm.
\end{align*}
On the other hand, from the matrix realization given in Appendix \ref{sec: embedding G2 into SO7}, we have
$$\bx_{-(\alpha+\beta)}(-r_1)\bx_{-\alpha}(-r_2)\bx_{\alpha}(s_1)\bx_{\alpha+\beta}(s_2)=\bpm b_1&y_1&y_2\\ u_1'&*&*\\ u_2'&*&*  \epm,$$
where 
\begin{align*}
b_1&=\bpm 1-r_2s_1& r_2s_2\\  r_1s_1&1-r_1s_2  \epm,\\
y_1&=\bpm 0&-2s_2(1-r_2s_1)&r_2\\ -s_1^2s_2& -2s_1(1+r_1s_2)&-r_1  \epm, \\
u_1'&=\bpm 0&r_1r_2^2\\ r_1-r_1r_2s_1& r_2+r_1r_2s_2\\ -s_1&s_2 \epm,\\
u_2'&=\bpm 0&-r_2^2 \\ -r_1^2(1-r_2s_1) & -r_1r_2(2+r_1s_2)\epm.
\end{align*}
From the identity \eqref{eq: an equation in the matrix computation of G2}, we have
$$b_1=a_1a_2, \quad y_1=a_1a_2x_1, \quad 
y_2=a_1a_2x_2, \quad
u_1'=\bar u_1a_2, \quad 
u_2'=a_1^* \bar u_2 a_2.$$
Since $a_1a_2x_1$ is still of the form 
$$\bpm *&0&0\\ *&0&0\epm,$$
the equation $y_1=a_1a_2x_1$ implies that 
$$\bpm -2s_2(1-r_2s_1)& r_2\\ -2s_1(1+r_1s_2)& -r_1 \epm=\bpm 0&0\\ 0&0 \epm,$$
which then implies that $r_1=r_2=s_1=s_2=0$ since $2\ne 0$ in $k$. 

If  $r_1=r_2=s_1=s_2=0$, we then have $u_1'=0,u_2'=0$ and thus $\bar u_1=0,\bar u_2=0$. Hence $\bar u=1$. Thus, if $r_1=r_2=s_1=s_2=0$, by \eqref{eq: intertwining operator}, we have 
\begin{align*}
\widetilde f_v(g,\RI_2)&=f_{\xi_v}(m^*s',d_1)=\Lambda(\tau(d_1)\xi_v(m^*s'))=\Lambda(\tau(d_1m^*)v)=W_v^*(m).
\end{align*}
This completes the proof of the lemma.
\end{proof}

\subsection{Proof of Theorem \ref{thm: converse theorem}}

Denote by $\CB(g)=\CB_1(g)-\CB_2(g)$. By the discussion in $\S$\ref{subsection: twisting by GL1} and Lemma \ref{lem: twisting by GL1}, we see that $\CB$ is supported on $Bw_2B\coprod Bw_\ell B$. 

Let $(\tau,V_\tau)$ be an irreducible generic representation of $\GL_2(k)$, $v\in V_\tau$, and $f_{\xi_v}\in I(\CW(\tau,\psi^{-1}))$ be the section constructed in $\S$\ref{subsection: section of induced representation}. We now compute $\Psi(\CB_i,f_{\xi_v})$ for $i=1,2$. Since the function $\CB_i(g)f_{\xi_v}(g)$ is supported on $G_2\cap \wt P=H$, we have 
\begin{align}\label{eq: left side of local zeta integral}\Psi(\CB_i,f_{\xi_v})&=\sum_{g\in U_H\backslash \RG_2(k)}\CB_i(g)f_{\xi_v}(g,\RI_2)\\
&=\sum_{g\in U_H\backslash H}\CB_i(g)f_{\xi_v}(g,\RI_2) \nonumber\\
&=\sum_{g\in U_\beta \backslash M}\CB_i(g)f_{\xi_v}(g,\RI_2) \nonumber\\
&=\sum_{g\in N_{\GL_2}\setminus \GL_2(k)}\CB_i( g) W_v(g),\nonumber
\end{align}
where an element $ g\in \GL_2(k)$ is identified with an element of $\RG_2(k)$ via the embedding $\GL_2(k)\cong M\ra \RG_2(k)$, and $W_v(g)=\Lambda(\tau(g)v)$, which is the Whittaker function of $\tau$ associated with $v\in V_\tau$. Note that $M\subset B\cup Bs_\beta B$, which has empty intersection with $Bw_2B\coprod Bw_\ell B$. Since $\CB$ is supported on $Bw_2B\coprod Bw_\ell B$,  it vanishes on $M$. We then have
$$\Psi(\CB_1,f_{\xi_v})-\Psi(\CB_2,f_{\xi_v})=\sum_{g\in N_{\GL_2}\setminus \GL_2(k)}\CB( g) W_v(g)=0.$$
Thus the assumption $\gamma(\Pi_1\times \tau,\psi)=\gamma(\Pi_2\times \tau,\psi)$ and the functional equation, see Theorem \ref{thm: existence of gamma factor for G2 times GL2}, implies that 
$$\Psi(\CB_1,\widetilde f_v)=\Psi(\CB_2,\widetilde f_v),$$
or 
\begin{equation}\label{eq: vanishing of an integral}\Psi(\CB,\widetilde f_v)=\Psi(\CB_1,\widetilde f_v)-\Psi(\CB_2,\widetilde f_v)=0.
\end{equation}
On the other hand, we have 
\begin{align*}
    \Psi(\CB,\widetilde f_v)&=\sum_{g\in U_H\setminus \RG_2(k)}\CB(g)\widetilde f_v(g,\RI_2).
\end{align*}
Note that $\RG_2(k)$ has the following  decomposition
\begin{equation}\label{eq: a Bruhat decomposition of G2}\RG_2(k)=P\coprod Pw_\alpha P\coprod Pw_\alpha w_\beta w_\alpha P \coprod Pw_2P.\end{equation}
Since $\CB$ is supported on $ Bw_2B\cup Bw_\ell B\subset Pw_2P$, it vanishes on $$ P\coprod Pw_\alpha P\coprod Pw_\alpha w_\beta w_\alpha P.$$
Furthermore, we have 
$$U_H\setminus Pw_2P=U_H\setminus (MVw_2V)=U_\beta\backslash M\times U_\alpha U_{\alpha+\beta}\times  w_2V.$$
By the above discussion and Lemma \ref{lem: computation of intertwining operator}, we have 
\begin{align}
    \Psi(\CB,\widetilde f_v)=&\sum_{m\in U_\beta\backslash M}\sum_{r_1,r_2,s_1,s_2\in k,s'\in Z}\CB(m\bx_\alpha(r_1)\bx_{\alpha+\beta}(r_2) w_2 \bx_\alpha(s_1)\bx_{\alpha+\beta}(s_2)s') \nonumber\\
    &\qquad \cdot \widetilde f_v (m\bx_\alpha(r_1)\bx_{\alpha+\beta}(r_2) w_2 \bx_\alpha(s_1)\bx_{\alpha+\beta}(s_2)s',\RI_2) \nonumber\\
    &=q^3\sum_{m\in U_\beta\backslash M} \CB(mw_2)W_v^*(m). \label{eq: computation of the other side of the G2 GL2 integral}
\end{align}
Then the equation \eqref{eq: vanishing of an integral} implies that 
\begin{equation}\label{eq: main identiy}\sum_{m\in U_\beta\backslash M} \CB(mw_2)W_v^*(m)=0,
\end{equation}
which holds for all $v\in V_\tau$ and all irreducible generic representations $\tau$ of $\GL_2(k)$. Thus by Lemma \ref{lem: density}, we have 
\begin{equation}\label{eq: main identity}\CB(mw_2)=0, \forall m\in M.
\end{equation}

If we take $m=h(x,y)\in M$ in \eqref{eq: main identity}, we get 
\begin{equation}\label{eq: vanishing on w2}
    \CB(h(x,y)w_2)=0,\forall x,y\in k^\times.
\end{equation}
If we take $m=h(x,y)w_\beta$ in \eqref{eq: main identity}, we then get 
\begin{equation}\label{eq: vanishing on wl}
    \CB(h(x,y)w_\beta w_2)=0, \forall x,y\in k^\times.
\end{equation}
Denote $\dot w_\ell=w_\beta w_2$. Note that $\dot w_\ell$ is a representative of $w_\ell$. Together with Lemma \ref{lem: basic properties of Bessel function}, equations \eqref{eq: vanishing on w2} \eqref{eq: vanishing on wl} imply that $\CB$ vanishes on the cells $Bw_2B$ and $Bw_\ell B$. This shows that $\CB$ is identically zero. Thus 
$$\CB_1(g)=\CB_2(g),\forall g\in \RG_2(k).$$
By the uniqueness of Whittaker model and irreducibility of $\Pi_1,\Pi_2$, we get $\Pi_1\cong \Pi_2$.

This completes the proof of Theorem \ref{thm: converse theorem}. 
\begin{rmk}
\rm{In \cite{Ye18}, Ye proved a variant of Lemma \ref{lem: density}, which can be used to refine Theorem \ref{thm: converse theorem} a little bit. Let the notations be the same as in that of Lemma \ref{lem: density}. Additionally, assume that the function $\phi$ satisfies $\sum_{u\in U'}\phi(g_1ug_2)=0$ for all $g_1,g_2\in \GL_t(k)$ and all standard unipotent subgroups $U'\subset \GL_t(k)$. If $\sum_{g\in N_t\backslash \GL_t(k)}\phi(g)W(g)=0$ for all $W\in \CW(\pi,\psi_t^{-1})$ and all irreducible generic \textit{cuspidal} representations $\pi$ of $\GL_t(k)$, then $\phi\equiv 0$ by \cite[Lemma 5.1]{Ye18}. Consequently, in Theorem \ref{thm: converse theorem}, the condition can be relaxed to: $\gamma(\Pi_1\times \chi,\psi)=\gamma(\Pi_2\times \chi,\psi)$, $\gamma(\Pi_1\times \tau,\psi)=\gamma(\Pi_2\times \tau,\psi)$ for all characters $\chi$ of $k^\times$ and all irreducible generic \textit{cuspidal} representations $\tau$ of $\GL_2(k)$, i.e., one only needs the irreducible generic \textit{cuspidal} $\GL_2(k)$ twists condition in Theorem \ref{thm: converse theorem}. In fact, in the above proof, for $i=1,2$, the function $\CB_i$ satisfies the additional condition $\sum_{u\in N_{\GL_2}}\CB_i(g_1ug_2)=0$ by the cuspidality condition of $\Pi_i$.\footnote{In fact, let $0\neq l_i\in \Hom_{N_{\GL_2}}(\GL_2(k),\psi^{-1})$ and let $v_i\in \Pi_i$ be the corresponding Whittaker vector such that $\CB_i(g)=l_i(\Pi_i(g)v_i)$. Then $\sum_{u\in N_{\GL_2}}\CB_i(g_1ug_2)=\sum_{u\in N_{\GL_2}}l_i(\Pi(g_1ug_2)v_i)=l_i(\Pi(g_1)\wt v_i)$ with $\wt v_i=\sum_{u\in N_{\GL_2}} \Pi(u)\Pi(g_2)v_i$. Since $\Pi_i$ is cuspidal, we have $\wt v_i=0$ and thus $\sum_{u\in N_{\GL_2}}\CB_i(g_1ug_2)=0.$} Thus the function $\CB=\CB_1-\CB_2$ also satisfies the condition $\sum_{u\in N_{\GL_2}}\CB(g_1ug_2)=0$. Combining the above proof with \cite[Lemma 5.1]{Ye18}, we can see that we only need the irreducible generic cuspidal $\GL_2(k)$ twists condition in Theorem \ref{thm: converse theorem}.}
\end{rmk}

\appendix

\section{Computation of certain Gauss sums}\label{sec: computation of gauss sum}
\subsection{Basic Gauss sum}   Let $\psi$ be a nontrivial additive character of $k=\BF_q$. Recall that we have fixed a square root $\sqrt{\epsilon_0}$ of $\epsilon_0$ such that 
$$\sum_{x\in k}\psi(ax^2)=\epsilon(a)\sqrt{\epsilon_0 q}.$$
For $a\in k^\times$, let 
$$A_r(a)=\sum_{x\in k^{\times,2}}\psi(arx), r=1,\kappa.$$
We then have 
$$1+2A_1(a)=\sum_{x\in k}\psi(ax^2)=\epsilon(a)\sqrt{\epsilon_0 q},$$
and 
$$1+2A_\kappa(a)=\sum_{x\in k}\psi(a\kappa x^2)=-\epsilon(a)\sqrt{\epsilon_0 q}.$$
Thus we get the following
\begin{lem}\label{lem: gauss sum}
We have $A_1(a)-A_\kappa(a)=\epsilon(a)\sqrt{\epsilon_0 q}$. 
\end{lem}
We write $A_r(1)$ as $A_r$ for simplicity, for $r=1,\kappa$.

\subsection{Computation of \texorpdfstring{$B_r^i$}{Lg}} We now compute the sums $B_r^i$ for $r=1,\kappa$ and $i=0,1,2,3$ in \eqref{eq: defn of B} used in $\S$\ref{sec: proof when p>3}.  We assume $q\equiv 1\mod 3$. Given $r\in \wpair{1,\kappa}, r_3\in k^\times, r_4\in k^\times/{\pm 1}$, let $z(r,r_3,r_4)=-2-\frac{rr_4^2}{r_3^3}\in k$. Note that for any $x\in k$, the equation $t+t^{-1}=a$ for $t$ is solvable over $k_2$.  Given $r,r_3,r_4$ as above, and let $t(r,r_3,r_4)$ be a solution of the equation $t+t^{-1}=z(r,r_3,r_4)$. Although there are two choices of $t(r,r_3,r_4)$ in general, one can check that the condition $t(r,r_3,r_4)\in \wpair{\pm 1}$ (resp. $t(r,r_3,r_4)\in k^{\times,3}-\wpair{\pm 1}$, $t(r,r_3,r_4)\in k^{\times}-k^{\times,3}$, $t(r,r_3,r_4)\in k_2- k^{\times}$) is independent on the choice of $t(r,r_3,r_4)$. Recall that
\begin{align*}
B_r^0&=\sum_{r_3\in k^\times,r_4\in k^{\times}/\wpair{\pm 1}, t(r,r_3,r_4)\in \wpair{\pm 1}}\psi(r_3),\\
B_r^1&=\sum_{r_3\in k^\times,r_4\in k^{\times}/\wpair{\pm 1}, t(r,r_3,r_4)\in k^{\times,3}-\wpair{\pm 1}}\psi(r_3),\\
B_r^2&=\sum_{r_3\in k^\times,r_4\in k^{\times}/\wpair{\pm 1}, t(r,r_3,r_4)\in k^\times- k^{\times,3}}\psi(r_3),\\
B_r^3&=\sum_{r_3\in k^\times,r_4\in k^{\times}/\wpair{\pm 1}, t(r,r_3,r_4)\notin k^{\times}}\psi(r_3),
\end{align*}
for $r=1,\kappa$.
\begin{lem}\label{lem: computation of B}
We have
\begin{align*}
   B_1^0-B_\kappa^0&=\epsilon_0\sqrt{\epsilon_0 q},\\
    B_1^1-B_\kappa^1&=-\frac{1}{2}(1+\epsilon_0)\sqrt{\epsilon_0 q},\\
    B_1^2-B_\kappa^2&=0,\\
    B_1^3-B_\kappa^3&=\frac{1}{2}(1-\epsilon_0)\sqrt{\epsilon_0 q}.
\end{align*}
\end{lem}
\begin{proof} For any $t\in k_2^\times$, $t\ne -1$, the condition $-2-rr_4^2/r_3^3=t+t^{-1}$ implies that 
\begin{equation}\label{eq: rt}
    (-r_3)^3=rt\left(\frac{r_4}{t+1} \right)^2.
\end{equation}
We first compute $B_r^0$. The condition $t(r,r_3,r_4)\in \wpair{\pm 1}$ implies that $t=1$ since $rr_4\ne 0$. Thus \eqref{eq: rt} becomes $(-r_3)^3=(r_4/2)^2$. Since $k^{\times}$ is a cyclic group generated by $\kappa$, the condition $(-r_3)^3=(r_4/2)^2$ implies that $-r_3\in k^{\times,2}$. Moreover, for each $-r_3\in k^{\times,2}$, there exists a unique $r_4\in k^{\times}/\wpair{\pm 1}$ such that the equation $(-r_3)^3=(r_4/2)^2$ holds. Thus we get 
$$B_1^0=\sum_{-r_3\in k^{\times,2}}\psi(r_3)=A_1(-1).$$
Similarly, we have $B_\kappa^0=A_\kappa(-1)$. Thus we have $B_1^0-B_\kappa^1=A_1(-1)-A_\kappa(-1)=\epsilon_0\sqrt{\epsilon_0 q}$ by Lemma \ref{lem: gauss sum}.

We next compute $B_{r}^1$, $r=1,\kappa$. Let $t=t(r,r_3,r_4)\in k^{\times,3}-\wpair{\pm 1}$. Let $a\in k^\times$ with $t=a^3$. We first assume that $r=1$. From \eqref{eq: rt}, we have $-a^{-1}r_3\in k^{\times,2}$. Thus the contribution of each fixed $t=t(1,r_3,r_4)$ to the sum $B_1^1$ is 
$$\sum_{x\in k^{\times,2}}\psi(-t^{1/3} x),$$
where $t^{1/3}$ is any cubic root of $t$ in $k^{\times}$. Because $t$ and $t^{-1}$ contributes the same to the sum $B_1^1$, we have $$B_1^1=\frac{1}{2}\sum_{t\in k^{\times,3}-\wpair{\pm 1}}\sum_{x\in k^{\times,2}}\psi(-t^{1/3}x).$$
Similarly, we have 
$$B_\kappa^1=\frac{1}{2}\sum_{t\in k^{\times,3}-\wpair{\pm 1}}\sum_{x\in k^{\times,2}}\psi(-t^{1/3}\kappa x).$$
Thus by Lemma \ref{lem: gauss sum}, we have 
\begin{align*}
    B_1^1-B_\kappa^1&=\frac{1}{2}\sum_{t\in k^{\times,3}-\wpair{\pm 1}}(A_1(-t^{1/3})-A_\kappa(-t^{1/3}))\\
    &=\frac{1}{2}\epsilon_0\sqrt{\epsilon_0 q}\sum_{t\in k^{\times,3}-\wpair{\pm 1}}\epsilon(t^{1/3}).
\end{align*}
We have $k^{\times,3}=\wpair{\kappa^{3i}: 1\le i\le \frac{q-1}{3}}$. Thus we get 
$$\sum_{t\in k^{\times,3}}\epsilon(t^{1/3})=\sum_{i=1}^{\frac{q-1}{3}}\epsilon(\kappa)^i=0,$$
where the last equality follows from the fact that $\epsilon(\kappa)=-1$ and $\frac{q-1}{3}$ must be even. Thus we get 
$$B_1^1-B_\kappa^1=-\frac{1}{2}\epsilon_0 \sqrt{\epsilon_0 \kappa}(1+\epsilon_0)=-\frac{1}{2}(1+\epsilon_0)\sqrt{\epsilon_0 q}.$$

 We next consider $B_r^2$. Note that $ k^\times-k^{\times,3}=\kappa k^{\times,3}\coprod \kappa^2 k^{\times,3}$. For $j=1,2$, we define 
$$B_r^{2,j}=\sum_{r_3\in k^\times,r_4\in k^{\times}/\wpair{\pm 1}, t(r,r_3,r_4)\in \kappa^j k^{\times,3}}\psi(r_3).$$
We have $B_{r}^2=B_r^{2,1}+B_{r}^{2,2}$. Take an element $t\in k^\times-k^{\times,3}$ with $t(r,r_3,r_4)=t$. Then the condition $-2-\frac{rr_4^2}{r_3^3}=t+t^{-1}$ implies \eqref{eq: rt}.
Note that if $r=1$ and $t\in \kappa k^{\times,3}$, equation \eqref{eq: rt} implies that $-r_3\in \kappa k^{\times,2}$, and for such an $r_3$, there is a unique $r_4$ satisfying that equation. Thus we get 
$$B_{1}^{2,1}=\sum_{t\in \kappa k^{\times,3}}\sum_{x\in k^{\times,2}}\psi(-\kappa x)=\frac{q-1}{3}A_\kappa'.$$
For $t\in \kappa^2 k^{\times,2}$ and $r=\kappa$, we also have that $-r_3\in \kappa k^{\times,2}$ and a unque $r_4$ dertermined by these datum. This shows that 
$$B_1^{2,1}=\sum_{t\in \kappa k^{\times,3}}\sum_{x\in k^{\times,2}}\psi(-\kappa x)=\frac{q-1}{3}A_\kappa'=B_{\kappa}^{2,2}.$$ Similarly, we have $B_\kappa^{2,1}=B_1^{2,2}$. Thus we get $B_1^2=B_{\kappa}^2$. 

Finally, we consider $B_r^3$. We have 
$$B_{r}^0+B_r^1+B_r^2+B_r^3=\sum_{r_3\in k^\times, r_4\in k^{\times}/\wpair{\pm 1}}\psi(r_3)=-\frac{q-1}{2}.$$
Thus, from the previous results, we get
$$B_1^3-B_\kappa^3=-(B_1^0-B_\kappa^0)-(B_1^1-B_\kappa^1).$$
This concludes the proof of the lemma.
\end{proof}

\subsection{Computation of \texorpdfstring{$C_r^i$}{}} In this subsection, we compute the sums $C_r^i$ for $r=1,\kappa,$ and $i=0,1,2,3$ defined in \eqref{eq: defn of C}. Note that in this case, $q\equiv -1\mod 3$. Recall that $k_2$ is the unique quadratic extension of $k=\BF_q$. We can realize $k_2$ as $k[\sqrt \kappa]$. Let $\Nm:k_2\ra k$ be the norm map. We have $\Nm(x+y\sqrt\kappa)=x^2-y^2\kappa$. Recall that $k_2^1$ is the norm 1 subgroup of $k_2^\times$.

\begin{lem}\label{lem: preparation for the computation of C}
\begin{enumerate}
    \item If an element $u\in k_2^1$ has a cubic root $v\in k_2^\times$, then we must have $v\in k_2^1$.
    \item Let $t\in k_2^1$. Then $t+t^{-1}+2$ is a square in $k^\times$ if and only if $t$ is a square in $k_2^1$.
\end{enumerate}
\end{lem}
\begin{proof}
(1) Since $u=v^3\in k_2^{1}$, we have $v^{3q+3}=1$. On the other hand, we have $v^{q^2-1}=1$ since $v\in k_2^\times$. Since $q\equiv -1 \mod 3$, the greatest common divisor of $q^2-1$ and $3q+3$ is $q+1$. Thus $\beta^{q+1}=1$, which means that $\beta\in k_2^1$.

(2) Suppose that $t=\beta^2$ with $\beta\in k_2^1$. We write $\beta=a+b\sqrt{\kappa}$ with $a,b\in k$. Then $\beta\in k_2^1$ means that $a^2-b^2\kappa=1$, which implies that $b^2\kappa=a^2-1$. We have $t=\beta^2=a^2+b^2\kappa+2ab\sqrt \kappa$. Thus
\begin{align*}
    t+t^{-1}+2=2(a^2+b^2\kappa)+2=4a^2\in k^{\times,2}.
\end{align*}
Conversely, suppose that $t+t^{-1}+2\in k^{\times,2}$. Suppose that $t=x+y\sqrt \kappa$ with $x,y\in k$ and $t+t^{-1}+2=a^2$ with $a\in k^{\times,2}$. Note that $t+t^{-1}+2=2x+2$. Thus $a^2=2x+2$. On the other hand, we have 
$$a^2=t+t^{-1}+2=t^{-1}(t+1)^2.$$
Thus, we have $t=(a^{-1}(t+1))^2$. It suffices to show that $a^{-1}(t+1)\in k_2^1$. We have 
$$\Nm(t+1)=(x+1)^2-y^2\kappa=2+2x=a^2,$$
where we used $x^2-y^2\kappa=1$. Thus $\Nm(a^{-1}(t+1))=1$.
\end{proof}

\begin{lem}\label{lem: computation of C} We have
\begin{align*}
C_1^0-C_\kappa^0&=\epsilon_0\sqrt{\epsilon_0 q},\\
C_1^1-C_\kappa^1&=-\frac{1}{2}(1+\epsilon_0)\sqrt{\epsilon_0 q},\\
C_1^2-C_\kappa^2&=\frac{1}{2}(1-\epsilon_0)\sqrt{\epsilon_0 q},\\
C_1^3-C_\kappa^3&=0.
\end{align*}
\end{lem}

\begin{proof}
Note that $C_r^0=B_r^0$ and thus $C_1^0-C_2^0=\epsilon_0\sqrt{\epsilon_0 q}$ follows from Lemma \ref{lem: computation of B}. To compute $C_r^2$, we take an element $t\in k^\times-\wpair{\pm 1}$ and let $t(r,r_3,r_4)=t$, which implies 
$$(-r_3)^3=rt\left(\frac{r_4}{t+1}\right)^2, $$
see \eqref{eq: rt}. Note that any $t\in k^{\times}$ is has a cubic root in $k^{\times}$. Let $t^{1/3}\in k^\times$ be one cubic root of $t$. Then the above equation implies that 
$$(-r_3/t^{1/3})^3=r\left(\frac{r_4}{t+1}\right)^2.$$
If $r=1$, this implies that $r_3\in -t^{1/3}k^{\times,2}$, and for such an $r_3$ (and a fixed $t$), there is a unique $r_4\in k^\times/\wpair{\pm 1}$ such that $(-r_3/t^{1/3})^3=r\left(\frac{r_4}{t+1}\right)^2 $. Thus the contribution of a single $t$ with $t(1,r_3,r_4)$ to the sum $C_1^1$ is $$\sum_{k^{\times,2}}\psi(-t^{1/3}x).$$
Since $t$ and $t^{-1}$ have the same contribution, we have
$$C_1^1=\frac{1}{2}\sum_{t\in k^{\times}-\wpair{\pm 1}}\sum_{x\in k^{\times,2}}\psi(-t^{1/3}x).$$
Since $t\mapsto t^{3}$ is a bijection from $k^\times-\wpair{\pm 1}$ to itself, we get 
$$C_1^1=\frac{1}{2}\sum_{t\in k^{\times}-\wpair{\pm 1}}\sum_{x\in k^{\times,2}}\psi(-tx)=\frac{1}{2}\sum_{t\in k^\times-\wpair{\pm 1}}A_1(-t).$$
Similarly, we have 
$$C_\kappa^1=\frac{1}{2}\sum_{t\in k^\times-\wpair{\pm 1}}A_\kappa(-t).$$
Thus by Lemma \ref{lem: gauss sum}, we have 
\begin{align*}
    C_1^1-C_\kappa^1&=\frac{1}{2}\sum_{t\in k^\times-\wpair{\pm 1}}(A_1(-t)-A_\kappa(-t))\\
    &=\frac{1}{2}\sum_{t\in k^\times-\wpair{\pm 1}} \epsilon_0 \epsilon(t)\sqrt{\epsilon_0 q}.
\end{align*}
Since $\epsilon$ is a nontrivial character on $k^\times$, we have $\sum_{t\in k^\times}\epsilon(t)=0$. Thus we have 
$$C_1^1-C_\kappa^1=-\frac{1}{2}(1+\epsilon_0)\sqrt{\epsilon_0 q}.$$

We next consider $C_r^3$. Let $\alpha$ be a generator of $k_2^1$. Note that $\alpha$ has no cubic root in $k_2^1$. By Lemma \ref{lem: preparation for the computation of C} (1), we have 
$$k_2^1-k_2^{\times,3}=\wpair{\alpha^i:0\le i\le q, 3\nmid i}.$$
Consider the subsets $S_1,S_2$ of $k_2^1-k_2^{\times,3}:$
$$S_1=\wpair{\alpha^i:0\le i\le q, 3\nmid i,2\nmid i},S_2=\wpair{\alpha^i:0\le i\le q, 3\nmid i, 2|i}.$$
Note that $|S_1|=|S_2|=\frac{q+1}{3}$. For $i=1,2$, let 
$$C_r^{3,i}=\sum_{r_3\in k^{\times},r_4\in k^{\times}/\wpair{\pm 1},t(r,r_3,r_4)\in S_i}\psi(r_3).$$
We have $C_r^{3}=C_{r}^{3,1}+C_{r}^{3,2}$. Take $t\in S_i$, the condition $t(r,r_3,r_4)=t$ implies that 
$$(-r_3)^3=\frac{rr_4^2}{t+t^{-1}+2}.$$
If $t\in S_1$, by Lemma \ref{lem: preparation for the computation of C}, we have $t+t^{-1}+2\in \kappa k^{\times,2}$. Thus for $r=1,t\in S_1$, we have $-r_3\in \kappa k^{\times,2}$, and for each $-r_3\in \kappa k^{\times,2}$, there is a unique $r_4\in k^{\times}/\wpair{\pm 1}$ such that $t(1,r_3,r_4)=t$ (for fixed $t$). Thus, we get 
$$C_1^{3,1}=\frac{1}{2}\sum_{t\in S_1}\sum_{x\in k^{\times,2}}\psi(-\kappa x)=\frac{q+1}{6}A_{\kappa}(-1),$$
where the $1/2$ was appeared since $t$ and $t^{-1}$ have the same contribution to the above sum. Similarly, we have 
$$C_\kappa^{3,2}=\frac{1}{2}\sum_{t\in S_2}\sum_{x\in k^{\times,2}}\psi(-\kappa x)=\frac{q+1}{6}A_{\kappa}(-1).$$
In particular, we have $C_{1}^{3,1}=C_{\kappa}^{3,2}$. Similarly, we have $C_1^{3,2}=C_\kappa^{3,1}$. Thus we have $C_1^3-C_\kappa^3=0$. 

Finally, to compute $C_1^2-C_{\kappa}^2$, it suffices to notice that 
$$\sum_{i=0}^3 C_1^i=\sum_{i=0}^3C_\kappa^i,$$
and thus $$C_1^2-C_\kappa^2=-(C_1^0-C_\kappa^0)-(C_1^1-C_\kappa^1)-(C_1^3-C_\kappa^3).$$
One can also compute $C_1^2-C_\kappa^2$ directly from Lemma \ref{lem: preparation for the computation of C}.
\end{proof}
\subsection{Computation of \texorpdfstring{$D_r^i$}{Lg}}\label{sec: computation of D}
In this subsection, let $q=3^f$ and $k=\BF_q$. We compute the Gauss sums in $\S$\ref{sec: proof when p=3}. Recall that (see \eqref{eq: defn of D}), 
$$D_r^0=\sum_{r_3\in k^\times, r_4\in k^{\times}/\wpair{\pm 1}, t(r,r_3,r_4)\in \wpair{\pm 1}} \psi(r_3)$$
and $$D_r^1=\sum_{r_3\in k^\times, r_4\in k^{\times}/\wpair{\pm 1}, t(r,r_3,r_4)\notin \wpair{\pm 1}} \psi(r_3).$$
\begin{lem}\label{lem: computation of D}
We have 
\begin{align*}
    D_1^0-D_\kappa^0&= \epsilon_0 \sqrt{\epsilon_0 q},\\
    D_1^1-D_\kappa^1&=-\frac{1}{2}(1+\epsilon_0)\sqrt{\epsilon_0 q},\\
    D_1^2-D_\kappa^2&=\frac{1}{2}(1-\epsilon_0)\sqrt{\epsilon_0 q}.
\end{align*}
\end{lem}
\begin{proof}
Note that we have $D_r^0=B_r^0$. Thus the first identity follows from Lemma \ref{lem: computation of B}. The second identity can be computed similarly as the computation of $C_1^1-C_\kappa^1$. Since $D_1^0+D_1^1+D_1^2=D_\kappa^0+D_\kappa^1+D_\kappa^2$, the last identity follows from the first one.
\end{proof}

\section{Embedding of \texorpdfstring{$\RG_2$}{Lg} into \texorpdfstring{$\SO_7$}{Lg}}\label{sec: embedding G2 into SO7} In this appendix, based on \cite{RS}, we give an explicit matrix realization of $\bx_\gamma(r)$ for each root $\gamma$ of $\RG_2$, which gives an explicit embedding of $\RG_2(k)$ into $\SO_7(k)$. Here $\SO_7(k)=\wpair{g\in \GL_7(k): {}^t g Qg=Q}$, with $$Q=\begin{pmatrix} &&s_3\\ &2& \\ {}^t \!s_3 \end{pmatrix},$$
where $$s_3=\begin{pmatrix} &&1\\ &1&\\ -1&&\end{pmatrix}.$$ The explicit realization of $\bx_\gamma(r)$ is given as follows. 
\begin{align*}
    \bx_\alpha(r)=\bpm 1&0&0&0&0&0&0\\ 0&1&0&-2r&0&-r^2&0\\ 0&0&1&0&0&0&0\\ 0&0&0&1&0&r&0\\ -r&0&0&0&1&0&0\\0 &0&0&0&0&1&0\\ 0&0&-r&0&0&0&1\epm, &\quad \bx_{-\alpha}(r)=\bpm 1&0&0&0&-r&0&0\\ 0&1&0&0&0&0&0 \\ 0&0&1&0&0&0&-r \\0 &-r&0&1&0&0&0 \\ 0&0&0&0&1&0&0 \\ 0&-r^2&0&2r&0&1&0 \\ 0&0&0&0&0&0&1\epm,
\end{align*}
\begin{align*}
    \bx_{\alpha+\beta}(r)=\bpm 1&0&0&-2r&0&0&-r^2\\ 0&1&0&0&0&0&0 \\ 0&0&1&0&0&0&0 \\0&0&0&1&0&0&r \\0&r&0&0&1&0&0 \\ 0&0&r&0&0&1&0 \\0&0&0&0&0&0&1 \epm,& \quad 
    \bx_{-(\alpha+\beta)}(r)=\bpm 1&0&0&0&0&0&0\\ 0&1&0&0&r&0&0 \\ 0&0&1&0&0&r&0 \\-r&0&0&1&0&0&0 \\0&0&0&0&1&0&0 \\ 0&0&0&0&0&1&0 \\-r^2&0&0&2r&0&0&1 \epm,
\end{align*}
\begin{align*}
    \bx_{2\alpha+\beta}(r)&=\bpm 1&0&0&0&0&-r&0\\ 0&1&0&0&0&0&r \\ 0&0&1&0&0&0&0 \\0&0&-r&1&0&0&0 \\0&0&r^2&-2r&1&0&0 \\ 0&0&0&0&0&1&0 \\0&0&0&0&0&0&1 \epm, & \quad \bx_{-(2\alpha+\beta)}(r)&=\bpm 1&0&0&0&0&0&0\\ 0&1&0&0&0&0&0 \\ 0&0&1&-2r&r^2&0&0 \\0&0&0&1&-r&0&0 \\0&0&0&0&1&0&0 \\ -r&0&0&0&0&1&0 \\0&r&0&0&0&0&1 \epm,
\end{align*}
\begin{align*}
    \bx_\beta(r)=\bpm 1&r&0&0&0&0&0\\ 0&1&0&0&0&0&0 \\ 0&0&1&0&0&0&0 \\0&0&0&1&0&0&0 \\0&0&0&0&1&0&0 \\ 0&0&0&0&0&1&-r \\0&0&0&0&0&0&1 \epm, & \quad \bx_{3\alpha+\beta}(r)=\bpm 1&0&0&0&0&0&0\\ 0&1&r&0&0&0&0 \\ 0&0&1&0&0&0&0 \\0&0&0&1&0&0&0 \\0&0&0&0&1&r&0 \\ 0&0&0&0&0&1&0 \\0&0&0&0&0&0&1 \epm,\\
   \bx_{3\alpha+2\beta}(r)=\bpm 1&0&r&0&0&0&0\\ 0&1&0&0&0&0&0 \\ 0&0&1&0&0&0&0 \\0&0&0&1&0&0&0 \\0&0&0&0&1&0&r \\ 0&0&0&0&0&1&0 \\0&0&0&0&0&0&1 \epm, &
\end{align*}
and
$$\bx_{-\beta}(r)={}^t \bx_\beta(r), \bx_{-(3\alpha+\beta)}(r)={}^t \bx_{3\alpha+\beta}(r), \bx_{-(3\alpha+2\beta)}(r)={}^t\bx_{3\alpha+\beta}(r).$$

\end{document}